\documentclass[jsg]{ipart_v1}

\Vol{13}
\Issue{3}
\Year{2015}
\firstpage{545}

\usepackage{mathrsfs}

\newtheorem{theorem}{Theorem}[section]
\newtheorem{lemma}[theorem]{Lemma}
\newtheorem{proposition}[theorem]{Proposition}
\newtheorem{corollary}[theorem]{Corollary}

\newtheorem*{theorem*}{Theorem}
\newtheorem*{proposition*}{Proposition}
\newtheorem*{corollary*}{Corollary}

\theoremstyle{definition}
\newtheorem{definition}[theorem]{Definition}
\newtheorem{example}[theorem]{Example}

\newtheorem{remark}[theorem]{Remark}

\numberwithin{equation}{section}

\input xy
\xyoption{all}

\usepackage{tikz}
\usetikzlibrary{positioning}
\usepackage{graphicx}
\usepackage{subfig}

\def\co{\colon\thinspace}

\begin{document}

\title[Localization for doubly-periodic knots]{Localization and the link Floer homology of doubly-periodic knots}

\author[Kristen Hendricks]{Kristen Hendricks}


\begin{abstract}

A knot $\widetilde{K} \subset S^3$ is $q$-periodic if there is a $\mathbb Z_q$-action preserving $\widetilde{K}$ whose fixed set is an unknot $\widetilde{U}$.  The quotient of $\widetilde{K}$ under the action is a second knot $K$.  We construct equivariant Heegaard diagrams for $q$-periodic knots, and show that Murasugi's classical condition on the Alexander polynomials of periodic knots is a quick consequence of these diagrams. For $\widetilde{K}$ a two-periodic knot, we show there is a spectral sequence whose $E^1$ page is $\widehat{\mathit{HFL}}(S^3,\widetilde{K}\cup \widetilde{U})\otimes V^{\otimes (2n-1)})\otimes \mathbb Z_2((\theta))$ and whose $E^{\infty}$ page is isomorphic to $(\widehat{\mathit{HFL}}(S^3,K\cup U)\otimes V^{\otimes (n-1)})\otimes \mathbb Z_2((\theta))$, as $\mathbb Z_2((\theta))$-modules, and a related spectral sequence whose $E^1$ page is $(\widehat{\mathit{HFK}}(S^3,\widetilde{K})\otimes\linebreak V^{\otimes (2n-1)}\otimes W)\otimes \mathbb Z_2((\theta))$ and whose $E^{\infty}$ page is isomorphic to $(\widehat{\mathit{HFK}}(S^3,K)\otimes V^{\otimes (n-1)} \otimes W)\otimes \mathbb Z_2((\theta))$. As a consequence, we\linebreak use these spectral sequences to recover a classical lower bound of Edmonds on the genus of $\widetilde{K}$, along with a weak version of a classical fibredness result of Edmonds and Livingston. We give an example of a knot $\widetilde{K}$ which is not obstructed from being two-periodic with a particular quotient knot $K$ by Edmonds' and Murasugi's conditions, but for which our spectral sequence cannot exist.

\end{abstract}

\maketitle

\tableofcontents

\section{Introduction}

We say $\widetilde{K} \subset S^3$ is a \textit{periodic knot} if there is a $\mathbb Z_q$-action on $(S^3, \widetilde{K})$ which preserves $\widetilde{K}$ and whose fixed set is an unknot $\widetilde{U}$ disjoint from $\widetilde{K}$.  Let $\tau$ be a generator for the action.  An important special case is $\tau^2=1$, in which case $\widetilde{K}$ is said to be \textit{doubly-periodic}.

The quotient of $(S^3, \widetilde{K})$ under the group action is a second knot $(S^3, K)$ such that the map $(S^3, \widetilde{K}) \rightarrow (S^3,K)$ is an $q$-fold branched cover over $U$, the unknot which is the image of $\widetilde{U}$ in the quotient.  The knot $K$ is said to be the $q$-fold \textit{quotient knot} of $\widetilde{K}$.

\begin{figure}
\centering
\includegraphics{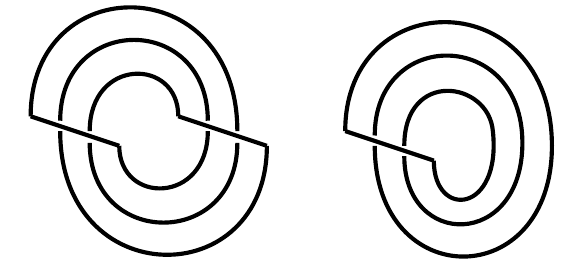}
\caption{A doubly-periodic diagram for the trefoil, and its quotient knot (an unknot) under the $\mathbb Z_2$-action.}
\end{figure}

Link Floer homology is an invariant of a link in $S^3$ introduced by Ozsv{\'a}th and Szab{\'o} \cite{MR2443092} as a generalization of the knot Floer homology developed in 2003 by Ozsv{\'a}th and Szab{\'o} \cite{MR2065507} and independently by Rasmussen \cite{MR2704683}.  The theory associates to $(S^3,L)$ a bigraded $\mathbb Z_2$-vector space $\widehat{\mathit{HFL}}(S^3,L)$ which arises as the homology of the Floer chain complex of two Lagrangian tori in the symmetric product of a punctured Heegaard surface for $(S^3,L)$.  The graded Euler characteristic of link Floer cohomology is the multivariable Alexander polynomial of the knot multiplied by certain standard factors \cite{MR2443092}; indeed, the theory categorifies the Thurston norm of the knot complement \cite{MR2393424}.

The purpose of this paper is to use localization theorems of Seidel and Smith to analyze the link Floer homology of a periodic (usually doubly-periodic) knot together with its axis compared to the link Floer homology of the quotient knot together with the axis.  We will first construct Heegaard diagrams for $(S^3, \widetilde{K} \cup \widetilde{U})$ which are preserved by the action of $\mathbb Z_q$ and  whose quotients under the action are Heegaard diagrams for $(S^3,K)$. These periodic Heegaard diagrams will allow us to give a simple Heegaard Floer reproof of one of Murasugi's conditions for the Alexander polynomial of a periodic knot in the case that $q=p^r$ for some prime $p$. Let $\lambda = \operatorname{lk}(\widetilde{K}, \widetilde{U}) = \operatorname{lk}(K,U)$. Since $\widetilde{K}$ is a knot (not a link), it follows that $\lambda$ is coprime to the periodicity~$q$.

\begin{theorem} \label{Murasugi Theorem}\hspace{-.5ex}\cite[Corollary 1]{MR0292060} $\Delta_{\widetilde{K}}(t) \!\equiv\!  t^{\pm i}(1 \!+\! t \!+\! \cdots \!+\! t^{\lambda-1})^{q-1} (\Delta_{K}(t))^{q}$ modulo $p$.
\end{theorem}

Restricting to the case that $\widetilde{K}$ is doubly periodic, we will proceed to prove the following localization theorem. Let $\mathbb Z_2((\theta)) = \mathbb Z_2[[\theta]](\theta^{-1})$. Let $V$ and $W$ be two-dimensional vector spaces over $\mathbb Z_2$, which will later be distinguished by their gradings. Let $A_1$ be the Alexander grading relative to the knot and $A_2$ be the Alexander grading relative to the unknotted axis (these gradings will be defined in Section \ref{Heegaard Floer Background Section}).

\begin{theorem} \label{Link Floer Homology Spectral Sequence}
There is an integer $n_1$ less than half the number of crossings of a periodic diagram $D$ for $\tilde{K}$ such that there is a spectral sequence whose $E^1$ page is $\big(\widehat{\mathit{HFL}}(S^3, \widetilde{K} \cup \widetilde{U}) \otimes V^{\otimes (2n_1-1)}\big) \otimes \mathbb Z_2((\theta))$ and whose $E^{\infty}$ page is isomorphic to $\big(\widehat{\mathit{HFL}}(S^3, K \cup U) \otimes V^{\otimes (n_1-1)}\big) \otimes  \mathbb Z_2((\theta))$ as $\mathbb Z_2((\theta))$-modules. This sequence splits along Alexander multigradings, and carries the grading $(A_1,A_2) = (2a+\frac{1}{2}, b)$ to $(A_1,A_2)=(a + \frac{1}{2}, b)$ for all $a,b \in \mathbb Z$. In any grading $(A_1,A_2)$ which cannot be written as $(2a+\frac{1}{2}, b)$, the last page of the spectral sequence is empty.
\end{theorem}

We may also reduce the spectral sequence of Theorem \ref{Link Floer Homology Spectral Sequence} to contain only the information of knot Floer homology of $\widetilde{K}$ and $K$.

\begin{theorem} \label{Knot Floer Homology Spectral Sequence}
There is an integer $n_1$ less than half the number of crossings on a periodic diagram $D$ for $\widetilde{K}$ such that there is a spectral sequence whose $E^1$ page is $\big(\widehat{\mathit{HFK}}(S^3, \widetilde{K}) \otimes V^{\otimes (2n_1-1)} \otimes W \big)\otimes \mathbb Z_2((\theta))$ and whose $E^{\infty}$ page is isomorphic to $\big(\widehat{\mathit{HFK}}(S^3, K) \otimes V^{\otimes (n_1-1)} \otimes W\big) \otimes \mathbb Z_2((\theta))$. This sequence splits along Alexander gradings, and carries the Alexander grading $A_1= 2a + \frac{1-\lambda}{2}$ to $A_1 = a + \frac{1-\lambda}{2}$ for all integers $s \in \mathbb Z$. In any grading $A_1$ which cannot be written as $2a + \frac{1-\lambda}{2}$, the last page of the spectral sequence is empty.
\end{theorem}

Analyzing the behavior of the Alexander gradings in these spectral se\-quences yields some geometric consequences. In particular, it is worth noting the following rank inequalities implied by Theorems \ref{Link Floer Homology Spectral Sequence} and \ref{Knot Floer Homology Spectral Sequence}. In Section 2, we will construct a Heegaard diagram $\mathcal D$ for $K \cup U$ with equivariant lift a Heegaard diagram $\widetilde{\mathcal D}$ for $\widetilde{K} \cup \widetilde{U}$. Then Theorem \ref{Link Floer Homology Spectral Sequence} can be rephrased as asserting the existence of a spectral sequence from the theory $\widetilde{\mathit{HFL}}(\widetilde{\mathcal D})\otimes \mathbb Z_2((\theta))$ to $\widetilde{\mathit{HFL}}(\mathcal D)\otimes \mathbb Z_2((\theta))$, and Theorem \ref{Knot Floer Homology Spectral Sequence} can be rephrased as a spectral sequence from $\widetilde{\mathit{HFK}}(\widetilde{\mathcal D})\otimes \mathbb Z_2((\theta))$ to $\widetilde{\mathit{HFK}}(\mathcal D)\otimes \mathbb Z_2((\theta))$.We have the following rank inequalities.

\begin{corollary} \label{Alexander Gradings Link Lemma} For all $a,b \in \mathbb Z$, there is a rank inequality
\begin{align*}
rk\left(\widetilde{\mathit{HFL}}\left(\widetilde{\mathcal D}, \left(\frac{1}{2} +2a, b\right)\right)\right) \geq rk\left(\widetilde{\mathit{HFL}}\left(\mathcal D, \left(\frac{1}{2}+a, b\right)\right)\right).
\end{align*}

\end{corollary}

\begin{corollary} \label{Alexander Gradings Knot Lemma}
For all $a \in \mathbb Z$, there is a rank inequality
\begin{align*}
rk\left(\widetilde{\mathit{HFK}}\left(\widetilde{\mathcal D}, 2a + \frac{1 - \lambda}{2}\right)\right) \geq rk\left(\widetilde{\mathit{HFK}}\left(\mathcal D, a + \frac{1 - \lambda}{2}\right)\right).
\end{align*}

In particular, if Edmonds' condition is sharp,there is a rank inequality
\begin{align*}
\widehat{\mathit{HFK}}(S^3, \widetilde{K}, g(\widetilde{K})) \geq \widehat{\mathit{HFK}}(S^3,K, g(K)).
\end{align*}
\end{corollary}


These relationships yield a reproof of a classical result, proved by Alan Edmonds using minimal surface theory.

\begin{corollary} \label{Genus Inequality Corollary} \cite[Theorem 4]{MR769284} Let $\widetilde{K}$ be a doubly-periodic knot in $S^3$ and $K$ be its quotient knot.  Then
\[
g(\widetilde{K}) \geq 2g(K) + \frac{\lambda-1}{2}.
\]

\end{corollary}

We will also observe from Corollary \ref{Alexander Gradings Knot Lemma} a proof of the following corollary.

\begin{corollary}  \label{Fiberedness Corollary}

Let $\widetilde{K}$ be a doubly-periodic knot in $S^3$ and $K$ its quotient knot.  If Edmonds' 
condition is sharp and $\widetilde{K}$ is fibered, $K$ is fibered.

\end{corollary}

This is a weaker version of the following theorem, proved by Edmonds and Livingston.

\begin{theorem}\cite[Prop. 6.1]{MR728451} Let $\widetilde{K}$ be a doubly-periodic knot in $S^3$ and $K$ its quotient knot.  If  $\widetilde{K}$ is fibered, $K$ is fibered.

\end{theorem}

In some cases, Corollary \ref{Alexander Gradings Knot Lemma} may provide an obstruction to a knot being two-periodic with a particular quotient knot, even if Corollary \ref{Murasugi Theorem} and Corollary \ref{Genus Inequality Corollary} are satisfied. In Section 3.3 we will use it to show the following.

\begin{example}\label{Main Example}
Let $\widetilde{K}$ be the connect sum of the right-handed trefoil and the Kinoshita-Terasaka knot and $K$ be the positive untwisted Whitehead double of the figure eight knot. The Alexander polynomials and genera of $\widetilde{K}$ and $K$ satisfy the conditions of Corollaries \ref{Murasugi Theorem} and \ref{Genus Inequality Corollary}, but the spectral sequence of Theorem \ref{Knot Floer Homology Spectral Sequence} cannot exist.
\end{example}

The spectral sequences of Theorems \ref{Link Floer Homology Spectral Sequence} and \ref{Knot Floer Homology Spectral Sequence} are analogs of the spectral sequence for the knot Floer homology of double branched covers of knots in $S^3$ constructed in \cite{Hendricks}; their construction requires only slightly more technical complexity.  As in that paper, the key technical tool in the proofs of Theorems \ref{Link Floer Homology Spectral Sequence} and \ref{Knot Floer Homology Spectral Sequence}  is a result of Seidel and Smith concerning equivariant Floer cohomology. Let $M$ be an exact symplectic manifold, convex at infinity, containing exact Lagrangians $L_0$ and $L_1$ and equipped with an involution $\tau$ preserving $(M, L_0, L_1)$. Let $(M^{\operatorname{inv}}, L_0^{\operatorname{inv}}, L_1^{\operatorname{inv}})$ be the submanifolds of each space fixed by $\tau$. Then under certain stringent conditions on the normal bundle $N(M^{\operatorname{inv}})$ of $M^{\operatorname{inv}}$ in $M$, there is a rank inequality between the Floer cohomology $\mathit{HF}(L_0,L_1)$ of the two Lagrangians $L_0$ and $L_1$ in $M$ and the Floer cohomology $\mathit{HF}(L_0^{\operatorname{inv}}, L_1^{\operatorname{inv}})$ of $L_0^{\operatorname{inv}}$ and $L_1^{\operatorname{inv}}$ in $M^{\operatorname{inv}}$. More precisely, they consider the normal bundle $N(M^{\operatorname{inv}})$ to $M^{\operatorname{inv}}$ in $M$ and its pullback $\Upsilon(M^{\operatorname{inv}})$ to $M^{\operatorname{inv}} \times [0,1]$. We ask that $M$ satisfy a $K$-theoretic condition called \textit{stable normal triviality} relative to two Lagrangian subbundles over $L_0^{\operatorname{inv}} \times \{0\}$ and $L_1^{\operatorname{inv}} \times \{1\}$. Seidel and Smith prove the following.

\begin{theorem} \label{SeidelSmith} \cite[Section 3f]{MR2739000}
If $\Upsilon(M^{\operatorname{inv}})$ carries a stable normal trivialization, there is a spectral sequence whose $E^1$ page is $\mathit{HF}(L_0,L_1) \otimes \mathbb Z_2((\theta))$ and whose $E^{\infty}$ page is isomorphic to $\mathit{HF}(L_0^{\operatorname{inv}}, L_1^{\operatorname{inv}}) \otimes \mathbb Z_2((\theta))$ as $\mathbb Z_2((\theta))$ modules. 
\end{theorem}

This paper is organized as follows: In Section 2 we recall the construction and important properties of link Floer homology.  In Section 3 we construct equivariant Heegaard diagrams for periodic knots, prove Theorem \ref{Murasugi Theorem} from these diagrams, and explain how Corollaries \ref{Genus Inequality Corollary} and \ref{Fiberedness Corollary} follow from Theorems \ref{Link Floer Homology Spectral Sequence} and \ref{Knot Floer Homology Spectral Sequence}. We also explain Example \ref{Main Example}. In Section 4 we briefly review Seidel and Smith's localization theory for Floer cohomology, and describe the symmetric products to which we will apply it.  In Section 5 we provide a description of the homotopy type and cohomology of the most general of these symmetric products (which contains the symmetric products used to compute knot and link Floer homology as submanifolds).  In Section 6 we give a proof that this general symmetric product carries a stable normal trivialization, which will imply that the symmetric products used in the computation of knot and link Floer homology do as well.  In Section 7 we compute the spectral sequences of Theorems \ref{Link Floer Homology Spectral Sequence} and \ref{Knot Floer Homology Spectral Sequence} for the unknot and the trefoil as doubly-periodic knots as examples.

\subsection{Acknowledgements}

I am grateful to Robert Lipshitz for suggesting this problem, providing guidance, and reading a draft of this paper.  Many thanks also to Allison Gilmore,  Matthew Hedden, Jennifer Hom, Tye Lidman, Ciprian Manolescu, and Dylan Thurston for helpful conversations, and to Adam Levine for pointing out the argument of Lemma \ref{Generator Matching Lemma}.  I am also indebted to Chuck Livingston and Paul Kirk for their enthusiasm and commentary, particularly concerning the grading arguments in the proofs of Theorems \ref{Alexander Gradings Link Lemma} and \ref{Alexander Gradings Knot Lemma}. Furthermore, thanks to reviewer for clear and helpful comments, and in particular for suggesting considering a specific quotient knot in constructing Example \ref{Main Example}.

I was partially supported by an NSF grant number DMS-0739392. Most of the content of this paper also appeared in my Ph.D. thesis \cite{HendricksThesis}.

\section{Heegaard Floer homology theories} \label{Heegaard Floer Background Section}

We pause for some discussion of link Floer homology in the three sphere, first defined by Ozsv{\'a}th and Szab{\'o} in \cite{MR2443092}. Our purpose is primarily to give notation and to emphasize the points which are important to the proofs in this paper; for a complete treatment, we refer the reader to \cite{MR2443092}. All work is done over $\mathbb Z_2$.

\begin{definition}
A \textit{multipointed Heegaard diagram} $\mathcal D = (S, \boldsymbol \alpha, \boldsymbol \beta, {\bf w}, {\bf z})$ consists of the following data.

\begin{itemize}
\item An oriented surface $S$ of genus $g$. 
\item Two sets of basepoints ${\bf w} = (w_1,\ldots,w_{n})$ and ${\bf z} = (z_1,\ldots,z_{n})$. 
\item Two sets of closed embedded curves $\boldsymbol \alpha = \{\alpha_1,\ldots,\alpha_{g+n-1}\}$ and $\boldsymbol \beta = \{\beta_1,\ldots,\beta_{g+n-1}\}$ such that each of $\boldsymbol \alpha$ and $\boldsymbol \beta$ spans a $g$-dimensional subspace of $H_1(S)$, $\alpha_i \cap \alpha_j = \emptyset = \beta_i \cap \beta_j$ for $i\neq j$, each $\alpha_i$ and $\beta_j$ intersect transversely, and each component of $S - \bigcup \alpha_i$ and of $S - \bigcup \beta_i$ contain exactly one point of ${\bf w}$ and one point of ${\bf z}$. 

\end{itemize}

\end{definition}

From any Heegaard diagram $\mathcal D$, we may construct a unique $3$-manifold $Y$ containing a link $L$; in this paper we are only interested in the case that $Y=S^3$. In particular, to recover the link, we connect the $w$ basepoints to the $z$ basepoints in the complement of the curves $\beta_i$, then connect the $z$ basepoints to the $w$ basepoints in the complement of the curves $\alpha_i$, letting the first set of arcs overcross the second. We number our basepoints such that if $L = K_1 \cup \cdots \cup K_{\ell}$ is the link produced, there are $n_j$ pairs of basepoints on $K_j$, and there are integers $0 =k_0<k_1<\cdots<k_{\ell} =n$ with $k_j - k_{j-1} = n_j$ such that $w_{k_{j-1}+1},\ldots,w_{k_j}, z_{k_{j-1}+1},\ldots,z_{k_{j}}$ are the basepoints on $K_j$.  (In the examples we are actually interested in, we will have $L = K_1 \cup K_2$ with $n_1$ pairs of basepoints on $K_1$ and a single pair of basepoints on $K_2$, so the notation will not be too bad.)

We impose one further technical condition on our Heegaard diagram. Call the components of $S - \bigcup \alpha_i - \bigcup \beta_i$ the \textit{elementary regions} of the Heegaard diagram.

\begin{definition}
A \textit{periodic domain} $P$ on $S \backslash {\bf w}$ is a linear combination of elementary regions whose boundary may be expressed as a linear combination of the $\alpha$ and $\beta$ curves. 
\end{definition}

We say that $D$ is \textit{weakly admissible} if every periodic domain on $S$ has both positive and negative local multiplicities, and require that any Heegaard diagram we use to compute link Floer homology have this property.

The link Floer homology $\widehat{\mathit{HFL}}(Y,L)$ is the Lagrangian Floer cohomology of the two Lagrangian tori $\mathbb T_{\boldsymbol \alpha} = \alpha_1 \times\cdots \times \alpha_{g+n -1}$ and $\mathbb T_{\boldsymbol \beta} = \beta_1 \times\cdots\times \beta_{g+n-1}$ inside the symmetric product $\operatorname{Sym}^{g+n-1}(S \backslash ({\bf w} \cup {\bf z}))$. Therefore the chain complex $\widehat{\mathit{CFL}}(\mathcal D)$ for knot Floer homology is generated by the finite set of intersection points of $\mathbb T_{\boldsymbol \alpha}$ and $\mathbb T_{\boldsymbol \beta}$. (More concretely, a generator of $\widehat{\mathit{CFL}}(\mathcal D)$ is a point ${\bf x} = (x_1 \cdots x_{g+n -1}) \in \operatorname{Sym}^{g+n-1}(S)$ such that each $\alpha$ or $\beta$ curve contains a single $x_i$.) The differential $\partial$ counts holomorphic representatives of homotopy classes of Whitney disks $[\phi] \in \pi_2({\bf x}, {\bf y})$ whose image lies in the symmetric product of the punctured sphere. Here a Whitney disk is a map $\phi: B_1(0) \rightarrow \operatorname{Sym}^{g+n-1}(S)$ which maps the left half of the boundary of the unit disk $B_1(0) \cup \mathbb C$ to $\mathbb T_{\boldsymbol \beta}$ and the right half to $\mathbb T_{\boldsymbol \alpha}$, and has $\phi(-i) = {\bf x}$, $\phi(-i) = {\bf y}$.

\subsection{The Maslov index and grading} Recall that if $\phi: B_1(0) \rightarrow \operatorname{Sym}^{g+n-1}(S)$ is a Whitney disk, we can associate to $\phi$ its shadow  $D = \Sigma a_i D_i$ on the Heegaard diagram, where $D_i$ are the closures of the elementary regions of $\mathcal D$ and $a_i$ is the algebraic multiplicity of the intersection of the holomorphic submanifold $V_{x_i} = \{x_i\} \times \operatorname{Sym}^{g+n-2}(S)$ with $\phi(B_1(0))$ for any interior point $x_i$ of $D_i$. The boundary of $D$ consists of $\alpha$ arcs from points of ${\bf x}$ to points of ${\bf y}$ and $\beta$ arcs from points of ${\bf y}$ to points of ${\bf x}$. If $D_i$ contains a basepoint $z_j$, we let $a_i = n_{z_i}(\phi)$ be the algebraic intersection number of $z_i \times \operatorname{Sym}^{g+n-2}(S)$ with the image of $\phi$.

Given $\phi \in \pi_2({\bf x}, {\bf x})$, we define the \textit{Maslov index} as follows. Recall that $\phi \co D^2 \rightarrow \operatorname{Sym}^{g+n-1}(S)$ maps the portion of the boundary of the unit disk $D^2$ in the right half of the complex plane $\mathbb C = \{u+iv: u,v \in \mathbb R\}$ to a loop in $\mathbb T_{\boldsymbol \alpha}$ and the portion of the boundary in the left half to $\mathbb T_{\boldsymbol \beta}$.  Choose a constant trivialization of the orientable real vector bundle $\phi^*(T(\mathbb T_{\boldsymbol \alpha}))$ over $\partial D|_{v\geq 0}$.  We may tensor this real trivialization with $\mathbb C$ and extend to a complex trivialization of $\phi^*(T(\operatorname{Sym}^{g+n-1}(S))$ by pushing across the disk linearly. Relative to this trivialization, the real bundle $\phi^*(T(\mathbb T_{\boldsymbol \beta}))$ over $\partial D|_{v \geq 0}$ induces a loop of real subspaces of $\mathbb C^{g+n-1} = \phi*(T\operatorname{Sym}^{g+n-1}(S))$.  The winding number of this loop is the Maslov index of the map $\phi$.  Notice that we could also have used $\phi^*(J(T(\mathbb T_{\boldsymbol \alpha})))$ and $\phi^*(T(\mathbb T_{\boldsymbol \beta}))$, where $J$ is the complex structure on the vector bundle $T\operatorname{Sym}^{g+n-1}(S)$, and obtained the same number.

The Maslov index $\mu(\phi)$ can equivalently be computed using the associated domain $\Sigma a_i D_i$ in a formula of Lipshitz's \cite[Proposition 4.2]{MR2240908}. For each domain $D_i$, let $e(D_i)$ be the Euler measure of $D_i$. In particular, if $D_i$ is a disk with $2k$ corners, $e(D_i) = 1- \frac{k}{2}$. Let $p_{{\bf x}}(D)$ be the sum of the average of the multiplicities of $D$ at the four corners of each point in ${\bf x}$ and likewise for $p_{{\bf y}}(D)$. Then the Maslov index is
\begin{equation} \label{Maslov index formula}
\mu(\phi) = \sum a_i e(D_i) + p_{{\bf x}}(D) + p_{{\bf y}}(D).
\end{equation}

In the case that $[\phi] \in \pi_2({\bf x}, {\bf x})$ is a domain from ${\bf x}$ to itself, and therefore a periodic domain, we have the following alternate interpretation of the Maslov index, which will be very important to the proof of the main theorems of this paper. Because $\phi(i) = \phi(-i)$, we see that $\phi$ sends $\partial D|_{v \geq 0}$ to a loop in $\mathbb T_{\boldsymbol \alpha}$ and $\partial D_{v \leq 0}$ to a loop in $\mathbb T_{\boldsymbol \beta}$. Therefore we may replace $\phi$ by a map $\widehat{\phi}\co S^2 \times I \rightarrow \operatorname{Sym}^{g+n-1}(S)$ which maps $S^1 \times \{1\}$ to $\mathbb T_{\boldsymbol \alpha}$ and $S^1 \times \{0\}$ to $\mathbb T_{\boldsymbol \beta}$.  We then consider the complex pullback bundle $E = \widehat{\phi}^*(T(\operatorname{Sym}^{g+n-1}(S))$ to $S^2 \times I$ and the totally real subbundles $\widehat{\phi}|_{S^1 \times \{1\}}^{*} (T(\mathbb T_{\boldsymbol \alpha}))$ of $E|_{S^1 \times \{1\}}$ and  $\widehat{\phi}|_{S^1 \times \{0\}}^{*} (T(\mathbb T_{\boldsymbol \beta}))$ of $E|_{S^1 \times \{0\}}$. The Maslov index is still calculated by trivializing $\widehat{\phi}|_{S^1 \times \{1\}}^{*} (T(\mathbb T_{\boldsymbol \alpha}))$, complexifying, and computing the winding number of the loop of real-half dimensional subspaces in $\mathbb C^{g+n-1}$ represented by $\widehat{\phi}|_{S^1 \times \{0\}}^{*} (T(\mathbb T_{\boldsymbol \beta}))$ with respect to the trivialization. This number classifies the bundle in the following way: complex vector bundles over the annulus whose restriction to the boundary of the annulus carries a canonical real subbundle are in bijection with maps $[(S^1 \times I, \partial(S^1 \times I)), (BU,BO)] = \langle (S^1 \times I, \partial(S^1 \times I)), (BU,BO) \rangle \cong \mathbb Z$, where the map to $\mathbb Z$ is the Maslov index $\mu(\phi)=\mu(\widehat{\phi})$ \cite[Theorem C.3.7]{MR2045629}.

Now let us look at the bundle $E$ over $S^1 \times I$ and its real subbundles over the boundary components of $S^1 \times I$ from a slightly different perspective. The real bundles $\widehat{\phi}|_{S^1 \times \{1\}}^{*} (T(\mathbb T_{\boldsymbol \alpha}))$ and  $\widehat{\phi}|_{S^1 \times \{0\}}^{*} (T(\mathbb T_{\boldsymbol \beta}))$ are orientable, hence trivializable over the circle, so we may choose real trivializations and tensor with $\mathbb C$ to obtain a complex trivialization of $E|_{\partial(S^1 \times I)}$.  We can now regard $E$ as a relative vector bundle $E_{\operatorname{rel}}$ over $(S^1 \times I, \partial(S^1 \times I))$, and consider its relative first Chern class $c_1(E|_{\operatorname{rel}})$.  Equivalently, we may use this trivialization to construct a vector bundle $\widetilde{E}$ over $(S^1 \times I)/ \partial(S^1 \times I) \cong S^2$ such that the pullback $q^*(\widetilde{E})$ along the quotient map is $E$. Then $c_1(\widetilde{E})$ is the relative first Chern class $c_1(E|_{\operatorname{rel}})$ under the identification $H^1(S^2) \simeq H^1(S^1\times I, \partial(S^1 \times I))$. Moreover, isomorphism classes of vector bundles over $S^2$ are in bijection with homotopy classes of maps $[S^2, BU] \cong \langle S^2, BU \rangle = \pi_2(BU) \cong \mathbb Z$, where the identification with $\mathbb Z$ is via the first Chern class.  Using the homotopy long exact sequence of the pair $(BU,BO)$, we observe the following relationship between $\mu$ and $c_1$.
\[
\xymatrix{
\pi_2(BU) \ar[r] \ar[d]^{c_1}& \pi_2(BU, BO) \ar[r] \ar[d]^{\mu} & \pi_1(BO) \ar[d] \\
\mathbb Z \ar[r]^-{\times 2} & \mathbb Z \ar[r] & \mathbb Z_2
}
\]
Therefore the Maslov index $\mu(\phi)$ is twice the relative first Chern class $c_1(E|_{\operatorname{rel}})$.

%
%

The complex $\widehat{\mathit{CFL}}(\mathcal D)$ carries a (relative, for our purposes) homological grading called the Maslov grading $M({\bf x})$ which uses the Maslov index and takes values in $\mathbb Z$.  Suppose ${\bf x}$ and ${\bf y}$ are connected by a Whitney disk $\phi$. Then the relative Maslov grading is determined by
\begin{align*}
M({\bf x}) - M({\bf y}) &= \mu(\phi) - 2\sum_i n_{w_i}(\phi).
\end{align*}
A formula for the absolute Maslov grading may be found in \cite[Theorem 3.3]{MR2249248}, but will not be needed in this paper.

\subsection{The Alexander grading and relative $\operatorname{spin}^{\operatorname{c}}$ structures} The complex $\widehat{\mathit{CFL}}(\mathcal D)$ also carries an Alexander multigrading ${\bf A} = (A_1,\ldots,\linebreak A_{\ell})$. This multigrading takes values in an affine lattice $\mathbb H$ over  $H_1(S^3-L; \mathbb Z)\cong H_1(L)$.  Recall that $H_1(S^3-L; \mathbb Z) \cong \mathbb Z^{\ell}$ generated by the homology classes of meridians $\mu_{K_j}$ of the component knots $K_j$ of $L$.  Define the lattice $\mathbb H$ to consist of elements
\begin{align*}
\sum_{i=1}^{\ell} a_i[\mu_{K_i}]
\end{align*}
\noindent where $a_i \in \mathbb Q$ satisfies the property that $2a_i + \operatorname{lk}(K_i, L-K_i)$ is an even integer. To determine the relative Alexander multigrading, recall that the basepoints $w_{k_{j-1} +1},\ldots,w_{k_j},z_{k_{j-1} +1},\ldots,z_{k_j}$ lie on $K_j$ (with the convention that $k_0 = 0$). Then once again if $\phi$ is a Whitney disk connecting ${\bf x}$ and ${\bf y}$,
\begin{align*}
A_j({\bf x}) - A_j({\bf y}) &= \sum_{i=k_{j-1} +1}^{k_j} n_{z_i}(\phi) - \sum_{i=k_{j-1}+1}^{k_j} n_{w_i}(\phi).
\end{align*}
We can also see the relative Alexander multigrading as a linking number.  Let ${\bf x},{\bf y} \in \mathbb T_{\boldsymbol \alpha} \cap \mathbb T_{\boldsymbol \beta}$, we find paths
\begin{align*}
a\co[0,1]\rightarrow \mathbb T_{\boldsymbol \alpha}  \quad\text{and}\quad  b\co [0,1]\rightarrow \mathbb T_{\boldsymbol \beta}
\end{align*}
\noindent such that $\partial a = \partial b = {\bf x} - {\bf y}$.  (For example, $a \cup b$ may be the boundary of a Whitney disk $\phi$ from ${\bf y}$ to ${\bf x}$.)  View these paths as as one-chains on $S \backslash ({\bf w} \cup {\bf z})$. Since attaching one- and two-handles to the $\alpha$ and $\beta$ curves on $S$ and filling in three-balls at the basepoints yields $Y$, we obtain a trivial one-cycle in $Y$.  Indeed, a domain $D$ on $\mathcal D$ is the shadow of a Whitney disk if and only if its boundary, viewed as a cycle on $S$, descends to a trivial cycle on $Y$.  However, if we attach $\alpha$ and $\beta$ circles to $S \backslash ({\bf w} \cup {\bf z})$ (and no three balls) we obtain the manifold $Y-L$ and a one cycle $\epsilon({\bf x}, {\bf y})$ in $Y-L$. 
\begin{align*}
\underline{\epsilon}_{{\bf w,z}}\co (\mathbb T_{\boldsymbol \alpha} \cap \mathbb T_{\boldsymbol \beta}) \times (\mathbb T_{\boldsymbol \alpha} \cap \mathbb T_{\boldsymbol \beta}) \rightarrow H_1(Y-L; \mathbb Z)
\end{align*}
We obtain the following lemma (which has only been very slightly adjusted from the original to account for the possibility of multiple pairs of basepoints on a link component).

\begin{lemma} \cite[Lemma 3.10]{MR2443092} An oriented $\ell$-component link $L$ in $Y$ induces a map
\begin{align*}
\Pi \co H_1(Y-L) \rightarrow \mathbb Z^{\ell}
\end{align*}
where $\Pi_i(\gamma)$ is the linking number of $\gamma$ with the $i$th component $K_i$ of $L$.  In particular, for ${\bf x},{\bf y} \in \mathbb T_{\boldsymbol \alpha}\cap \mathbb T_{\boldsymbol \beta}$ and $\phi \in \pi_2({\bf x}, {\bf y})$, we have
\begin{align*}
\Pi_i(\underline{\epsilon}_{{\bf w,z}}({\bf x},{\bf y})) = \sum_{n_{i-1}+1}^{n_i} n_{z_j}(\phi) - \sum_{n_{i-1}+1}^{n_i} n_{z_j}(\phi).
\end{align*}

\end{lemma}

\begin{proof}
The proof is nearly identical to the original: $\phi$ induces a nulhomology of $\underline{\epsilon}({\bf x}, {\bf y})$, which meets the $i$th component $K_i$ of $L$ with intersection number $\sum_{n_{i-1}}^{n_1} n_{z_i}(\phi) - \sum_{n_{i-1}}^{n_1} n_{w_i}(\phi)$.
\end{proof}

There is another less practical, but more intrinsic, method of thinking about Alexander multigradings via \textit{relative $\operatorname{spin}^{\operatorname{c}}$ structures}. Let $(N, \partial N)$ be a manifold with toroidal boundary $T_1 \cup \cdots \cup T_{\ell}$. The boundary of a torus contains, up to isotopy, a canonical nowhere-vanishing vector field preserved by translation. Therefore we consider nowhere-vanishing vector fields $v$ on $N$ which restrict to the canonical vector field on each boundary component of $N$. In this case vector fields $v$ and $v'$ are homologous if they are homotopic on $N-B$ for $B$ a ball in $N^{\circ}$. The set of such homotopy classes is the set of relative $\operatorname{spin}^{\operatorname{c}}$ structures, and is an affine space for the relative homology $H^2(N, \partial N)$. This space is denoted $\underline{\operatorname{Spin}}^{\operatorname{c}}(N, \partial N)$. If $v$ is a nowhere-vanishing vector field on $N$ with canonical restriction to the toroidal boundary $\partial N$, the restriction of the field of two-planes $v^{\bot}$ has a canonical trivialization along $\partial N$. Therefore there is a well-defined notion of the relative first Chern class $c_1(\underline{\mathfrak{s}})$ of a relative $\operatorname{spin}^{\operatorname{c}}$ structure.

There is a canonical way of associating to every generator ${\bf x}$ in $\mathbb T_{\boldsymbol \alpha} \cap \mathbb T_{\boldsymbol \beta}$ a relative $\operatorname{spin}^{\operatorname{c}}$ structure $\underline{\mathfrak{s}}_{{\bf w, \bf z}}({\bf x})$ on $(S^3 - \nu(L))$ by modifying the gradient vector field of the Morse function on $S^3$ corresponding to the Heegaard diagram near flowlines corresponding to ${\bf x}$. This is described in detail in \cite[Section 3.3]{MR2443092}. Then the relationship between this relative $\operatorname{spin}^{\operatorname{c}}$ structure and the Alexander grading as defined previously is
\begin{align*}
c_1(\underline{\mathfrak{s}}_{{\bf w, \bf z}}({\bf x})) + \sum_{i=1}^{\ell} PD[\mu_{K_i}] = 2PD\left(A_1[\mu_1]+\cdots A_{\ell}[\mu_{\ell}]\right).
\end{align*}

Finally, before moving on, observe that if $\mathcal D$ is a Heegaard diagram for $(Y, L)$, then there is an identification
\begin{align*}
&H_2(Y-\nu(L), \partial(Y - \nu(L)) = H_1(Y - L)\\
\cong\;& \frac{H_1(S\backslash \{{\bf w\cup z}\})}{\langle [\alpha_1],\ldots, [\alpha_{g_n-1}], [\beta_1],\ldots,[\beta_{g+n-1}] \rangle} \cong \frac{H_1(\mathrm{Sym}^{g+n-1}(S \backslash\{\bf w\cup z\}))}{H_1(\mathbb T_{\boldsymbol \alpha}) \oplus H_1(\mathbb T_{\boldsymbol \beta})}.
\end{align*}
Under this identification the set $\underline{\text{Spin}}^{\text{c}}(Y-\nu(L), \partial (Y-\nu(L)))$ is canonically identified with the set of homotopy classes of paths $\mathcal P(\mathbb T_{\boldsymbol \alpha}, \mathbb T_{\boldsymbol \beta})$. We will discuss the homology and cohomology of punctured symmetric products at greater length in Section \ref{Geometry Section}.

\subsection{Link Floer homology and an alternate differential} The differential $\partial$ lowers the Maslov grading by one and preserves the Alexander multigrading. Therefore $\widehat{\mathit{CFL}}(\mathcal D)$ also splits along Alexander grading.

The homology of $\widehat{\mathit{CFL}}(\mathcal D)$ with respect to the differential $\partial$ is very nearly the link Floer homology of $(Y,K)$. There is, however, a slight subtlety having to do with the number of pairs of basepoints $z_i$ and $w_i$ on $\mathcal D$. Let $V_i$ be a vector space over $\mathbb Z_2$ with generators in gradings $(M,(A_1,\ldots,A_{\ell})) = (0,(0,\ldots,0))$ and $(M,(A_1,\ldots,A_j,\ldots,A_{\ell})) = (-1,(0,\ldots,-1,\ldots,0))$, with the $-1$ in the $j$th component.  As before, let $K_j$ carry $n_j$ pairs of basepoints.

\begin{definition}
The homology of the complex $\widehat{\mathit{CFL}}(\mathcal D)$ with respect to the differential $\partial$ is
\[
\widetilde{\mathit{HFL}}(\mathcal D) = \widehat{\mathit{HFL}}(S^3,L) \otimes V_1^{\otimes (n_1-1)}\otimes\cdots\otimes V_{\ell}^{\otimes (n_{\ell}-1)}.
\]

\end{definition}

The theory $\widehat{\mathit{HFL}}(S^3,L)$ is symmetric with respect to the Alexander multigrading as follows.  Let $\widehat{\mathit{HFL}}_d(S^3,L, (A_1,\ldots,A_\ell))$ be the summand of the link Floer homology of $L$ in Alexander multigrading $(A_1,\ldots,A_\ell)$ and Maslov grading $d$.

\begin{proposition} \cite[Proposition 8.2]{MR2443092} There is an isomorphism 
\[
\widehat{\mathit{HFL}}_d(S^3, L, (A_1,\ldots,A_\ell)) \cong \widehat{\mathit{HFL}}_{d-\sum A_i}(S^3, L, (-A_1,\ldots,-A_\ell)).
\]

\end{proposition}

In particular, ignoring Maslov gradings, we see that the link Floer homology is symmetric in each of its Alexander gradings.

Before moving on, let us consider one addition differential on the complex $\widehat{\mathit{CFL}}(\mathcal D)$.  Suppose that in addition to the disks we counted previously, we also include disks passing over $z$ basepoints on the component $K_j$ of $L$.  In other words, let the differential $\partial_{K_j}$ be the differential of Lagrangian Floer cohomology in $\operatorname{Sym}^{g+n-1}(S \backslash ({\bf w} \cup \{z_i: z_i \not\in K_j\})$. This has the effect of discounting the contribution of the component $K_j$ to the link Floer homology, but of maintaining the effect of an extra $n_j$ pairs of basepoints on the Heegaard surface.  Ergo we have the following proposition.  Let $W$ be a two-dimensional vector space over $\mathbb Z_2$ with summands in gradings $(M, (A_1,\ldots,\widehat{A_j},\ldots,A_{\ell})) = (0, (0,\ldots,0))$ and $(M, (A_1,\ldots,\widehat{A_j},\ldots, A_{\ell})) = (-1, (0,\ldots,0))$.

\begin{proposition} \label{Ignore Component Proposition} \cite[Proposition 7.2]{MR2443092} The homology of the complex\linebreak $\widehat{\mathit{CFL}}(S^3, L)$ with respect to the differential $\partial_{K_j}$ is isomorphic to $\widehat{\mathit{HFL}}(S^3,\linebreak L-K_j) \otimes V_1^{\otimes (n_1-1)} \otimes\cdots\otimes W^{\otimes n_j} \otimes\cdots\otimes V_{\ell}^{\otimes (n_{\ell} -1)}$.

\end{proposition}

We may think of Proposition \ref{Ignore Component Proposition} as the assertion that there is a spectral sequence from the $\mathbb Z^{\ell}$ graded theory $\widehat{\mathit{HFL}}(S^3, L)$ to the $\mathbb Z^{\ell-1}$ graded theory $\widehat{\mathit{HFL}}(S^3, L-K_j)$ by computing all differentials that change the $k_j$th entry of the $\mathbb Z^{\ell}$ multigrading.  This spectral sequence comes with an overall shift in relative Alexander gradings, which is computed by considering fillings of relative $\operatorname{spin}^{\operatorname{c}}$ structures on $S^3-L$ to relative $\operatorname{spin}^{\operatorname{c}}$ structures on $S^3 - (L- K_j)$.  The generally slightly complicated formula admits a simple expression in the case of two-component links, which is the only case of interest to this paper.

\begin{lemma} \label{Gradings Shift Lemma} \cite[Lemma 3.13]{MR2443092} Let $L = K_1 \cup K_2$, and $\lambda= \operatorname{lk}(K_1,K_2)$. Then suppose $\mathcal D$ is a Heegaard diagram for $(S^3,L)$, and ${\bf x} \in \widetilde{\mathit{CFL}}(\mathcal D)$.  If $(A_1({\bf x}), A_2({\bf x})) = (i,j)$ in the complex $\widetilde{\mathit{CFL}}(\mathcal D)$ with differential $\partial$, then in the complex $\widetilde{\mathit{CFK}}(\mathcal D)$ with differential $\partial_{K_2}$, the Alexander grading of ${\bf x}$ is $A_1({\bf x}) = i - \frac{\lambda}{2}$.

\end{lemma}

That is, forgetting one component of a two-component link has the effect of shifting Alexander gradings of the other component downward by $\frac{\lambda}{2}$.  The proof comes from an analysis of filling relative $\operatorname{spin}^{\operatorname{c}}$ structures; the effect of extending a relative $\operatorname{spin}^{\operatorname{c}}$ structure $\mathfrak s$ on $S^3-\nu L$ to $S^3 - \nu(L-K_j)$ is to shift the Chern class $c_1(\mathfrak s)$ by the Poincar\'{e} dual of the homology class of $K_j$ in $S^3 - \nu(L-K_j)$.  For a two-component link this is a shift by the linking number.

\subsection{Link Floer homology, the multivariable Alexander polynomial, and the Thurston norm}

Recall that the \textit{multivariable Alexander polynomial} of an oriented link $L = K_1 \cup \cdots \cup K_{\ell}$ is a polynomial invariant $\Delta_L(t_1,\ldots,t_{\ell})$ with one variable for each component of the link.  While its relationship to the Alexander polynomials of the component knots is in general slightly complicated, in the case of a two-component $L = K_1 \cup K_2$ Murasugi proved the following using Fox calculus.

\begin{lemma} \cite[Proposition 4.1]{MR0292060} \label{Two Component Link Lemma}
Let $L=K_1\cup K_2$ be an oriented two-component link with $\operatorname{lk}(K_1, K_2) = \lambda$.  If $\Delta_L(t_1,t_2)$ is the multivariable Alexander polynomial of $L$ and $\Delta_{K_1}(t_1)$ is the ordinary Alexander polynomial of $K_1$, then 
\begin{align*}
\Delta_L(t_1, 1) = (1+t+t^2+ \cdots +t^{\lambda-1}) \Delta_{K_1}(t)
\end{align*}

\end{lemma}

The Euler characteristic of link Floer homology encodes the multivariable Alexander polynomial of the link as follows. Let
\begin{align*}
\chi(\widehat{\mathit{HFL}}(S^3,L)) = \sum_{(A_1,\ldots,A_{\ell})\in \mathbb H} t_1^{A_1({\bf x})}\cdots t_{\ell}^{A_{\ell}({\bf x})}\chi(\widehat{\mathit{HFL}}(S^3, L, (A_1,\ldots, A_{\ell})). 
\end{align*}

\begin{proposition} \cite[Theorem 1.3]{MR2443092} If $L$ is an oriented link, the Euler characteristic of its link Floer homology is equal to
\begin{align*}
\chi(\widehat{\mathit{HFL}}(S^3,L)) = \begin{cases}
	\left(\prod_{i=1}^{\ell} (t_i^{\frac{1}{2}} - t_i^{-\frac{1}{2}})\right) \Delta_L(t_1,\ldots , t_{\ell}) & \ell>1 \\
	\Delta_L(t_1) & \ell=1.
\end{cases}
\end{align*}

\end{proposition}

Link Floer homology also categorifies the Thurston seminorm of the link complement.  Let us recall the definition of the Thurston seminorm in this context. Recall that the \textit{complexity} of a compact, oriented, surface with boundary $F$ with components $F_i$ is
\begin{align*}
\chi_{-}(F) = \sum_{F_i: \chi(F_i)\leq 0} \chi(F_i).
\end{align*}

For any homology class in $H_2(S^3,L)$, there is some compact, oriented surface with boundary in $S^3 - \bigcup\nu(K_i)$ which represents $h$. We can consider a function
\begin{align*}
x_L \co H_2(S^3,L) &\rightarrow \mathbb Z \\
h &\mapsto \min_{F\hookrightarrow S^3\bigcup\nu(K_i), [F] = h} \chi_{-}(F).
\end{align*}

Thurston \cite{MR823443} shows this extends to a seminorm on $x_L$ on $H_2(S^3, L; \mathbb R)$, called the Thurston seminorm.

Let us pause to be a little more concrete about what surfaces we will consider in computing the Thurston seminorm of a link. Notice that $H_2(S^3,L)$ is generated by the duals of the meridians $\mu_{K_i}$ of each component $K_i$ of the link. Through an abuse of notation, and to avoid confusion with cohomology classes later, we will also refer to these duals as $[\mu_{K_i}]$. Computing the Thurston seminorm of the element of $H_2(S^3,L)$ which is dual to $\sum a_i[\mu_{K_i}] \in H_1(L)$ is a matter of computing the minimal Euler characteristic of an embedded surface $F$ with no sphere components whose intersection with a meridian $\mu_{K_i}$ of $K_i$ is $a_i$ for each $i$.  In particular, $x([\mu_{K_i}])$ is the minimal Euler characteristic of surface $F$ with boundary one longitude of $K_i$ and an arbitrary number of meridians of the components of $L$. (For practical purposes, one may consider taking a Seifert surface $F$ for $K_i$ and puncturing $F$ wherever it intersects some other component of $L$. However, take note that puncturing a minimal Seifert surface for $K_i$ does not necessarily result in a Thurston-norm minimizing surface.)

Link Floer homology yields a related function.  Recall that $\mathbb H \subset H^2(S^3,L;\linebreak \mathbb R)\cong H_1(L; \mathbb R)$ is the affine lattice of real second cohomology classes $h=\sum A_i [\mu_{K_i}]$ for which $\widehat{\mathit{HFL}}(S^3,L, {\bf A})$ is defined. We have
\begin{align*}
y \co H^1(S^3 - L; \mathbb R) \rightarrow \mathbb R
\end{align*}
\noindent which is defined by 
\[
y(\gamma) = \operatorname{max}_{\{\sum A_i [\mu_{K_i}] \in \mathbb H: \widehat{\mathit{HFL}}(L,{\bf A}) \neq 0\}} \left|\left\langle \sum A_i\widehat{[\mu_{K_i}]},\gamma \right\rangle\right|.
\]

The categorification considers the case of links with no \textit{trivial components}, that is, unknotted components unlinked with the rest of the link.

\begin{proposition} \cite[Thm 1.1]{MR2393424} \label{Statement of Categorification} Let $L$ be an oriented link with no trivial components.  Given $\gamma \in H^1(S^3 - L; \mathbb R)$, the link Floer homology groups determine the Thurston norm of $L$ via the relationship
\begin{align*}
x_L(\operatorname{PD}[\gamma]) + \sum_{i=1}^{\ell} | \langle \gamma, \mu_{K_i} \rangle | = 2y(\gamma).
\end{align*}

Here $\mu_{K_i}$ is the homology class of the meridian for the $i$th component of $L$ in $H_1(S^3 -L; \mathbb R)$, and therefore $| \langle \gamma, \mu_{K_i} \rangle |$ is the absolute value of the Kronecker pairing of $h$ with $\mu_{K_i}$.

\end{proposition}

We will primarily evaluate this equality on the dual classes to the meridians $\mu_{K_i}$ themselves. In the case that $K$ is a knot, so that $x_K([\mu_K]) = 2g(K) - 1$, Proposition \ref{Statement of Categorification} reduces to the familiar theorem of  \cite[Theorem 1.2]{MR2023281} that the top Alexander grading $i$ for which $\widehat{\mathit{HFK}}(S^3,K,i)$ is nontrivial is the genus of the knot.  In general, observe that if we evaluate on $\mu_{K_i}$, we obtain $x_L([\mu_{K_i}]) + 1 = 2y(\mu_{K_i})$.  In other words, the total breadth of the $A_i$ Alexander grading in the link Floer homology is the Thurston norm of the dual to $K_i$ plus one.

Before leaving the realm of link Floer homology background, we will require one further result concerning the knot Floer homology of fibred knots.

\begin{proposition}\cite[Thm 1.1]{MR2357503}, \cite[Thm 1.4]{MR2450204} \label{Fiberedness Condition Lemma}
Let $K$ be a knot, and $g(K)$ its genus. Then $K$ is fibered if and only if $\widehat{\mathit{HFK}}(S^3, K, g(K)) = \mathbb Z_2)$.
\end{proposition}

The forward direction (that if $K$ is fibred, then the knot Floer homology in the top nontrivial Alexander grading is $\mathbb Z_2$) is due to Ozsv{\'a}th and Szab{\'o} \cite[Theorem 1.1]{MR2153455}, whereas the other direction was proved by Ghiggini \cite{MR2450204} in the case $g=1$ and Ni \cite{MR2357503} in the general case.

We are now ready to consider the specific case of periodic knots.

\section{Proofs of Murasugi's and Edmonds' conditions} \label{Periodic Knots Section}

Let $\widetilde{K}\subset S^3$ be an oriented $q$-periodic knot and $K$ its quotient knot. We will begin by constructing a Heegaard diagram for $(S^3, \widetilde{K} \cup \widetilde{U})$ which is preserved by the action of $\mathbb Z_q$ on $(S^3, \widetilde{K} \cup \widetilde{U})$ and whose quotient under this action is a Heegaard diagram for $(S^3,K\cup U)$.

\subsection{Heegaard diagrams for periodic knots}

As in \cite{Hendricks}, it will be necessary to work with Heegaaard diagrams for $(S^3, K \cup U)$ on the sphere $S^2$. Regard $S^3$ as $\mathbb R^2 \cup \{\infty\}$ and arrange $\widetilde{K}$ such that the unknotted axis of periodicity $\widetilde{U}$ is the $z$-axis together with the point at infinity, and the action is given by rotation by $2\pi/q$.  Then the projection of $\widetilde{K}$ to the $xy$-plane together with the point at infinity is a periodic diagram $\widetilde{E}$ for $\widetilde{K}$.  Taking the quotient of $(S^3, \widetilde{K})$ by the action of $\mathbb Z_q$ and similarly projecting to the $xy$-plane together with the point at infinity produces a quotient diagram $E$ for $K$.

Construct a Heegaard diagram for $K  \cup  U$ as follows: Begin with the diagram $E$ on $S^2 = \mathbb R^2 \cup \{\infty\}$.  Place a basepoint $w_0$ at $\infty$ and $z_0$ at $0$; these will be the sole basepoints on $U$. (This is a slight departure from the notation of Section \ref{Heegaard Floer Background Section}; it will be more convenient to have the indexing start at $w_0$ rather than $w_1$ for the diagrams we construct.) Arrange basepoints $z_1,w_1,\ldots ,z_{n_1},w_{n_1}$ on $K$ such that traversing $K$ in the chosen orientation, one passes through the basepoints in that order.  Moreover, we insist that while travelling from $z_i$ to $w_i$ one passes only through undercrossings and travelling from $w_i$ to $z_{i+1}$ or from $w_n$ to $z_1$ one passes only through overcrossings.  In other words, we choose basepoints so as to make $E$ into a bridge diagram for $K$.  Notice that $n_1$ is at most the number of crossings on the diagram $E$, or half the number of crossings on $\widetilde{E}$.  Encircle the portion of the knot running from $z_i$ to $w_i$ with a curve $\alpha_i$, oriented counterclockwise in the complement of $w_0$. Similarly, encircle the portion of the knot running from $w_{i}$ to $z_{i+1}$ (or from $w_{n_1}$ to $z_1$) with a curve $\beta_i$, oriented counterclockwise in the complement of $w_0$. Notice that both $\alpha_i$ and $\beta_i$ run counterclockwise around $w_i$, and moreover for each $i$, $S^2 \backslash \{\alpha_i, \beta_i\}$ has four components: one each containing $z_{i}, w_i$, and $z_{i+1}$, and one containing all other basepoints. This yields a Heegaard diagram $\mathcal D = (S^2, \boldsymbol{\alpha}, \boldsymbol{\beta}, {\bf w}, {\bf z})$ for $(S^3, K \cup U)$.

We may now take the branched double cover of $\mathcal D$ over $z_0$ and $w_0$ to produce a Heegaard diagram $\widetilde{\mathcal D}$ for $(S^3, \widetilde{K} \cup \widetilde{U})$ compatible with $\widetilde{E}$.  This diagram has basepoints $w_0$ and $z_0$ for $\widetilde{U}$ and basepoints $z_1^1, w_1^1,\ldots,z_{n_1}^1,w_{n_1}^1,\linebreak  z_1^2,\ldots,w_{n_1}^2,\ldots,z_1^q,\ldots,w_{n_1}^q$ arranged in that order along the oriented knot $\widetilde{K}$. Each adjacent pair $z_i^a$ and $w_i^a$ is encircled by $\alpha_i^a$ a lift of $\alpha_i$, and each adjacent pair $w_{i}^a$ and $z_{i+1}^a$ is encircled by $\beta_i^a$ a lift of $\beta_i$.  (Pairs $w_{n_1}^a$ and $z_1^{a+1}$, as well as $w_{n_1}^q$ and $z_1^1$, are encircled by lifts $\beta_{n_1}^a$ of $\beta_{n_1}$.)  This yields a diagram $\widetilde{\mathcal D} = (S^2, \widetilde{\boldsymbol{\alpha}}, \widetilde{\boldsymbol{ \beta}}, \widetilde{\bf w}, \widetilde{\bf z})$ with $qn_1$ each of $\alpha$ and $\beta$ curves and $qn_1 +1$ pairs of basepoints.

\begin{remark}
The notation above is not quite the notation of \cite{Hendricks}, in which the two lifts of a curve $\alpha_i$ in a Heegaard surface $\mathcal D$ to its double branched cover $\widetilde{\mathcal D}$ were $\widetilde{\alpha_i}$ and $\tau(\widetilde{\alpha_i})$.  In the new slightly more streamlined notation, adopted in view of the need to work with $q$-fold branched covers and to consider multiple lifts of some of the basepoints, these two curves would be $\alpha_i^1$ and $\alpha_i^2$.
\end{remark}

Let us pause to introduce some notation on the diagram $\mathcal D$ that will be useful in Section \ref{Stable Normal Triv Section}.  Let $x_i$ be the single positive intersection point between $\alpha_i$ and $\beta_i$ and $y_i$ the negative intersection point.  Moreover, let $F_i$ be the closure of the component of $S - \alpha_i - \beta_i$ containing $z_{i}$ and $E_i$ be the closure of the component of $S - \alpha_i -\beta_i$ containing $z_{i+1}$ (or $z_1$ if $i=n_1$).  Then $P_i = E_i - F_i$ is a periodic domain of index zero on $\mathcal D$.  Finally, let $\gamma_i$ be the union of the arc of $\alpha_i$ running from $x_i$ to $y_i$ and the arc of $\beta_i$ running from $x_i$ to $y_i$.  In particular, this specifies that $\gamma_i$ has no intersection with any $\alpha$ or $\beta$ curves other than $\alpha_i$ and $\beta_i$, and moreover the component of $S - \gamma_i$ which does not contain $w_0$ contains only a single basepoint $w_i$.

\begin{figure}
\centering
\includegraphics[scale=.8]{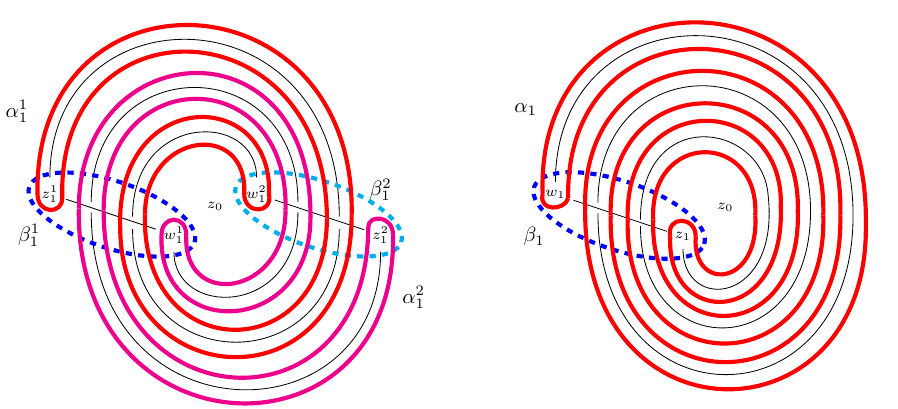}
\caption{An equivariant Heegaard diagram $\widetilde{\mathcal D}$ for the trefoil together with the unknotted axis, and its quotient Heegaard diagram $\mathcal D$ for the Hopf link.}
\label{Trefoil Heegaard Diagrams Figure}
\end{figure}

Our next goal will be to investigate the behavior of the relative Maslov and (particularly) Alexander gradings of generators of $\widehat{\mathit{CFK}}\!(\mathcal D)$ and $\widehat{\mathit{CFK}}\!(\widetilde{\mathcal D})$.  We begin with two relatively simple lemmas. As before, let $\tau$ be the involution on $(S^3, \widetilde{K})$ (and on $\widetilde{\mathcal D}$).  Let $\tau_{\#}$ be the induced involution on $\widehat{\mathit{CFK}}\!(\widetilde{\mathcal D})$.

\begin{lemma}
The induced map $\tau_{\#}$ preserves Alexander and Maslov gradings.
\end{lemma}

\begin{proof}
Let ${\bf s} \in \mathbb T_{\widetilde{\boldsymbol \alpha}} \cap \mathbb T_{\widetilde{\boldsymbol \beta}}$.  Choose a generator ${\bf x} \in \widehat{\mathit{CFK}}(\mathcal D)$, and let $\widetilde{\bf x} =\linebreak \pi^{-1}({\bf x})$, such that $\widetilde{\bf x}$ is a generator in $\widehat{\mathit{CFK}}(\widetilde{\mathcal D})$ which is invariant under $\tau_{\#}$.  Choose a Whitney disk $\phi$ in $\pi_2({\bf s},\widetilde{\bf x})$ and let  $D$ be its shadow on $\mathcal D$.  Then $\tau \circ \phi$ is a Whitney disk in $\pi_2(\tau({\bf s}),\widetilde{\bf x})$ with shadow $\tau(D)$. Furthermore, since $w_0$ and $z_0$ are fixed by the involution, $n_{z_0}(D)=n_{z_0}(\tau(D))$ and $n_{w_0}(D) = n_{w_0}(\tau(D))$, whereas since the remaining basepoints are interchanged by the involution, we have $\sum_{i=0}^{n_1} n_{z_i}(D) = \sum_{i=0}^{n_1} n_{z_i}(\tau(D))$ and $\sum_{i=0}^{n_1} n_{w_i}(D)\linebreak = \sum_{i=0}^{n_1} n_{w_i}(\tau(D))$.  These equalities imply that $M(\widetilde{\bf x}) - M({\bf s}) = M(\widetilde{\bf x}) -\linebreak M(\tau_{\#}({\bf s}))$, and similarly for $A_1$ and $A_2$.  Therefore ${\bf s}$ and $\tau_{\#}({\bf s})$ are in identical gradings.
\end{proof}

\begin{lemma} \label{Maslov Index Lemma}
Let $\phi \in \pi_2({\bf x},{\bf y})$ be a Whitney disk between generators ${\bf x}$ and ${\bf y}$ in $\widehat{\mathit{CFK}}(\mathcal D)$ with Maslov index $\mu(\phi)$, with shadow the domain $D$ on $\mathcal D$. There is a Whitney disk $\widetilde{\phi} \in \pi_2(\widetilde{{\bf x}},\widetilde{{\bf y}})$ with shadow the domain $\pi^{-1}(D)$ on $\widetilde{\mathcal D}$, and $\mu(\widetilde{\phi}) = q \mu(\phi) - (q-1)(n_{z_0}(\phi) + n_{w_0}(\phi))$.

\end{lemma}

\begin{proof}
The boundary of the lift $\pi^{-1}(D)$ is trivial as a one cycle in $S^3$, implying that $\pi^{-1}(D)$ is the shadow of a Whitney disk $\widetilde{\phi} \in \pi_2(\widetilde{{\bf x}},\widetilde{{\bf y}})$.  We will compare the Maslov index of $\phi$ with the Maslov index of $\widetilde{\phi}$ using the formula~\ref{Maslov index formula}.  As in that formula, we will write $D$ as a sum of the closures of the components of $S^2 - {\boldsymbol \alpha} - {\boldsymbol \beta}$. Say there are $m$ such components, and label them as follows.  There are two domains in $S^2 - {\boldsymbol \alpha} -{\boldsymbol \beta}$ which contain a branch point.  Let these be $D_1$ containing $z_0$ and $D_2$ containing $w_0$.  Let the shadow of $\phi$ be $a_1D_1 + a_2D_2 +\sum_{i=3}^{m} a_iD_i$. Then the Maslov index of $\phi$ is
\[
\mu(\phi) = \sum_i a_i e(D_i) + p_{\bf x}(D) + p_{\bf y}(D)
\]
Let us now consider applying the same formula to $\pi^{-1}(D) = \sum_i a_i \pi^{-1}(D_i)$. For $i\geq 3$, the lift $\pi^{-1}(D_j)$ of $D_j$ consists of $q$ copies of $D_i$, and by additivity of the Euler measure we see that $e(\widetilde{D}_i) = qe(D_i)$.  For $i=1,2$, let $2k_i$ be the number of corners of $D_i$.  Then $\pi^{-1}(D_i) = \widetilde{D}_i$ is a single component of $S^2 - \widetilde{\boldsymbol \alpha} - \widetilde{\boldsymbol \beta}$ with $2qk_i$ corners and Euler measure $e(\widetilde{D}_i) = 1 - \frac{qk_i}{2} = q(e(D_i)) - (q-1)$.  Notice, furthermore, that $p_{\widetilde{\bf x}}(\widetilde{D}) = q p_{\bf x}(D)$ and $p_{\widetilde{\bf y}}(\widetilde{D}) = q p_{\bf y}(D)$. Therefore we compute
\begin{align*}
\mu(\widetilde{\phi}) &= \sum_i a_i e(\pi^{-1}(D_i)) + p_{\widetilde{\bf x}}(\widetilde{D}) + p_{\widetilde{\bf y}}(\widetilde{D}) \\
						&= a_1 e(\widetilde{D}_1) + a_2 e(\widetilde{D}_2) + q \sum_{i=3}^m e(D_i) + qp_{\bf x}(D) + q p_{\bf y}(D)\\
						&=a_1 (qe(D_1) - (q-1)) + a_2(qe(D_2) - (q-1)) \\
                        &\quad + q \sum_{i=3}^m e(D_i) + qp_{\bf x}(D) + q p_{\bf y}(D)\\
						&=q\mu(\phi) - (q-1)(a_1 + a_2).
\end{align*}

Since $a_1+a_2$ is exactly $n_{z_0}(D) + n_{w_0}(D)$ the total algebraic intersection of $\phi$ with the branch points, this proves the result. \end{proof}

We can now construct the relationship between the Alexander gradings of the generators of $\widehat{\mathit{CFL}}(\widetilde{\mathcal D})$ and $\widehat{\mathit{CFL}}(\mathcal D)$.  For the case of a $q$-periodic knot, we will look only at the relative gradings; later in the particular case of a doubly-periodic knot we will fix the absolute gradings using symmetries of link Floer homology.  Let $\pi \co \widetilde{\mathcal D} \rightarrow \mathcal D$ be the restriction of the branched covering map $(S^3, \widetilde{K} \cup \widetilde{U}) \rightarrow (S^3, K\cup U)$.  We have the following, which is analogous to \cite[Lemma 3.1]{MR2443111}.

\begin{lemma} \label{Generator Matching Lemma}

Let ${\bf s} \in \widehat{\mathit{CFL}}(\widetilde{\mathcal D})$, thought of as a $qn_1$-tuple of points on $S^2$ with one point on each $\alpha_i^{k}$ and one on each $\beta_i^{k}$. Consider its projection $\pi({\bf s})$ to $qn_1$ points on $\mathcal D$. There is a (not at all canonical) way to write $\pi({\bf s})$ as ${\bf s_1} \cup {\bf s_2} \cup \cdots \cup {\bf s_q}$ a union of generators in $\widehat{\mathit{CFL}}(\mathcal D)$.

\end{lemma}

The proof of this lemma (pointed out by Adam Levine) is an application of the following combinatorial result of Hall \cite{Hall}.  Let $A$ be a set, and $\{ A_i\}_{i=1}^m$ be a collection of finite subsets (a version also exists for infinitely many $A_i$). A \textit{system of distinct representatives} is a choice of elements $j_i \in A_i$ for each $i$ such that $j_{i_1} \neq j_{i_2}$ if $i_1 = i_2$.  Hall's theorem gives conditions under which a system of distinct representatives exists.

\begin{theorem} \cite[Theorem 1]{Hall} Let $\{A_i\}_{i=1}^m$ be finitely many subsets of a set $A$.  Then a system of distinct representatives exists if and only if, for any $1\leq s\leq m$ and $1<i_1<\cdots<i_s<m$,  $A_{i_1} \cup\cdots \cup A_{i_s}$ contains at least $s$ elements.
\end{theorem}

Using this, we may prove Lemma \ref{Generator Matching Lemma}

\begin{proof}[Proof of Lemma \ref{Generator Matching Lemma}]  Let $A = \{1,\ldots,q\}$. For $1 \leq i \leq n_1$, let $A_i$ be a set of integers $j$, with $1\leq j \leq n_1$, such that each $j$ appears once in $A_i$ for every point of $\pi({\bf s}) \cap (\alpha_i \cap \beta_j)$.  That is, the sets $A_i$ record how many intersection points on $\alpha_i$ also lie on $\beta_j$.  Notice that there are $q$ elements in each $A_i$, and each $j$ appears exactly $q$ times in $\coprod_i A_i$.  We claim the sets $A_i$ satisfy the condition of Hall's theorem.  For given $1 \leq i_1 \leq \cdots \leq i_m \leq n_1$, the disjoint union $\coprod_{k=1}^{m} A_{i_k}$ contains $qm$ elements, and therefore must contain at least $m$ different integers $j$.  Therefore $\bigcup_{k=1}^{m} A_{i_k}$ contains at least $m$ elements.  Hence we can choose a set of distinct representatives $j_i$ in $A_i$.  There is a generator ${\bf s_1}$ consisting of points in $\pi({\bf s})$ on $\alpha_i \cap \beta_{j_i}$.  Remove these points from $\pi({\bf s})$ (and the individual $j_i$ from the sets $A_i$, producing new sets $A_i'$) and repeat the argument, now with $q-1$ elements in each $A_i'$ and $q-1$ appearances of each symbol $j$ in $\coprod_i A_i'$.  After $q$ repetitions, we have broken ${\bf s}$ into ${\bf s_1}\cup\cdots\cup {\bf s_q}$, where each ${\bf s_i}$ is a generator for  $\widehat{\mathit{CFL}}(\mathcal D)$.  This choice of partition is not at all unique.
\end{proof}

We can start by determining the relative Alexander gradings of generators of $\widehat{\mathit{CFL}}(\widetilde{\mathcal D})$ which are invariant under the action of $\mathbb Z_q$ on $\widetilde{\mathcal D}$; that is, exactly those generators which are total lifts of generators in $\widehat{\mathit{CFL}}(\mathcal D)$ under the projection map $\pi$.

\begin{lemma}

Let ${\bf x}, {\bf y} \in \mathbb T_{\boldsymbol \alpha} \cap \mathbb T_{\boldsymbol \beta}$ and $\widetilde{\bf x}, \widetilde{\bf y}$ be their total lifts in $\mathbb T_{\widetilde{\boldsymbol \alpha}} \cap \mathbb T_{\widetilde{\boldsymbol \beta}}$.  Then
\begin{align*}
A_1(\widetilde{\bf y}) - A_1(\widetilde{\bf x}) &= q\left(A_1({\bf y}) - A_1({\bf x})\right) \\
A_2(\widetilde{\bf y}) - A_2(\widetilde{\bf x}) &= A_2({\bf y}) - A_2({\bf x}).
\end{align*}
\end{lemma}

\begin{proof}
Let $D$ be a domain from ${\bf x}$ to ${\bf y}$ on $\mathcal D$.  Then $\pi^{-1}(D) = \widetilde{D}$ is a domain from ${\bf x}$ to ${\bf y}$ on $\widetilde{\mathcal D}$. Since $z_0$ and $w_0$ are branch points of $\pi$, we see that $n_{z_0}(\widetilde{D}) = n_{z_0}(D)$ and $n_{w_0}(\widetilde{D}) = n_{w_0}(D)$.  Therefore
\begin{align*}
A_2(\widetilde{\bf x}) - A_2(\widetilde{\bf y}) &= n_{z_0}(\widetilde{D}) - n_{w_0}(\widetilde{D}) \\
									&=n_{z_0}(D) - n_{w_0}(D) \\
									&=A_2({\bf x}) - A_2({\bf y})
\end{align*}

However, for $i\neq 0$, each of $z_i$ and $w_i$ basepoints on $\mathcal D$ has $q$ preimages in $\widetilde{\mathcal D}$.  Moreover $n_{z_i^j}(\widetilde{D}) = n_{z_i}(D)$ for all $1 \leq i \leq n_1$ and $1 \leq j \leq q$, and similarly for $w_i^j$, so we compute as follows.
\begin{align*}
A_1(\widetilde {\bf x}) - A_1(\widetilde{\bf y}) &= \sum_{i=1}^{n_1}\sum_{j=1}^q n_{z_i^j}(\widetilde{D}) - \sum_{i=1}^{n_1}\sum_{j=1}^q n_{w_i^j}(\widetilde{D}) \\
&= q \left( \sum_{i=1}^{n_1} n_{z_i}(D) - \sum_{i=1}^{n_1} n_{w_i}(D)\right) \\
&= q \left( A_1({\bf x}) - A_1({\bf y})\right)\\[-3.5em]
\end{align*}
\end{proof}

Finally, we can use this fact to compute the relative Alexander gradings in $\widehat{\mathit{CFL}}(\widetilde{\mathcal D})$.

\begin{lemma}

Let ${\bf s}, {\bf r}$ be generators of $\widehat{\mathit{CFL}}(\widetilde{\mathcal D})$, whose projection to $\mathcal D$ can be written as $\pi_{*}({\bf s}) = {\bf s}_1 \cup \cdots \cup {\bf s}_q$ and ${\bf r} = {\bf r}_1 \cup \cdots \cup {\bf r}_q$. Then the relative Alexander gradings between ${\bf s}$ and ${\bf r}$ is described by 
\begin{align*}
A_1({\bf s}) - A_1({\bf r}) &= \sum_{j=1}^{q} \left(A_1({\bf s}_j) - A_1({\bf r}_j)\right) \\
A_2({\bf s}) - A_2({\bf r}) &= \frac{1}{q} \left(\sum_{j=1}^{q} \left(A_2({\bf s}_j) - A_2({\bf r}_j)\right)\right).
\end{align*}
\end{lemma}

\begin{proof}
The proof is quite similar to the first half of the argument of \cite[Proposition 3.4]{MR2443111}. As in Section \ref{Heegaard Floer Background Section}, choose paths $a \co [0,1] \rightarrow \mathbb T_{\widetilde{\boldsymbol \alpha}}$, $b \co [0,1] \rightarrow \mathbb T_{\widetilde{\boldsymbol \beta}}$, with $\partial a = \partial b = {\bf s} - {\bf r}$. Then let $\epsilon({\bf s},{\bf r}) = a -b $ be a one-cycle in $H_1(S^3 - (\widetilde{K} \cup \widetilde{U}); \mathbb Z)$.  Consider the projection $\pi_{*}(\epsilon({\bf s},{\bf r}))$ to $\mathcal D$.  The restriction of this one-cycle to any $\alpha$ or $\beta$ curve consists of $q$ possibly overlapping arcs. By adding copies of the $\alpha$ or $\beta$ circle if necessary, we may arrange that these arcs connect a point in ${\bf s_j}$ to a point in ${\bf r_j}$.  That is, modulo $\alpha$ and $\beta$ curves, which have linking number zero with the knot, $\pi_{*}\epsilon({\bf s},{\bf r}) = \epsilon({\bf r}_1,{\bf s}_1)+ \cdots +\epsilon({\bf r}_q, {\bf s}_q)$.  Notice also that $\operatorname{lk}(\pi_{*}\epsilon({\bf s},{\bf r}), K) = \operatorname{lk} \big(\epsilon({\bf s},{\bf r}), \widetilde{K}\big)$, whereas $\operatorname{lk}(\pi_{*}\epsilon({\bf s},{\bf r}), U) = q \operatorname{lk} \big(\epsilon({\bf s},{\bf r}), \widetilde{U} \big)$. Therefore we compute
\begin{align*}
A_1({\bf r}) - A_1({\bf s}) &= \operatorname{lk}\left(\epsilon({\bf s},{\bf r}),\widetilde{K}\right)\\
 & =\operatorname{lk}(\pi_{*}(\epsilon({\bf s},{\bf r})), K) \\
						    &=\sum_{j=1}^{q}\operatorname{lk} \left(\epsilon\left({\bf s}_j,{\bf r}_j\right),K \right)\\
			                &=\sum A_1({\bf r}) - A_1({\bf s})
\end{align*}
and moreover
\begin{align*}
A_2({\bf r}) - A_2({\bf s}) &= \operatorname{lk}\left(\epsilon({\bf s},{\bf r}), \widetilde{U}\right)\\
							&=\frac{1}{q}\operatorname{lk}\left(\pi_{*}(\epsilon({\bf s},{\bf r})), U\right)\\
							&=\frac{1}{q}\left(\sum_{j=1}^{q}\left(\operatorname{lk}(\epsilon({\bf s}_j,{\bf r}_j),U 											\right)\right) \\
							&=\frac{1}{q}\left(\sum_{j=1}^{q}\left( A_2({\bf r}_j) - A_2({\bf s}_j)\right)\right)\\[-3.5em]
\end{align*}
\end{proof}

\enlargethispage{2em}

We may now give a proof of Theorem \ref{Murasugi Theorem}.

\begin{proof}[Proof of Theorem \ref{Murasugi Theorem}]

Let $\widetilde{K}$ be a periodic knot with period $q=p^r$ for some prime $p$, and $K$ be its quotient knot, and let $\lambda = \operatorname{lk}(\widetilde{K},\widetilde{U}) = \operatorname{lk}(K,U)$.  Choose a periodic diagram $\widetilde{D}$ for $(S^3, \widetilde{K} \cup \widetilde{U})$ and its quotient diagram $D$ for $(S^3, K \cup U)$ as outlined above.

Consider the Euler characteristic of $\widehat{\mathit{CFL}}(\widetilde{\mathcal D})$ computed modulo $p$.  Let ${\bf s}$ be a generator in  $\widehat{\mathit{CFL}}(\widetilde{\mathcal D})$. Either ${\bf s} = \widetilde{{\bf y}} = \pi^{-1}(\bf y)$ for some ${\bf y} \in \widehat{\mathit{CFL}}(\mathcal D)$, and thus ${\bf s}$ is invariant under the action of $\tau$, or the order of the orbit of ${\bf s}$ under the action of $\tau$ is a multiple of $p$.  Since the action preserves the Alexander and Maslov gradings, modulo $p$ the terms of the Euler characteristic of $\widehat{\mathit{CFL}}(\widetilde{\mathcal D})$ corresponding to noninvariant generators sum to zero.  Moreover, there is a one-to-one correspondence between generators ${\bf y}$ of $\widehat{\mathit{CFL}}(\mathcal D)$ and their total lifts $\pi^{-1}({\bf y}) = {\bf s}$ in $\widehat{\mathit{CFL}}(\widetilde{\mathcal D})$.

This correspondence preserves relative Alexander $A_2$-gradings, and multiplies Alexander $A_1$-gradings by a factor of $q$.  We also claim that it preserves the parity of relative Maslov gradings if $p$ is odd.  In particular, any two generators ${\bf x}$ and ${\bf y}$ of $\widehat{\mathit{CFL}({\mathcal D})}$ are joined by a domain $D$ which does not pass over $w_0$. Let $\widetilde{ D}$ be the lift of this domain. Then the Maslov index $\mu(\widetilde{D})$ is equal to $q\mu(D) - (q-1)n_{z_0}(D)$ by Lemma \ref{Maslov Index Lemma}, and we have
\begin{align*}
M(\widetilde{{\bf x}}) - M(\widetilde{{\bf y}}) &= \mu(\widetilde{D}) - 2\sum_{i=1}^{n} \left(n_{w_i^1}(\widetilde{D}) + \cdots + n_{w_i^q}(\widetilde{D})\right) \\
&= q\mu(D) -(q-1) n_{z_0}(D) - 2q\sum_{i=1}^n n_{w_i}(D) \\[-1ex]
&\equiv \mu(D) - 2\sum_{i=1}^n n_{w_i}(D) \text{ mod } 2\\
&\equiv M({\bf x}) - M({\bf y}) \text{ mod } 2
\end{align*}

Here we have used the assumption that $p$, and therefore $q$, is odd, and $q-1$ is even.  Ergo if $p$ is odd, $(-1)^{M(\widetilde{{\bf x}})} = \pm (-1)^{M({\bf x})}$ where the choice of sign is the same for all generators ${\bf x}$ of $\widehat{\mathit{CFL}}(\mathcal D)$.  (If $p=2$ the sign is of course immaterial.)

Therefore, we capture the following equality.
\begin{align*}
\chi(\widehat{\mathit{CFL}}(\widetilde{\mathcal D})) (t_1, t_2) \doteq \chi(\widehat{\mathit{CFL}}(\widetilde{\mathcal D}) )(t_1^q, t_2) \text{ mod } p
\end{align*}

\noindent Here $\doteq$ denotes equivalence up to an overall factor of $\pm t_1^{i_1}t_2^{i_2}$.

Notice that $\chi(\widehat{\mathit{CFL}}(\widetilde{\mathcal D})) = \chi(\widetilde{\mathit{HFL}}(\widetilde{\mathcal D}))$ and similarly $\chi(\widehat{\mathit{CFL}}(\mathcal D)) =\linebreak \chi(\widetilde{\mathit{HFL}}(\mathcal D))$.  Hence we may express the Euler characteristics of the chain complexes as the multivariable Alexander polynomials of $\widetilde{L} = \widetilde{K} \cup \widetilde{U}$ and $L = K \cup U$ multiplied by appropriate powers of $(1- t_1^{-1})$ and $(1-t_2^{-1})$ according to the number of basepoints on each component of each link. Therefore the equality above reduces to 
\begin{align*}
\Delta_{\widetilde{L}}(t_1, t_2)(1-t_1^{-1})^{qn_1}(1-t_2^{-1}) &\doteq \Delta_{L}(t_1^q, t_2)(1-(t_1^{-1})^{q})^{n_1}(1-t_2^{-1}) \text{ mod } p\\
\Delta_{\widetilde{L}}(t_1, t_2)(1-t_1^{-1})^{qn_1} &\doteq \Delta_{L}(t_1^q, t_2)(1-(t_1^{-1})^{q})^{n_1} \text{ mod } p.
\end{align*}
\noindent Recalling that $q$ is a power of $p$, and that therefore $(a+b)^q \equiv a^q + b^q$ mod $p$, we may reduce farther.
\begin{align*}
\Delta_{\widetilde{L}}(t_1,t_2)(1-t_1^{-1})^{qn_1} &\doteq \Delta_{L}(t_1^q, t_2)(1-t_1^{-1})^{qn_1} \text{ mod } p\\
\Delta_{\widetilde{L}}(t_1,t_2) &\doteq \Delta_{L}(t_1^q, t_2) \text{ mod } p
\end{align*}
We now set $t_2 = 1$.  By Lemma \ref{Two Component Link Lemma}, this reduces the equality above to
\begin{align*}
\Delta_{\widetilde{K}}(t_1)(1 \!+\! t_1 \!+\! \cdots \!+\! t_1^{\lambda-1}) &\doteq \Delta_{K}(t_1^q)(1 \!+\! t_1^q \!+\! (t_1^q)^2 \!+\! \cdots \!+\! ((t_1)^q)^{\lambda-1} \text{ mod } p.
\end{align*}
\noindent Again using the fact that $q$ is a power of $p$, we produce
\begin{align*}
\Delta_{\widetilde{K}}(t_1)(1 + t_1 + \cdots + t_1^{\lambda-1}) &\doteq (\Delta_{K}(t_1))^q(1 + t_1 + t_1^2 + \cdots + t_1^{(\lambda-1)})^q \text{ mod } p \\
\Delta_{\widetilde{K}}(t_1) &\doteq (\Delta_K(t_1))^q (1 + t + \cdots + t^{\lambda -1})^{q-1} \text{ mod } p.
\end{align*}
\noindent This last is Murasugi's condition.\end{proof}

\subsection{Spectral sequences for doubly-periodic knots}

From now on we restrict ourselves entirely to the case of $\widetilde{K}$ a doubly-periodic knot.  Moreover, we insist that $\widetilde{K}$ be oriented such that $\operatorname{lk}(\widetilde{K},\widetilde{U}) = \lambda$ is positive.  Recall that $\lambda$ is necessarily odd; otherwise $\widetilde{K}$ would be disconnected. We proceed to explain how Corollary \ref{Genus Inequality Corollary} follows from Theorem \ref{Knot Floer Homology Spectral Sequence}.  Consider the map $\tau_{\#}\co \widehat{\mathit{CFL}}(\widetilde{\mathcal D}) \rightarrow \widehat{\mathit{CFL}}(\widetilde{\mathcal D})$ induced by the involution $\tau$ on $\widetilde{\mathcal D}$.  As a consequence of our application of Seidel and Smith's localization theory to the symmetric products $\operatorname{Sym}^{2n_1}(S^2 \backslash (\widetilde{\bf w} \cup \widetilde{\bf z}))$ and $\operatorname{Sym}^{n_1}(S^2 \backslash ({\bf w} \cup {\bf z}))$, we will replace $\widehat{\mathit{CFL}}(\widetilde{\mathcal D})$ with a chain homotopy equivalent complex $\widehat{\mathit{CFL}}(\widetilde{\mathcal D})'$ with the same invariant generators and $\tau_{\#}$ with an involution $\tau_{\#}'$, not necessarily chain homotopy equivalent, which has the same fixed set as $\tau_{\#}$ and is a chain map. (For more on this replacement, see Section \ref{Floer Cohomology Section}). This map $\tau_{\#}'$ also preserves Alexander gradings. Consider the double complex
$$
\xymatrix{
 & \ar[d]^-{\partial'} & \ar[d]^-{\partial'} & \ar[d]^-{\partial'} & \\
0 \ar[r] & \widehat{\mathit{CFL}}_{i+1}(\widetilde{\mathcal D})' \ar[r]^-{1 + \tau_{\#}'} \ar[d]^-{\partial'} & \widehat{\mathit{CFL}}_{i+1}(\widetilde{\mathcal D})' \ar[r]^-{1 + \tau_{\#}'} \ar[d]^-{\partial'}& \widehat{\mathit{CFL}}_{i+1}(\widetilde{\mathcal D})' \ar[r]^-{1 + \tau_{\#}'} \ar[d]^-{\partial'} &\\
0 \ar[r] & \widehat{\mathit{CFL}}_{i}(\widetilde{\mathcal D})' \ar[r]^-{1 + \tau_{\#}'} \ar[d]^-{\partial'} & \widehat{\mathit{CFL}}_{i}(\widetilde{\mathcal D})' \ar[r]^-{1 + \tau_{\#}'} \ar[d]^-{\partial'}& \widehat{\mathit{CFL}}_{i}(\widetilde{\mathcal D})' \ar[r]^-{1 + \tau_{\#}'} \ar[d]^-{\partial'} &\\
0 \ar[r] & \widehat{\mathit{CFL}}_{i-1}(\widetilde{\mathcal D})' \ar[r]^-{1 + \tau_{\#}'} \ar[d]^-{\partial'} & \widehat{\mathit{CFL}}_{i-1}(\widetilde{\mathcal D}) \ar[r]^-{1 + \tau_{\#}'} \ar[d]^-{\partial'}& \widehat{\mathit{CFL}}_{i-1}(\widetilde{\mathcal D})' \ar[r]^-{1 + \tau_{\#}'} \ar[d]^-{\partial'}& \\
 & & & &
}
$$

\begin{definition}
The homology of the complex $(\widehat{\mathit{CFL}}(\widetilde{\mathcal D})' \otimes \mathbb Z_2[[\theta]], d +\linebreak \theta(1+ \tau_{\#}'))$ is $\widehat{\mathit{HFL}}_{\operatorname{Borel}}(\widetilde{\mathcal D})$.
\end{definition}

Computing vertical differentials first, we obtain a spectral sequence from $\widetilde{\mathit{HFL}}(\widetilde{\mathcal D}) \otimes \mathbb Z_2[[\theta]]$ to $\widehat{\mathit{HFL}}_{\operatorname{Borel}}(\widetilde{\mathcal D})$. Moreover, our application of Seidel and Smith's localization theory will lead us to the following theorem.

\begin{theorem}
There is a localization map
\begin{align*}
\widehat{\mathit{HFL}}_{Borel}(\widetilde{\mathcal D})\rightarrow \widetilde{\mathit{HFL}}(\mathcal D) \otimes Z_2[[\theta]]
\end{align*}
which becomes an isomorphism after tensoring with $\mathbb Z_2((\theta))$.
\end{theorem}

Therefore after tensoring the spectral sequence from $\widetilde{\mathit{HFL}}(\widetilde{\mathcal D}) \otimes \mathbb Z_2[[\theta]]$ to $\widehat{\mathit{HFL}}_{\operatorname{Borel}}(\widetilde{\mathcal D})$ with $\mathbb Z_2((\theta))$, we obtain the spectral sequence of Theorem~\ref{Link Floer Homology Spectral Sequence}. The proof reduces to finding a \textit{stable normal trivialization} for the triple $(\operatorname{Sym}^{2n_1}(\Sigma(S)\backslash \widetilde{\bf w}), \mathbb T_{\widetilde{\boldsymbol \alpha}}, \mathbb T_{\widetilde{\boldsymbol \beta}})$; see Sections 4--6 for details. Since  $\partial$ and $\tau_{\#}'$ also preserve both Alexander gradings on $\widehat{\mathit{CFL}}(\widetilde{D})$, the spectral sequence splits along the grading $(A_1,A_2)$.

The knot Floer homology spectral sequence of Theorem \ref{Knot Floer Homology Spectral Sequence} arises from a similar double complex.  Recall that for a link $L=K_1\cup\cdots\cup K_\ell$, the differential $\partial_{K_j}$ corresponds to forgetting the component $K_j$ of the link.  Therefore for the link $\widetilde{L}=\widetilde{K}\cup \widetilde{U}$ we begin with the chain complex $(\widehat{\mathit{CFK}}(\widetilde{\mathcal D}), \partial_{\widetilde{U}})$ and again replace with a chain homotopy equivalent complex $(\widehat{\mathit{CFK}}(\widetilde{\mathcal D})'', \partial'_{\widetilde{U}})$ and a possibly not chain homotopy equivalent involution $\tau_{\#}''$ which is a chain map, preserves the Alexander grading $A_1$, and has the same fixed set as $\tau_{\#}$. Consider the following double complex.
$$
\xymatrix{
 & \ar[d]^-{\partial'_{\widetilde{U}}} & \ar[d]^-{\partial'_{\widetilde{U}}} & \ar[d]^-{\partial'_{\widetilde{U}}} & \\
0 \ar[r] & \widehat{\mathit{CFK}}_{i+1}(\widetilde{\mathcal D})'' \ar[r]^-{1 + \tau_{\#}''} \ar[d]^-{\partial'_{\widetilde{U}}} & \widehat{\mathit{CFK}}_{i+1}(\widetilde{\mathcal D})'' \ar[r]^-{1 + \tau_{\#}''} \ar[d]^-{\partial'_{\widetilde{U}}}& \widehat{\mathit{CFK}}_{i+1}(\widetilde{\mathcal D})'' \ar[r]^-{1 + \tau_{\#}''} \ar[d]^-{\partial'_{\widetilde{U}}} &\\
0 \ar[r] & \widehat{\mathit{CFK}}_{i}(\widetilde{\mathcal D})'' \ar[r]^-{1 + \tau_{\#}''} \ar[d]^-{\partial'_{\widetilde{U}}} & \widehat{\mathit{CFK}}_{i}(\widetilde{\mathcal D})'' \ar[r]^-{1 + \tau_{\#}''} \ar[d]^-{\partial'_{\widetilde{U}}}& \widehat{\mathit{CFK}}_{i}(\widetilde{\mathcal D})'' \ar[r]^-{1 + \tau_{\#}''} \ar[d]^-{\partial'_{\widetilde{U}}} &\\
0 \ar[r] & \widehat{\mathit{CFK}}_{i-1}(\widetilde{\mathcal D})'' \ar[r]^-{1 + \tau_{\#}''} \ar[d]^-{\partial'_{\widetilde{U}}} & \widehat{\mathit{CFK}}_{i-1}(\widetilde{\mathcal D})'' \ar[r]^-{1 + \tau_{\#}''} \ar[d]^-{\partial'_{\widetilde{U}}}& \widehat{\mathit{CFK}}_{i-1}(\widetilde{\mathcal D})'' \ar[r]^-{1 + \tau_{\#}''} \ar[d]^-{\partial'_{\widetilde{U}}}& \\
 & & & &
}
$$

\begin{definition}
The homology of the complex $(\widehat{\mathit{CFK}}(\mathcal D)'' \otimes \mathbb Z_2[[\theta]], \partial_{\widetilde{U}}' + \theta(1+ \tau_{\#}'))$ is $\widehat{\mathit{HFK}}_{\operatorname{Borel}}(\widetilde{\mathcal D})$.
\end{definition}

Computing vertical differentials first gives a spectral sequence from $E^1 = \widetilde{\mathit{HFK}}(\widetilde{\mathcal D}) \otimes \mathbb Z_2[[\theta]]$ to $\widehat{\mathit{HFK}}_{\operatorname{Borel}}(\widetilde{\mathcal D})$. As before, our application of Seidel and Smith's localization theory will lead us to the following theorem.

\begin{theorem}
There is a localization map
\begin{align*}
\widehat{\mathit{HFK}}_{Borel}(\widetilde{\mathcal D})\rightarrow \widetilde{\mathit{HFK}}(\mathcal D) \otimes Z_2[[\theta]]
\end{align*}
which becomes an isomorphism after tensoring with $\mathbb Z_2((\theta))$.
\end{theorem}

Therefore after tensoring the spectral sequence from $\widetilde{\mathit{HFK}}(\widetilde{\mathcal D}) \otimes \mathbb Z_2[[\theta]]$ to $\widehat{\mathit{HFK}}_{\operatorname{Borel}}(\widetilde{\mathcal D})$ with $\mathbb Z_2((\theta))$, we obtain the spectral sequence of \ref{Knot Floer Homology Spectral Sequence}. As $\partial'_{\widetilde{U}}$ does not preserve Alexander $A_2$ gradings relative to the axis $U$, neither does the spectral sequence; however, it still splits along Alexander $A_1$ gradings, the grading relative to the knot itself.

Let us now complete the proofs of Theorems \ref{Link Floer Homology Spectral Sequence} and \ref{Knot Floer Homology Spectral Sequence}, and Corollaries \ref{Alexander Gradings Knot Lemma} and \ref{Alexander Gradings Link Lemma}, by fixing the relationship between the absolute Alexander gradings of $\widetilde{\mathit{HFL}}(\widetilde{\mathcal D})$ and $\widetilde{\mathit{HFL}}(\mathcal D)$.

For ${\bf x}$ any intersection point in $\mathbb T_{\boldsymbol \alpha} \cap \mathbb T_{\boldsymbol \beta}$, let $\underline{\mathfrak s}_{{\bf w,z}}({\bf x})$ be the relative $\operatorname{spin}^{\operatorname{c}}$ structure on $(S^3-\nu(L), \partial(S^3-\nu(L)))$ associated to ${\bf x}$ as in Section 2. Then since the map $\pi\co S^3 - \nu(\widetilde{L}) \rightarrow S^3-\nu(L)$ is a local diffeomorphism, we can pull back relative $\operatorname{spin}^{\operatorname{c}}$ structures along this map, and examining the construction in \cite[Section 3.6]{MR2443092} we see that the pullback of $\underline{\mathfrak s}_{{\bf w,z}}({\bf x})$ is the relative $\operatorname{spin}^{\operatorname{c}}$ structure on $(S^3 - \nu(\widetilde{L}), \partial(S^3 - \nu(\widetilde{L}) ))$ associated to $\widetilde{\bf x}$ the lift of $\bf{x}$. In particular, taking the first Chern class of both relative $\operatorname{spin}^{\operatorname{c}}$ structures, we conclude that $\pi^*(c_1(\underline{\mathfrak s}_{{\bf w,z}}({\bf x}))) = c_1(\underline{\mathfrak s}_{{\bf w,z}}(\widetilde{\bf x}))$.

Now let $\gamma$ be an arbitrary element of $H_2(S^3-\nu(\widetilde{L}), \partial(S^3-\nu(\widetilde{L})))\simeq H_2(S^3, \widetilde{L})$. As before, let this group be generated by the classes dual to $[\mu_{\widetilde{K}}]$ and $[\mu_{\widetilde{U}}]$ in $H_1(L)$, and called by the same names via our previous abuse of notation, and let $\gamma = c[\mu_{\widetilde{K}}] + d[\mu_{\widetilde{U}}]$ for integers $c,d$. Moreover, let $H_2(S^3-\nu(L), \partial(S^3-\nu(L)))\simeq H_2(S^3, L)$ be similarly generated by classes $[\mu_K]$ and $[\mu_U]$, and let $H^2(S^3-\nu(L), \partial(S^3-\nu(L)))\simeq H^2(S^3, L)$ be generated by their duals $[\mu_K]^*$ and $[\mu_U]^*$. Similarly, we have $H^2(S^3-\nu(\widetilde{L}), \partial(S^3-\nu(\widetilde{L})))\simeq H^2(S^3, \widetilde{L})$ generated by the cohomology classes $[\mu_{\widetilde{K}}]^*$ and $[\mu_{\widetilde{L}}]^*$. Finally, suppose that $c_1(\underline{\mathfrak s}_{{\bf w,z}}({\bf x})) = e[\mu_K]^* + f[\mu_U]^*$. Then we compute:
\begin{align*}
\langle c_1(\underline{\mathfrak s}_{{\bf w,z}}(\widetilde{\bf x})), \gamma \rangle &= \langle \pi^*(c_1(\underline{\mathfrak s}_{{\bf w,z}}({\bf x}))), \gamma \rangle 
				= \langle c_1(\underline{\mathfrak s}_{{\bf w,z}}({\bf x})), \pi_*(\gamma) \rangle \\
				&= \langle c_1(\underline{\mathfrak s}_{{\bf w,z}}({\bf x})), 2c[\mu_K] + d[\mu_U] \rangle \\
				&= 2ec+fd
\end{align*}
Since the equation above holds for all $c,d \in \mathbb Z$, we conclude that if $c_1(\underline{\mathfrak s}_{{\bf w,z}}({\bf x}))\linebreak = e[\mu_K]^* + f[\mu_U]^*$, we must then have $c_1(\underline{\mathfrak s}_{{\bf w,z}}(\widetilde{\bf x})) = 2e[\widetilde{\mu_K}]^* + f[\widetilde{\mu_U}]^*$. This means that upstairs we have
\begin{align*}
c_1(\underline{\mathfrak s}_{{\bf w,z}}(\widetilde{\bf x})) + PD([\mu_{\widetilde{K}}]+ [\mu_{\widetilde{U}}]) &= 2PD(A_1[\mu_{\widetilde{K}}] + A_2[\mu_{\widetilde{U}}]) \\
(2e+1)[\mu_{\widetilde{K}}]^* + (f+1)[\mu_{\widetilde{U}}]^* &= 2A_1[\mu_{\widetilde{K}}]^* + 2A_2[\mu_{\widetilde{U}}]^*
\end{align*}
whereas downstairs we have 
\begin{align*}
c_1(\underline{\mathfrak s}_{{\bf w,z}}({\bf x}) + PD([\mu_{\widetilde{K}}]+ [\mu_{\widetilde{U}}]) &= 2PD(A_1[\mu_{\widetilde{K}}] + A_2[\mu_{\widetilde{U}}]) \\
(e+1)[\mu_{\widetilde{K}}]^* + (f+1)[\mu_{\widetilde{U}}]^* &= 2A_1[\mu_{\widetilde{K}}]^* + 2A_2[\mu_{\widetilde{U}}]^*
\end{align*}

This justifies the assertion that the spectral sequence carries the gradings $(A_1,A_2) = (2a+\frac{1}{2}, b)$ upstairs to $(A_1,A_2) = (a+ \frac{1}{2},b)$ downstairs, where $a=e$ and $b= \frac{f+1}{2}$.

There is an alternate proof, more complicated but also more helpful to our geometric intuition, using the link Floer homology categorification of the Thurston norm. Let $x_L([\mu_K])$ be the Thurston seminorm of the class dual to $[\mu_K]$ in $H_2(S^3,L)$ and $x_{\widetilde{L}}([\mu_{\widetilde{K}}])$ be the Thurston seminorm of the class dual to $[\mu_{\widetilde{K}}]$ in $H_2(S^3,L)$.  Notice that if $F$ is a Thurston-norm minimizing surface for the class dual to $[\mu_K]$ of Euler characteristic $\chi(F) = -x_L([\mu_K])$, then the preimage $\widetilde{F} = \pi^{-1}(F)$ under the ordinary double cover $\pi\co S^3-\nu(\widetilde{L}) \rightarrow S^3-\nu(L)$ is an embedded surface representing the class dual to $[\mu_{\widetilde{K}}]$ in $H_2(S^3,L)$, and $\chi(\widetilde{F}) = 2\chi(F)$.  Hence $x_{\widetilde{L}}([\mu_{\widetilde{K}}]) \leq 2x_L([\mu_K])$.

However, recall that $x_{L}([\mu_K]) + 1$ is exactly the breadth of the Alexander $A_1$ grading in $\widehat{\mathit{HFL}}(S^3, L)$, and therefore that the breadth of the $A_1$ grading in $\widetilde{\mathit{HFL}}(\mathcal D)$ is $x_{L}([\mu_K]) + 1 + (n_1 -1) = x_{L}([\mu_K]) + n_1$.  Similarly, the breadth of the $A_1$ grading in $\widehat{\mathit{HFL}}(S^3, \widetilde{L})$ is $x_{\widetilde{L}}([\mu_K]) +1$, and therefore the total breadth of the $A_1$ grading in $\widetilde{\mathit{HFL}}(\widetilde{\mathcal D})$ is $x_{\widetilde{L}}([\mu_K]) + 1 + 2n_1 - 1 = x_{\widetilde{L}}([\mu_{\widetilde{K}}]) + 2n_1$.  Moreover, we have seen that in the spectral sequence from $E^1 = \widetilde{\mathit{HFL}}(\widetilde{\mathcal D}) \otimes \mathbb Z_2((\theta))$ to $E^{\infty} \cong \widetilde{\mathit{HFL}}(\mathcal D) \otimes \mathbb Z_2((\theta))$, the relative $A_1$ grading of two elements on the $E^1$ page is twice the relative $A_1$ grading of their residues on the $E^{\infty}$ page.  Therefore the total breadth of the $A_1$ grading on the $E^1$ page of the spectral sequence is at least twice the breadth of the $A_1$ grading on the last page of the spectral sequence. We thus have the inequality
\begin{align*}
x_{\widetilde{L}}([\mu_{\widetilde{K}}]) + 2n_1 &\geq 2x_{L}([\mu_K]) + 2n_1 \\
										&\geq 2(x_L([\mu_K]) + n_1)\\
										\text{i.e. } x_{\widetilde{L}}([\mu_{\widetilde{K}}])&\geq 2x_L([\mu_K]).
\end{align*}
Consequently we see directly from the spectral sequence that
\begin{align*}
x_{\widetilde{L}}([\mu_{\widetilde{K}}]) = 2 x_L([\mu_K]).
\end{align*}
\noindent This implies that the breadth of the $A_1$ grading on the $E^1$ page is exactly twice the breadth of the $A_1$ grading on the $E^{\infty}$ page.  Therefore the breadth of the $A_1$ grading cannot decrease over the course of the spectral sequence.  A similar argument, sans the factors of two, shows that $x_{\widetilde L}([\mu_{\widetilde{U}}]) = x_{L}([\mu_U])$ and that the breadths of the Alexander $A_2$ grading of the $E^1$ and $E^{\infty}$ pages of the spectral sequence are the same.  Therefore the breadth of the $A_2$ grading does not change either over the course of the spectral sequence.

In particular, the top $A_1$ grading in $\widetilde{\mathit{HFL}}(\widetilde{\mathcal D})$ is sent to the top $A_1$ grading in $\widetilde{\mathit{HFL}}(\mathcal D)$.  However, by the symmetry of $\widehat{\mathit{HFL}}$ and the determination of the breadth of the $A_1$ grading by the Thurston norm, the top $A_1$ grading of $\widetilde{\mathit{HFL}}(\widetilde{\mathcal D})$ is the same as the top $A_1$ grading of $\widehat{\mathit{HFL}}(S^3, \widetilde{L})$, which is $A_1 = \frac{x_{\widetilde{L}}([\mu_{\widetilde{K}}]) +1}{2} = \frac{2x_L(\mu_K) +1}{2}$.  Similarly, the top $A_1$ grading of $\widetilde{\mathit{HFL}}(\mathcal D)$ is the top $A_1$ grading of $\widehat{\mathit{HFL}}(S^3, L)$, to wit, $A_1 = \frac{x_L([\mu_K]) + 1}{2}$.  Therefore the spectral sequence carries the $A_1$ Alexander grading $\frac{2x_L([\mu_K]) + 1}{2}$ on the $E^1$ page to the $A_1$ grading $\frac{x_L([\mu_K]) + 1}{2}$ on the $E^{\infty}$ page. Therefore in general the $A_1$ grading $\frac{2x_L([\mu_K]) + 1}{2} + 2a$ on the $E^1$ page is sent to the $A_1$ grading $\frac{x_L([\mu_K]) +1}{2} + a$ for any integer $a$.  Notice that $x_L([\mu_K])$ is an even number: suppose $F$ is a Thurston-seminorm minimizing surface for $K$ in $S^3-L$ of genus $g'$ with geometric intersection number $\#(F\cap U) =\Lambda$.  Since the algebraic intersection number $\lambda$ is odd, so is $\Lambda$.  Then $x_L([\mu_K]) = 1-2g'-\Lambda$ is even.  Therefore $\frac{x_L([\mu_K])}{2}$ is an integer, and we can take $a= \frac{x_L([\mu_K])}{2}$ to see that the $A_1$ grading $\frac{1}{2}$ on the $E^1$ page is sent to the $A_1$ grading $\frac{1}{2}$ on the $E^{\infty}$ page.

A parallel but simpler argument for the $A_2$ gradings shows that the $A_2$ grading $b$ on the $E^1$ page is sent precisely to the $A_2$ grading $b$ on the $E^{\infty}$ page, using the fact that relative $A_2$ gradings of elements on the $E^1$ page that survive in the spectral sequence are preserved rather than doubled on the $E^{\infty}$ page.

By Lemma \ref{Gradings Shift Lemma}, computing the knot Floer homology complex using $\partial_{\widetilde{U}}$ yields a downward shift in Alexander gradings by $\frac{\operatorname{lk}(\widetilde{K}, \widetilde{U})}{2}$ on the $E^1$ page and $\frac{\operatorname{lk}(K, U)}{2}$ on the $E^{\infty}$ page.  Since both of these numbers are $\frac{\lambda}{2}$, we obtain an overall downward shift of $\frac{\lambda}{2}$ between that link and knot Floer homology spectral sequences.

This nearly completes the proofs of Theorems \ref{Link Floer Homology Spectral Sequence} and \ref{Knot Floer Homology Spectral Sequence}. Corollaries~\ref{Alexander Gradings Link Lemma} and \ref{Alexander Gradings Knot Lemma} follow almost immediately, but the last statement in Corollary \ref{Alexander Gradings Knot Lemma} deserves a few more words. If Edmonds' condition is sharp, then $g(\widetilde{K})  = 2g(K) + \frac{\lambda-1}{2}$. Since the knot Floer spectral sequence sends the Alexander grading $2a + \frac{1-\lambda}{2}$ to the Alexander grading $a+\frac{1-\lambda}{2}$, it sends the Alexander grading $g(\widetilde{K})$ to $g(K)$. Therefore there is a rank inequality between $\widetilde{\mathit{HFK}}(\widetilde{\mathcal D}, g(\widetilde{K}))$ to $\widetilde{\mathit{HFK}}(\mathcal D, g(K))$. However, recall that the summands of the vector space $V$ carry Alexander gradings $0$ and $1$, so $\widetilde{\mathit{HFK}}(\widetilde{\mathcal D}, g(\widetilde{K})) \simeq \widehat{\mathit{HFK}}(S^3,\widetilde{K}, g(\widetilde{K}))\otimes W$ and $\widetilde{\mathit{HFK}}(\mathcal D, g(K)) \simeq \widehat{\mathit{HFK}}(S^3,K, g(K))\otimes W$. So ignoring a consistent factor of two from $W$, we obtain the last inequality in Corollary \ref{Alexander Gradings Knot Lemma}.

\begin{remark} The observations that $x_{\widetilde{L}}([\mu_{\widetilde{K}}]) = 2 x_L([\mu_K])$ and $x_{\widetilde L}([\mu_{\widetilde{U}}])\linebreak = x_{L}([\mu_U])$ are a special case of Gabai's theorem \cite[Corollary 6.13]{MR723813} that the Thurston norm is multiplicative for ordinary finite covers.  Indeed, by appealing to Gabai's theorem (or by constructing a $\mathbb Z_p$ analog of Seidel and Smith's localization spectral sequence) we could similarly fix the relationship between the absolute gradings of $\widetilde{\mathit{HFL}}(\widetilde{\mathcal D})$ and $\widetilde{\mathit{HFL}}(\mathcal D)$ for $p$-periodic knots.
\end{remark}

\begin{remark}

This alternate proof illustrates the important role of the vector space $V_1$ in the existence of the spectral sequence; the breadth of the $A_1$ grading in $\widehat{\mathit{HFL}}(S^3, \widetilde{L})$ is one less than twice the breadth of the $A_1$ grading in $\widehat{\mathit{HFL}}(S^3, L)$, a trouble which is corrected for by increasing the breadth upstairs by $2n_1 - 1$ and the breadth downstairs by $n_1 -1$.  While one might hope to produce a spectral sequence in which $n_1 =1$, it seems impossible to produce a link Floer homology spectral sequence for doubly periodic knots which does not involve at least one copy of $V_1$ on the $E^1$ page.

\end{remark}

Having fixed the Alexander gradings in the spectral sequence, we may provide a proof of Corollary \ref{Genus Inequality Corollary} (Edmonds' Condition) from \ref{Knot Floer Homology Spectral Sequence}.

\begin{proof} [Proof of Corollary \ref{Genus Inequality Corollary}]

Once again, let $V_1$ denote a two-dimensional vector space over $\mathbb Z_2$ whose two sets of gradings $(M, (A_1, A_2))$ are $(0,(0,0))$ and $(-1, (-1, 0))$, and likewise let $W$ be a two-dimensional vector space over $\mathbb Z_2$ whose two sets of gradings $(M, (A_1,A_2))$ are $(0,(0,0))$ and $(-1,(0, 0))$.  By Theorem \ref{Knot Floer Homology Spectral Sequence}, there is a spectral sequence whose $E^1$ page is $\widehat{\mathit{HFK}}(S^3, \widetilde{K}) \otimes V_1^{\otimes n_1-1} \otimes W \otimes \mathbb Z_2((\theta))$ to a theory whose $E^{\infty}$ page is $\mathbb Z((\theta))$-isomorphic to $\widehat{\mathit{HFK}}(S^3, K) \otimes V_1^{\otimes 2n_1-1} \otimes W \otimes \mathbb Z((\theta))$.  Moreover, this spectral sequence splits along the $A_1$ grading, and by Theorem \ref{Alexander Gradings Knot Lemma} the subgroup of the $E^1$ page in $A_1$ grading $\frac{1-\lambda}{2} + 2a$ is carried to the subgroup of the $E^{\infty}$ page in grading $\frac{1-\lambda}{2} +a$.  Moreover, the top $A_1$ grading on the $E^1$ page is $g(\widetilde{K})$ and the top $A_1$ grading on the $E^{\infty}$ page is $g(K)$.  Since there must be something on the $E^1$ page in the $A_1$ grading which converges to the $A_1$ grading $g(K)$ on the $E^{\infty}$ page, we have the following inequality.
\begin{align*}
g(\widetilde{K}) - \frac{1-\lambda}{2} &\geq 2\left(g(K) - \frac{1 - \lambda}{2}\right) \\
\text{i.e. } g(\widetilde{K}) &\geq 2g(K) + \frac{\lambda - 1}{2} \\[-3.5em]
\end{align*}
\end{proof}

Finally, let us prove Corollary \ref{Fiberedness Corollary}.

\begin{proof}[Proof of Corollary \ref{Fiberedness Corollary}]

Suppose Edmonds' condition is sharp, that is, that $g(\widetilde{K}) = 2g(K) + \frac{\lambda - 1}{2}$. As in the proof of Corollary \ref{Alexander Gradings Knot Lemma}, this implies that the knot Floer spectral sequence sends the Alexander grading $g(\widetilde{K}) = 2g(K) + \frac{\lambda -1}{2}$ to the Alexander grading $g(K)$.  That is, sharpness of Edmonds' condition exactly says that the top Alexander grading on the $E^1$ page is not killed in the spectral sequence.

Suppose now that $\widetilde{K}$ is fibered.  Then the top Alexander grading of $\widetilde{\mathit{HFK}}(\widetilde{\mathcal D}) \otimes \mathbb Z_2((\theta)) = (\widehat{\mathit{HFK}}(S^3, K) \otimes V_1^{\otimes (2n_1-1)}\otimes W) \otimes \mathbb Z_2((\theta))$ has rank\linebreak two as a $\mathbb Z_2((\theta))$ module.  (The knot Floer homology of fibered knots is monic in the top Alexander grading by the forward direction of  Lemma \ref{Fiberedness Condition Lemma}, and the factor of $W$ doubles the number of entries in each Alexander grading.)  Since this Alexander grading is not killed in the spectral sequence, the top Alexander grading of $\widetilde{\mathit{HFK}}(\mathcal D) \otimes \mathbb Z_2((\theta))=\widehat{\mathit{HFK}}(S^3, K) \otimes V_1^{\otimes (n_1-1)} \otimes W) \otimes \mathbb Z_2((\theta))$ also has rank two as a $\mathbb Z_2((\theta))$-module.  Therefore $K$ is also fibered.\end{proof}

\begin{remark}

The converse of Corollary \ref{Fiberedness Corollary}, that when Edmonds' condition is sharp, the quotient knot $K$ being fibered implies $\widetilde{K}$ is fibered, is false.  Consider the following counterexample: the knot $\widetilde{K}=10_{144}$ is doubly periodic with quotient knot $K = 3_1$, the trefoil.  The linking number $\lambda = \operatorname{lk}(\widetilde{K}, \widetilde{U}) = 1$.  Since $g(10_{144}) = 2$ and $g(3_1)=1$, we see that $g(\widetilde{K}) = 2 = 2g(K) + \frac{\lambda-1}{2}$.  Therefore Edmonds' condition is sharp.  However, the trefoil is fibered, whereas $10_{144}$ is not \cite{KnotInfo}.

\end{remark}

\subsection{An example of an obstruction not given by Alexander polynomials and genera}

We now give an example of a knot $\widetilde{K}$ which is not obstructed from being two-periodic with a specific quotient knot $K$ by Edmonds' and Murasugi's conditions, taken together, but for which the spectral sequence of \ref{Knot Floer Homology Spectral Sequence} cannot exist. Indeed, we can do slightly better. There is a second condition of Murasugi (\textit{not} recovered by the spectral sequences of this paper), as follows.

\begin{theorem}\cite[Corollary 1]{MR0292060}\label{Second Condition}
If $\widetilde{K}$ is $q$-periodic with quotient knot $K$, then $\Delta_K(t)|\Delta_{\widetilde{K}}(t)$.
\end{theorem}

Our example will have $\Delta_K(t)=1$, and therefore will pass both of Murasugi's conditions.

Let $\tilde{K}$ be the connect sum of the Kinoshita-Terasaka knot $11\mathrm{n}42$ and the right-handed trefoil $3_1$. Then $g(\widetilde{K})= g(11\mathrm{n}42) + g(3_1) = 3$. Moreover since the Kinoshita-Terasaka knot has trivial Alexander polynomial, $\Delta_{\tilde{K}}(t) = \Delta_{3_1}(t) = 1 -t + t^2 \equiv (1 + t + t^2)(1)^2$ modulo two. Suppose that $\widetilde{K}$ is two-periodic. Then by Murasugi's condition, we must have $\lambda=3$. If $K$ is any genus one knot with trivial Alexander polynomial, then we see that $\Delta_{\tilde{K}}(t)\equiv(1+t+t^2)(\Delta(K))^2$ modulo two, $\Delta_K(t) | \Delta_{\widetilde{K}}(t)$, and finally $g(\widetilde{K}) = 2g(K) + \frac{\lambda-1}{2}$. Therefore neither Edmonds' nor either of Murasugi's conditions ob\-structs $\widetilde{K}$ from being two-periodic with quotient knot $K$.

Now, consider candidate quotient knot $K = D(4_1)$ the untwisted positive Whitehead double of the figure-eight knot, which is a genus one knot with trivial Alexander polynomial. Suppose $\widetilde{K}$ is two-periodic with quotient knot $K$. Since Edmonds' condition is sharp, by Corollary \ref{Alexander Gradings Link Lemma}, there is a rank inequality between $\widehat{\mathit{HFK}}(S^3, \widetilde{K}, 3)$ and $\widehat{\mathit{HFK}}(S^3, K, 1)$. Let us consider the ranks of these groups.


First, let us compute the rank of the group $\widetilde{\mathit{HFK}}(\widetilde{\mathcal D}, 3)$. Recall that the knot Floer homology of a connect sum of knots in the three-sphere obeys a Kunneth formula \cite[Theorem 7.1]{MR2065507}. Therefore the top-dimensional knot Floer homology of $\widetilde{K}$ is the tensor product of the top-dimensional knot Floer homologies of the trefoil and the Kinoshita-Terasaka knot. The knot Floer homology of the latter was first computed in the extremal gradings in \cite[Theorem 1.1]{MR2058681}, and later in full by Baldwin and Gillam \cite[Section 4]{MR2925428}. The fact from these computations that we will need is that $\operatorname{rk}\big(\widehat{\mathit{HFK}}(S^3,11\mathrm{n}42, 2)\big) = 2$. Therefore since $\operatorname{rk}\big(\widehat{\mathit{HFK}}(S^3,3_1, 1)\big) = 1$, we conclude that $\operatorname{rk}\big(\widehat{\mathit{HFK}}(S^3,\linebreak\widetilde{K}, 3)\big)=2$.
%

However, Hedden \cite[Proposition 7.1]{MR2372849} has computed the knot Floer homology of Whitehead doubles. Relevantly, from his work we know that $\operatorname{rk}\big(\widehat{\mathit{HFK}}(S^3, D(4_1),1)\big) =4$. Therefore there cannot be a spectral sequence sending $\widetilde{\mathit{HFK}}(\widetilde{\mathcal D}, 3)$ to $\widetilde{\mathit{HFK}}(\mathcal D, 1)$, implying that $\widetilde{K}$ cannot be two-periodic with quotient knot $K$.

\section{Spectral sequences for Lagrangian Floer cohomology} \label{Floer Cohomology Section}

Floer cohomology is an invariant for Lagrangian submanifolds in a symplectic manifold introduced by Floer \cite{MR965228, MR933228, MR948771}. In this section we briefly recall the setting and statement of Seidel and Smith's localization theorem for Floer cohomology. For further detail, we refer the reader to the longer exposition in \cite[Section 2]{Hendricks}.

Let $M$ be a manifold equipped with an exact symplectic form $\omega = d\theta$ and a compatible almost complex structure $J$ which is convex at infinity. Let $L_0$ and $L_1$ be two exact Lagrangian submanifolds of $M$. For our purposes we can restrict to the case that $L_0$ and $L_1$ are compact and intersect transversely. Let $\mathit{CF}(L_0,L_1) = \mathbb Z_2\langle L_0 \cap L_1 \rangle$ be the Floer cohomology chain group with differential $\partial$, such that $\mathit{HF}(L_0,L_1) = H_*(\mathit{CF}(L_0,L_1), \partial)$ is the Lagrangian Floer cohomology of $L_0$ and $L_1$ in $M$.

Now, suppose that $M$ carries a symplectic involution $\tau$ preserving $(M,\linebreak L_0, L_1)$ and the forms $\omega$ and $\theta$. Let the submanifold of $M$ fixed by $\tau$ be $M^{\operatorname{inv}}$, and similarly for $L_i^{\operatorname{inv}}$ for $i=0,1$. The Floer chain complex $\mathit{CF}(L_0, L_1)$ carries an induced involution $\tau_{\#}$ which takes $x \in L_0 \cap L_1$ to the intersection point $\tau(x) \in L_0 \cap L_1$. This map $\tau_{\#}$ is not a chain map with respect to a generic family of complex structures on $M$.  However, suppose that we are in the nice case that we can find a family of complex structures ${\bf J}$ on $M$ such that $\tau_{\#}$ commutes with the differential on $\mathit{CF}(L_0, L_1)$. Then $1 + \tau_{\#}$ is a second differential on $\mathit{CF}(L_0,L_1)$, and we can use the double complex below to define the Borel (or equivariant) cohomology of $(M, L_0, L_1)$ with respect to this involution.
$$
\xymatrix{
0 \rightarrow \mathit{CF}(L_0, L_1) \ar[r]^-{1 + \tau_{\#}} & \mathit{CF}(L_0, L_1) \ar[r]^-{1 + \tau_{\#}} & \mathit{CF}(L_0, L_1)\cdots
}
$$
\begin{definition}

If $\mathit{CF}(L_0, L_1)$ is the Floer chain complex and $\tau_{\#}$ is a chain map with respect to the complex structure on $M$, $\mathit{HF}_{\operatorname{Borel}}(L_0, L_1)$ is the homology of the complex $\mathit{CF}(L_0,L_1) \otimes \mathbb Z_2[[\theta]]$ with respect to the differential $\partial + (1 + \tau_{\#})\theta$. 

\end{definition}

As in the original, the choice of $\mathbb Z_2[[\theta]]$ instead of $\mathbb Z_2[\theta]$ is largely irrelevant since only finitely many powers of $\theta$ appear in each degree, but was chosen by Seidel and Smith to agree with more general contexts \cite[Section 2]{MR2739000}.

Let us now set up notation for Seidel and Smith's main definition and theorem. Consider the normal bundle $N(M^{\operatorname{inv}})$ to $M^{\operatorname{inv}}$ in $M$ and its Lagran\-gian subbundles $N(L_i^{\operatorname{inv}})$ the normal bundles to each $L_i^{\operatorname{inv}}$ in $L_i$. We pull back the bundle $N(M^{\operatorname{inv}})$ along the projection map $M^{\operatorname{inv}} \times [0,1] \rightarrow M^{\operatorname{inv}}$. Call this pullback $\Upsilon(M^{\operatorname{inv}})$. This bundle is constant with respect to the interval $[0,1]$. Its restriction to each $M^{\operatorname{inv}} \times \{t\}$ is a copy of $N(M^{\operatorname{inv}})$ which will occasionally, by a slight abuse of notation, be called $N(M^{\operatorname{inv}}) \times \{t\}$; similarly, for $i=0,1$ the copy of $N(L_i^{\operatorname{inv}})$ above $L_i^{\operatorname{inv}} \times \{t\}$ will be referred to as $N(L_i^{\operatorname{inv}}) \times \{t\}$.

We make a note here of the correspondence between our notation and Seidel and Smith's original usage. Our bundle $\Upsilon(M^{\operatorname{inv}})$ is their $TM^{\operatorname{anti}}$; while our $N(L_0^{\operatorname{inv}}) \times \{0\}$ is their $TL_0^{\operatorname{inv}}$ and our $N(L_1^{\operatorname{inv}}) \times \{1\}$ is their $TL_1^{\operatorname{anti}}$. (The name $TL_1^{\operatorname{anti}}$ is also used for the bundle that we denote $N(L_1^{\operatorname{inv}}) \times \{0\}$, using the obvious isomorphism between the bundles.)

\begin{definition} \label{stablenormaltriv} \cite[Defn 18]{MR2739000} A stable normal trivialization of the vector bundle $\Upsilon(M^{\operatorname{inv}})$ over $M^{\operatorname{inv}} \times [0,1]$ consists of the following data.

\begin{itemize}
\item A stable trivialization of unitary vector bundles $\phi \co \Upsilon(M^{\operatorname{inv}}) \oplus \mathbb C^{K} \rightarrow \mathbb C^{k_{\operatorname{anti}} + K}$ for some $K$.

\item A Lagrangian subbundle $\Lambda_0 \subset (\Upsilon(M^{\operatorname{inv}}))|_{[0,1] \times L^{\operatorname{inv}}_0}$ such that $\Lambda_0|_{\{0\} \times L^{\operatorname{inv}}_0}\linebreak = (N(L_0^{\operatorname{inv}})\times \{0\})\oplus \mathbb R^K$ and $\phi(\Lambda_0|_{\{1\} \times L_0^{\operatorname{inv}}}) = \mathbb R^{k_{\operatorname{anti}} + K}$.

\item A Lagrangian subbundle $\Lambda_1 \subset (\Upsilon(M^{\operatorname{inv}}))|_{[0,1] \times L^{\operatorname{inv}}_1}$ such that $\Lambda_1|_{\{0\} \times L^{\operatorname{inv}}_1}\linebreak = (N(L_1^{\operatorname{inv}})\times\{0\})\oplus \mathbb R^K$ and $\phi(\Lambda_1|_{\{1\} \times L_1^{\operatorname{inv}}}) = i\mathbb R^{k_{\operatorname{anti}} + K}$.
\end{itemize}

\end{definition}

The crucial theorem of \cite{MR2739000}, proved through extensive geometric analysis and comparison with the Morse theoretic case, is as follows.

\begin{theorem} \label{Localization} \cite[Thm 20]{MR2739000} If $\Upsilon(M^{\mathrm{inv}})$ carries a stable normal trivialization, then after an equivariant exact Lagrangian isotopy which replaces $L_0$ and $L_1$ with $\widetilde{L}_0$ and $\widetilde{L}_1$ and fixes the invariant sets,  $\mathit{HF}_{\mathrm{borel}}(\widetilde{L}_0,\widetilde{L}_1)$ is well-defined and there are localization maps
\[
\Delta^{(m)} \co \mathit{HF}_{\mathrm{borel}}(\widetilde{L}_0,\widetilde{L}_1) \rightarrow \mathit{HF}(L_0^{\mathrm{inv}}, L_1^{\mathrm{inv}})[[\theta]]
\]
defined for $m\gg 0$ and satisfying $\Delta^{(m+1)} = \theta\Delta^{(m)}$. Moreover, after tensoring over $\mathbb Z_2[[\theta]]$ with $\mathbb Z_2((\theta))$ these maps are isomorphisms. 

\end{theorem}

Let us say a few words about the appearance of the isotopy carrying $L_i$ to $\tilde{L}_i$ in Theorem \ref{Localization}, and explain how Theorem \ref{Localization} implies Theorem \ref{SeidelSmith}. One of the uses of the stable normal trivialization condition in the proof of Theorem \ref{Localization} is to construct both this isotopy and a family of complex structures ${\bf J}$ on $M$ with respect to which $\mathit{HF}_{\operatorname{Borel}}(\widetilde{L}_0, \widetilde{L}_1)$ is well-defined. After applying the isotopy, we replace $(\mathit{CF}(L_0,L_1), \partial)$ with a chain homotopy equivalent complex $(\mathit{CF}(\widetilde{L}_0,\widetilde{L}_1), \partial')$ and $\tau_{\#}$ with some possibly \textit{not} chain homotopy equivalent chain map $\tau_{\#}'$ which induced by the action of $\tau$ on the generators of $\mathit{CF}(\widetilde{L}_0,\widetilde{L}_1)$. The first page of the Seidel--Smith spectral sequence is the chain complex $(\mathit{CF}(\widetilde{L}_0,\widetilde{L}_1)\otimes \mathbb Z_2[[\theta]], \partial' + (1 + \tau_{\#}')\theta)$; computing vertical differentials first gives a spectral sequence from $\mathit{HF}(L_0,L_1)\otimes \mathbb Z_2[[\theta]] = \mathit{HF}(\widetilde{L}_0,\widetilde{L}_1)\otimes \mathbb Z_2[[\theta]]$ to $\mathit{HF}_{\operatorname{Borel}}(\widetilde{L}_0,\widetilde{L}_1)$, which after tensoring with $\theta^{-1}$ becomes a spectral sequence from $\mathit{HF}(L_0,L_1)\otimes \mathbb Z_2((\theta))$ to $\mathit{HF}(L_0^{\operatorname{inv}},L_1^{\operatorname{inv}})\otimes \mathbb Z_2((\theta))$.

We can in fact dispense with the symplectic structure on $M$ and work on the level of the complex normal bundle $\Upsilon(M^{\operatorname{inv}})$ with its totally real subbundles $NL_0^{\operatorname{inv}} \times \{1\}$ and $J(NL_1^{\operatorname{inv}} \times \{1\})$.  The following lemma is mentioned in \cite[Section 3d]{MR2739000}; a detailed proof is laid out in \cite[Proposition 7.1]{Hendricks}.

\begin{lemma} \label{Nullhomotopy Lemma}

The existence of a stable normal trivialization of $(M, L_0, L_1)$ is implied by the existence of a nullhomotopy of the map
\begin{align*}
(M, L_0, L_1) \rightarrow (BU, BO)
\end{align*}
\noindent which classifies the complex normal bundle $\Upsilon(M^{\operatorname{inv}}) = NM^{\operatorname{inv}} \times [0,1]$ and its totally real subbundles $NL_0^{\operatorname{inv}} \times \{0\}$ over $L_0^{\operatorname{inv}} \times \{0\}$ and $J(NL_1^{\operatorname{inv}}) \times \{1\}$ over $L_1^{\operatorname{inv}} \times \{1\}$.

\end{lemma}

Let $\widetilde{\mathcal D} = (S^2, \widetilde{\boldsymbol \alpha}, \widetilde{\boldsymbol \beta}, \widetilde{\bf w}, \widetilde{\bf z})$ be a multipointed Heegaard diagram for $\widetilde{K}$ defined using the method of Section \ref{Periodic Knots Section}, and $\mathcal D = (S^2, {\boldsymbol \alpha}, {\boldsymbol \beta}, {\bf w}, {\bf z})$ be its quotient under the involution $\tau$.  Given $x$ a point on $D$, let $x^1,x^2$ be its two lifts to $\widetilde{\mathcal D}$ in some order.  There is a natural map
\begin{align*}
\iota \co\operatorname{Sym}^{n_1}(S^2) &\rightarrow \operatorname{Sym}^{2n_1}(S^2) \\
(x_1\cdots x_{n_1}) &\mapsto (x_1^1x_1^2\cdots x_{n_1}^1x_{n_1}^2).
\end{align*}

This map is a holomorphic embedding; for a proof, see \cite[Appendix 1]{Hendricks}.  Moreover, consider the induced involution on $\operatorname{Sym}^{2n_1}(S^2)$, which through a slight abuse of notation we will also call $\tau$.  The fixed set of $\tau$ is exactly our embedded copy of $\operatorname{Sym}^{n_1}(S^2)$; moreover, $\tau$ preserves the two tori $\mathbb T_{\widetilde{\boldsymbol \alpha}}$ and $\mathbb T_{\widetilde{\boldsymbol \beta}}$, with fixed sets $T_{\widetilde{\boldsymbol \alpha}}^{\operatorname{inv}} = \mathbb T_{\boldsymbol \alpha}$ and $T_{\widetilde{\boldsymbol \beta}}^{\operatorname{inv}} = \mathbb T_{\boldsymbol \beta}$.

Perutz has shown that for an arbitrary Heegaard diagram $D = (S, \boldsymbol \alpha, \boldsymbol \beta,\linebreak {\bf w,z})$, there is a symplectic form $\omega$ on $\operatorname{Sym}^{g+n-1}(S)$  which is compatible with the complex structure induced by a complex structure on $S$, and with respect to which the submanifolds $\mathbb T_{\alpha}$ and $\mathbb T_{\beta}$ are in fact Lagrangian and the various Heegaard Floer homology theories are their Lagrangian Floer cohomologies \cite[Thm 1.2]{MR2509747}. In particular, the knot Floer homology is the Floer cohomology of these two tori in the ambient space $\operatorname{Sym}^{g+n-1}(S \backslash ({\bf z} \cup {\bf w}))$, where the removal of the basepoints accounts for the restriction that holomorphic curves not be permitted to intersect the submanifolds $V_{w_i}$ and $V_{z_j}$ of the symmetric product.





In order to apply Theorem \ref{Localization} to the case of doubly periodic knots, we will work with three different subspaces of $\operatorname{Sym}^{2n_1}(S^2)$ and their fixed sets under the involution $\tau$, as follows.
\begin{align*}
M_0 &= \operatorname{Sym}^{2n_1}(S^2 \backslash \widetilde{\bf w}) && 
M_0^{\operatorname{inv}} = \operatorname{Sym}^{n_1}(S^2 \backslash {\bf w}) \\
M_1 &= \operatorname{Sym}^{2n_1}(S^2 \backslash (\widetilde{\bf w} \cup ( \widetilde{\bf z}-z_0))) &&
M_1^{\operatorname{inv}} = \operatorname{Sym}^{n_1}(S^2 \backslash ({\bf w} \cup ({\bf z}-z_0))) \\
M_2 &= \operatorname{Sym}^{2n_1}(S^2 \backslash (\widetilde{\bf w} \cup \widetilde{\bf z})) &&
M_2^{\operatorname{inv}}= \operatorname{Sym}^{n_1}(S^2 \backslash ({\bf w} \cup {\bf z}))
\end{align*}

\noindent In all cases the Lagrangians and their invariant sets under the involution will be as follows.
\begin{align*}
L_0 = \mathbb T_{\widetilde{\boldsymbol \beta}}\qquad & L_0^{\operatorname {inv}} = \mathbb T_{\boldsymbol \beta} \\
L_1 = \mathbb T_{\widetilde{\boldsymbol \alpha}}\qquad & L_1^{\operatorname {inv}} = \mathbb T_{\boldsymbol \alpha}
\end{align*}

The following is immediate from the definitions.

\begin{lemma}

With respect to our choice of symplectic manifolds $M_i$ for $i=1,2,3$ and Lagrangians $L_0$ and $L_1$, we have the following Floer cohomology groups.

\noindent In $M_1$,
\begin{align*}
\mathit{HF}(L_0,L_1) &= \mathit{HF}(\mathbb T_{\widetilde{\boldsymbol \beta}}, \mathbb T_{\widetilde{\boldsymbol \alpha}}) = \widehat{\mathit{HFL}}(S^3, \widetilde{K} \cup \widetilde{U}) \otimes V^{\otimes (2n_1-1)}.\\
\mathit{HF}(L_0^{\operatorname{inv}}, L_1^{\operatorname{inv}})&=\mathit{HF}(\mathbb T_{\boldsymbol \beta}, \mathbb T_{\boldsymbol \alpha}) = \widehat{\mathit{HFL}}(S^3, K \cup U) \otimes V^{\otimes (n_1-1)}.
\end{align*}

\noindent In $M_2$,
\begin{align*}
\mathit{HF}(L_0,L_1)&=\mathit{HF}(\mathbb T_{\widetilde{\boldsymbol \beta}}, \mathbb T_{\widetilde{\boldsymbol \alpha}}) = \widehat{\mathit{HFK}}(S^3, \widetilde{K})\otimes V^{\otimes (2n_1-1)}\otimes W.\\
\mathit{HF}(L_0^{\operatorname{inv}},L_1^{\operatorname{inv}})&=\mathit{HF}(\mathbb T_{\boldsymbol \beta}, \mathbb T_{\boldsymbol \alpha}) = \widehat{\mathit{HFK}}(S^3,K) \otimes V^{\otimes (n_1-1)}\otimes W.
\end{align*}

\end{lemma}

Each of the triples $(M_i, L_0, L_1)$ satisfies the basic symplectic conditions of Seidel and Smith's theory, by arguments similar to \cite[Section 4]{Hendricks}. Therefore to demonstrate that that $(M_i,L_0,L_i)$ carries a stable normal trivialization, it suffices to check the existence of a nulhomotopy of the maps in Lemma~\ref{Nullhomotopy Lemma}. We will prove in Section \ref{Stable Normal Triv Section} that the map $(M_0, L_0, L_1)\rightarrow (BU,BO)$ admits a nulhomotopy; this then naturally restricts to a nulhomotopy of $(M_i, L_0, L_1) \rightarrow (M_0,L_0,L_1) \rightarrow (BU,BO)$ for $i=1,2$. In some moral sense, this is the correct level of generality: the most important feature of our punctured Heegaard diagram is that no periodic domain has nonzero index, and ${\bf w}$ (or ${\bf z}$) is the smallest set of points at which we may puncture $S^2$ and produce a diagram for which this property holds.  Moreover, $M_0^{\operatorname{inv}}$ conveniently deformation retracts onto each of $\mathbb T_{\boldsymbol \alpha}$ and $\mathbb T_{\boldsymbol \beta}$, making certain cohomology computations in Sections \ref{Geometry Section} and \ref{Stable Normal Triv Section} cleaner than they might otherwise be.

Let us pause to consider gradings in these theories. Recall that in general the involution $\tau_{\#}$ on $\widehat{\mathit{CFK}}(\widetilde{\mathcal D})$ arising from the action of $\tau$ on $\tilde{\mathcal D}$ is replaced by an involution $\tau_{\#}'$ on a chain homotopy equivalent complex $\widehat{\mathit{CFK}}(\widetilde{\mathcal D})'$.  To see that $\tau'_{\#}$ preserves Alexander (multi)gradings, consider that Alexander gradings are determined by relative $\operatorname{spin}^{\operatorname{c}}$ structures on $S^3 - \mu(\tilde{L})$, which correspond to homotopy classes of paths between $\mathbb T_{\boldsymbol \alpha}$ and $\mathbb T_{\boldsymbol \beta}$ in $\operatorname{Sym}^{2n_1}(\Sigma(S)\backslash\linebreak (\widetilde{\bf w} \cup \widetilde{\bf z}))$ \cite{MR2443092}. If the isotopy of Theorem \ref{Localization} replaces $\mathbb T_{\boldsymbol \alpha}$ and $\mathbb T_{\boldsymbol \beta}$ with $\mathbb T_{\boldsymbol \alpha}'$ and $\mathbb T_{\boldsymbol \beta}'$, then homotopy classes of paths between the original Lagrangians are in canonical bijection with homotopy classes of paths between the new Lagrangians, and therefore $\tau_{\#}'$ also preserves relative $\operatorname{spin}^{\operatorname{c}}$ structures, and therefore Alexander gradings. Moreover, within any fixed Alexander grading, $\tau_{\#}'$ must preserve the Maslov grading, because the composition of any holomorphic disk with $\tau$ is another holomorphic disk with the same Maslov index, and within any Alexander grading the Maslov grading is entirely determined by the Maslov indices of Whitney disks. In some simple cases, such as those computed in Section \ref{Examples Section}, knowing these properties of $\tau'_{\#}$ constrains the behavior of the spectral sequence very strongly.

Finally, the localization map of Theorem \ref{Localization} is entirely constructed by counting pseudoholomorphic disks of varying index and by multiplication and division by powers of $\theta$.  In particular,  on $M_1^{\operatorname{inv}}$, the localization isomorphism preserves the Alexander multigrading $(A_1,A_2)$ on the $E^{\infty}$ page, because no flowlines can pass over the missing basepoint divisors $V_{z_i}$ and $V_{w_i}$ for $0\leq i\leq n_1$, and on $M_2^{\operatorname{inv}}$, the localization isomorphism preserves the Alexander grading $A_1$ since flowlines can pass over $V_{z_0}$ and $V_{w_0}$ but no other basepoint divisors.  In general, however, we should expect that the localization isomorphisms will not preserve the data of the Maslov grading.

\section{The geometry of the symmetric product} \label{Geometry Section}

Over the next two sections, we will prove that the triple $(M_0^{\operatorname{inv}}, L_0^{\operatorname{inv}}, L_1^{\operatorname{inv}}) $ satisfies the the complex conditions of Lemma \ref{Nullhomotopy Lemma}, and therefore has a stable normal trivialization. The first thing to do is describe the homotopy type and cohomology of $M_0^{\operatorname{inv}}$.

We claim $M_0^{\operatorname{inv}}$ deformation retracts onto each of $\mathbb T_{\boldsymbol \alpha}$ and $\mathbb T_{\boldsymbol \beta}$.  To check this, we refer to a lemma whose proof is outlined in \cite[Lemma 5.1]{Hendricks} following an argument of Ong \cite{MR1993792}.

\begin{lemma}

The $r$th symmetric product of a wedge of $k$ circles $\vee_{i=1}^k S^1_i$ deformation retracts onto the $r$-skeleton of the $k$ torus $\prod_{i=1}^k S^1_i$, where each circle is given a CW structure consisting of the wedge point and a single one-cell, and the torus has the natural product CW structure.

\end{lemma}

We will apply this observation to $M_0^{\operatorname{inv}}$. In $S^2$, let $\nu_{i}\co[0,1] \rightarrow S^2$ be a small closed curve around $z_i$ for $0 \leq i \leq n_1$, such that $\nu_i$ is oriented counterclockwise in the complement of $w_{0}$. Then
\begin{align*}
H_1(S^2 \backslash {\bf z}) &= \mathbb Z\langle \nu_0, \nu_1,\ldots,\nu_{n_1} \rangle.
\end{align*}
Now $S^2\backslash \{{\bf w}\}$ deformation retracts onto a wedge of $n_1$ circles $\vee_{i=1}^{n_1} \nu'_i$, where $\nu_i'$ is a closed curve homotopic to $\nu_i$ which passes once through the origin.  Therefore $M_0^{\operatorname{inv}} =\operatorname{Sym}^{n_1}(S^2 \backslash {\bf w})$ deformation retracts onto the symmetric product of $\vee_{i=1}^{n_1} \nu'_i$, which in turn deformation retracts onto the product $\prod_{i=1}^{n_1} \nu'_i$.  However, this product is homotopy equivalent to $\prod_{i=1}^{n_1} \nu_i$, and indeed to $\prod_{i=1}^{n_1} \alpha_i$ and to $\prod_{i=1}^n \beta_i$.  We conclude that $M_0^{\operatorname{inv}}$ has the homotopy type of an $n_1$-torus and in particular admits a deformation retraction onto each of $\mathbb T_{\boldsymbol \alpha}$ and $\mathbb T_{\boldsymbol \beta}$.

We will require a concrete description of the cohomology rings of $M_0^{\operatorname{inv}}$, $\mathbb T_{\boldsymbol \alpha}$, and $\mathbb T_{\boldsymbol \beta}$ for the computations of Section \ref{Stable Normal Triv Section}, so we supply one now. Consider the one cycles
\begin{align*}
\overline{\nu_i}\co [0,1] &\rightarrow M_0^{\operatorname{inv}} \\
t &\mapsto (\nu_i(t)x_0\cdots x_0)
\end{align*}
\noindent where $x_0$ is any choice of basepoint. The $[\overline{\nu_i}]$ form a basis for $H_1(M_0^{\operatorname{inv}})$.  Ergo, letting $[\overline{\nu_i}]^*$ denote the dual of $\overline{\nu_i}$, we have
\begin{align*}
H^1(M_0^{\operatorname{inv}}) &= \mathbb Z \langle [\overline{\nu_0}]^*,\ldots,[\overline{\nu_{n_1}}]^* \rangle \\
H^k(M_0^{\operatorname{inv}}) &= \textstyle{\bigwedge^k}H^1(M_0^{\operatorname{inv}})
\end{align*}

Similarly, we can write down the homology of the tori $\mathbb T_{\boldsymbol \alpha}$ and $\mathbb T_{\boldsymbol \beta}$. Through a slight abuse of notation, let us insist that we have parametrizations $\alpha_i \co [0,1] \rightarrow S^2$ running counterclockwise in $S^2 \backslash \{w_0\}$ and $\beta_i$ running clockwise. The first homology of $\mathbb T_{\boldsymbol \alpha}$ is generated by one-cycles
\begin{align*}
\overline{\alpha_i}\co [0,1] &\rightarrow \alpha_1 \times \cdots \times \alpha_{n_1} \\
					t &\mapsto (y_1,\ldots,y_{i-1}, \alpha_i(t),y_{i+1},\ldots,y_{n_1}) \\					
\end{align*}
\noindent where $y_j = \alpha_j(0)$, and thus the cohomology of this torus is
\begin{align*}
H^1(\mathbb T_{\boldsymbol \alpha}) &= \mathbb Z\langle  [\overline {\alpha_1}]^*,\ldots,[\overline{ \alpha_{n_1}}]^* \rangle \\
H^k(\mathbb T_{\boldsymbol \alpha}) &= \textstyle{\bigwedge^k}H^1(\mathbb T_{\boldsymbol \alpha})
\end{align*}
\noindent We apply analagous naming conventions to $\mathbb T_{\boldsymbol \beta}$, obtaining
\begin{align*}
H^1(\mathbb T_{\boldsymbol \beta}) &= \mathbb Z\langle  [\overline {\beta_1}]^*,\ldots,[\overline {\beta_{n_1}}]^* \rangle \\
H^k(\mathbb T_{\boldsymbol \beta}) &= \textstyle{\bigwedge^k}H^1(\mathbb T_{\boldsymbol \beta})
\end{align*}
Let $X \subset M_0^{\operatorname{inv}}\times [0,1]$ be $(L_0^{\operatorname{inv}} \times \{0\}) \cup (L_1^{\operatorname{inv}} \times \{1\})$. Observe that under the map on homology induced by inclusion $\iota \co X \hookrightarrow M_0^{\operatorname{inv}} \times[0,1]$, both $[\overline{\alpha_i}]$ and of $[\overline{\beta_i}]$ are sent to $[\overline{\nu_i}]$. Therefore the map $\iota^*$ on cohomology $H^k(M_0^{\operatorname{inv}} \times [0,1]) \rightarrow H^k(X)$ is precisely the diagonal map, with
\[
\textstyle{\bigwedge_{j=1}^m}  [\overline{\nu_{i_j}}]^* \mapsto \textstyle{\bigwedge_{j=1}^m}  [\overline{\alpha_{i_j}}]^* + \textstyle{\bigwedge_{j=1}^m}  [\overline{\beta_{i_j}}]^*
\]

\begin{corollary}

The relative cohomology $H^*(M_0^{\operatorname{inv}} \times [0,1], X)$ is the cohomology of the torus $(S^1)^{n_1}$, and in particular is torsion-free.

\end{corollary}

\begin{proof}

Consider the long exact sequence
\begin{equation*}
\xymatrix{
\cdots H^{m-1}(X) \ar[r] & H^m(M_0^{\operatorname{inv}} \times[0,1], X)  \ar[r]^-{q^*} & H^m(M_0^{\operatorname{inv}}) \ar[r]^-{\iota^*} &  H^m(X)\cdots
}
\end{equation*}

Taking into account that the obvious isomorphism between $H^*(\mathbb T_{\boldsymbol \alpha})$ and $H^*(\mathbb T_{\boldsymbol \beta})$ respects our labelling of the cohomology classes, $\iota^*$ is the diagonal map, and in particular an injection. When $m\geq 2$ we have the following short exact sequences.
\begin{align*}
0\!&\xymatrix@R =.1 cm{
 \ar[r]& H^{m-1}(M_0^{\operatorname{inv}} \times [0,1]) \ar[r]^-{\iota^*} &H^{m-1}(X)}\\
&\hspace{11.5em}\xymatrix@R =.1 cm{\ar[r] & H^m(M_0^{\operatorname{inv}} \times [0,1], X) \ar[r] & 0
}
\end{align*}

Here the map $i^*$ sends the wedge  $\textstyle{\bigwedge_{j=1}^{m-1}} [\overline{\nu_{i_j}}]^*$ to the sum $\textstyle{\bigwedge_{j=1}^{m-1}} [\overline{\alpha_{i_j}}]^* + \textstyle{\bigwedge_{j=1}^{m-1}} [\overline{\beta_{i_j}}]^*$. We therefore see that for $m \geq 2$, $H^m(M_0^{\operatorname{inv}} \times [0,1], X) \cong\linebreak H^{m-1}(\mathbb T_{\boldsymbol \alpha}) \cong H^{m-1}(T_{\boldsymbol \beta})$.  (We will have occasion to be careful about the generators in Section \ref{Stable Normal Triv Section}.) The same result follows for $H^1(M_0^{\operatorname{inv}} \times [0,1], X)$ trivially.
\end{proof}

Let us consider the implications of this result for the relative $K$-theory of $(M_0^{\operatorname{inv}} \times [0,1], X)$, which is isomorphic to the reduced $K$-theory $\widetilde{K}((M_0^{\operatorname{inv}} \times [0,1])/ X)$.  (For a review of $K$-theory, see \cite[Chapter 2]{HatcherKTheory}; all of the facts used in this paper are also summarized in \cite[Section 6]{Hendricks}.) Notice that $(M_0^{\operatorname{inv}} \times [0,1], X)$ deformation retracts onto the compact CW pair $(\operatorname{Sym}^{n_1}(S^2\, \backslash \linebreak \bigcup_i \nu(w_i)) \times [0,1], X)$, where $\nu(w_i)$ is a small open neighborhood.  This deformation retraction makes it legitimate to consider the reduced $K$-theory of the pair by identifying it with $\widetilde{K}((\operatorname{Sym}^{n_1}(S^2 \backslash \bigcup_i \nu(w_i)) \times [0,1])/ X)$.  From now on we will apply this trick without mention.

Recall that there is a rational ring isomorphism
\begin{align*}
\widetilde{\operatorname{ch}} \co (\widetilde{K}^0(X) \oplus \widetilde{K}^1(X))\otimes \mathbb Q \rightarrow \widetilde{H}^*(X; \mathbb Q)
\end{align*}
\noindent between the rational reduced $K$-theory and rational reduced cohomology of any space $X$ chosen such that if $V$ is a line bundle over $X$ and $c_1(V)$ is the first Chern class of $X$, $\widetilde{\operatorname{ch}}(V) = \sum_{i=1}^{\infty} \frac{c_1(V)^i}{i!}$ and $\widetilde{\operatorname{ch}}$ is a ring homomorphism.  Using this isomorphism and the Atiyah-Hirzebruch spectral sequence, Atiyah and Hirzebruch have shown that if the reduced cohomology $\widetilde{H}^*(X)$ is torsion-free, the reduced $K$-theory $\widetilde{K}^*(X)$ is as well, and the stable isomorphism class of a vector bundle is determined entirely by its Chern classes \cite[Section 2.5]{MR0139181}.  In particular, since $H^*(M_0^{\operatorname{inv}} \times [0,1], X) = \widetilde{H}^*((M_0^{\operatorname{inv}} \times [0,1])/\linebreak X)$ is torsion-free, the stable isomorphism class of a complex vector bundle over $(M_0^{\operatorname{inv}}, X)$ --- that is, a bundle whose restriction to $X$ is stably trivial --- is entirely determined by its Chern classes.  To show such a bundle is stably trivial, it suffices to show that all its Chern classes are zero.

\section{Stable normal triviality of the normal bundle} \label{Stable Normal Triv Section}

Having remarked that the symplectic geometry conditions of Seidel and Smith's theory are satisfied for $(M_i, L_0, L_1)$, when $i=0,1,2$, we now proceed to check that $(M_0, L_0, L_1)$ fulfills the complex conditions that, by Lemma~\ref{Nullhomotopy Lemma}, imply the existence of a stable normal trivialization.

\begin{proposition} \label{Stable Normal Trivialization}

Consider the complex manifold $M_0=\operatorname{Sym}^{2n_1}(S^2 \backslash \widetilde{\bf w})$ to\-gether with its totally real submanifolds $L_0 = \mathbb T_{\boldsymbol \beta}$ and $L_1= \mathbb T_{\boldsymbol \alpha}$, and the holomorphic involution $\tau$ which preverves $L_0$ and $L_1$.  The map
\begin{align*}
(M_0^{\operatorname{inv}} \times [0,1], (L_0 \times \{0\}) \cup (L_1 \times \{1\})) \rightarrow (BU, BO)
\end{align*}
\noindent which classifies the pullback $\Upsilon(M_0^{\operatorname{inv}}) = N(M_0^{\operatorname{inv}}) \times [0,1]$ of the complex normal bundle of $M_0^{\operatorname{inv}}$ together with the totally real subbundles $NL_0^{\operatorname{inv}} \times\{0\}$ over $L_0^{\operatorname{inv}} \times \{0\}$ and $J(NL_1^{\operatorname{inv}}) \times \{1\}$ over $L_1^{\operatorname{inv}} \times \{1\}$ is nulhomotopic.

\end{proposition}

As a first step, we must establish the complex triviality of $NM_0^{\operatorname{inv}}$ (and thus of $\Upsilon(M_0^{\operatorname{inv}}$)) and the real triviality of $NL_i^{\operatorname{inv}}$ for $i=0,1$.

\begin{lemma}

The complex bundle $NM_0^{\operatorname{inv}}$ is stably trivial.

\end{lemma}

\begin{proof}

The inclusion map $\iota_1 \co (S^2 \backslash {\bf w}) \hookrightarrow S^2$ is nulhomotopic.  Therefore the induced inclusion $\operatorname{Sym}^{n_1}(\iota_1)\co \operatorname{Sym}^{n_1}(S^2 \backslash 
{\bf w}) \hookrightarrow  \operatorname{Sym}^{n_1}(S^2)$ is also nulhomotopic. The normal bundle of $M_0^{\operatorname{inv}} = \operatorname{Sym}^{n_1}(S^2 \backslash {\bf w})$ in $M_0 = \operatorname{Sym}^{2n_1}(S^2 \backslash \widetilde{\bf w})$ is exactly the restriction of the normal bundle 
to $\operatorname{Sym}^{n_1}(S^2)$ in $\operatorname{Sym}^{2n_1}(S^2)$ along the inclusion map $\operatorname{Sym}^{n_1}(\iota_1)$.  As the map $\operatorname{Sym}^{n_1}(\iota_1)$ is nulhomotopic, $NM_0^{\operatorname{inv}}$ is stably trivial.\end{proof}

The proof of the next lemma proceeds exactly as in \cite[Lemma 7.3]{Hendricks}.

\begin{lemma} \label{Torus Triviality Lemma}
The normal bundles of $\mathbb T_{\boldsymbol \alpha} \subset \mathbb T_{\widetilde{\boldsymbol \alpha}}$ and $\mathbb T_{\boldsymbol \beta} \subset \mathbb T_{\widetilde{\boldsymbol \beta}}$ are trivial.
\end{lemma}

We now turn to the question of relative triviality.  Let $X = (L_0 \times \{0\}) \cup (L_1 \times \{1\}))$ as in Section \ref{Geometry Section}.  Choose preferred trivializations of the totally real bundles  $NL_0^{\operatorname{inv}} \times \{0\}$ and $J(NL_1^{\operatorname{inv}}) \times \{1\}$ and tensor with $\mathbb C$ to extend to a preferred trivialization of the complex bundle $\Upsilon(M_0^{\operatorname{inv}})|_{X}$. We use this trivialization to pull back $[\Upsilon(M_0^{\operatorname{inv}})] \in \widetilde{K}^0(M_0^{\operatorname{inv}} \times [0,1])$ to a relative bundle $[\Upsilon(M_0^{\operatorname{inv}})]_{\operatorname{rel}} \in \tilde{K}^0((M_0^{\operatorname{inv}} \times [0,1])/X)$. Because the reduced cohomology, and therefore the reduced $K$-theory, of $(M_0^{\operatorname{inv}}\times [0,1])/X$ have no torsion, to verify triviality of the relative bundle, it suffices to show that the Chern classes of $[\Upsilon(M_0^{\operatorname{inv}})]_{\operatorname{rel}}$ are trivial. 

\begin{remark} 
It may be helpful to draw attention to a minor problem of notation: in earlier sections, $\widetilde{K}$ is a doubly periodic knot, but also in the present section $\widetilde{K}^0(B)$ is the reduced $K$-theory of the topological space $B$.  We hope this will not occasion confusion.
\end{remark}

Fundamentally, the argument for relative triviality rests on the fact that each of the $n_1$ linearly independent periodic domains in $S^2 \backslash {\bf w}$ has Maslov index zero. Therefore, we pause here to recall the notation of Section \ref{Periodic Knots Section}. Let $x_i$ be the single positive intersection point in $\alpha_i \cap \beta_i$ and $y_i$ the negative intersection point.  Let $F_i$ be the closure of the component of $S - \alpha_i - \beta_i$ containing $z_{i}$ and $E_i$ be the closure of the component of $S - \alpha_i -\beta_i$ containing $z_{i+1}$ (or $z_1$ if $i=n_1$).  Then $P_i = E_i - F_i$ is a periodic domain of index zero on $\mathcal D$ with boundary $\beta_i - \alpha_i$.  Finally, let $\gamma_i$ be the union of the arc of $\alpha_i$ running from $x_i$ to $y_i$ and the arc of $\beta_i$ running from $x_i$ to $y_i$.  In particular, this specifies that $\gamma_i$ has no intersection with any $\alpha$ or $\beta$ curves other than $\alpha_i$ and $\beta_i$, and moreover the component of $S - \gamma_i$ which does not contain $w_0$ contains only a single basepoint $w_i$. See Figure \ref{Periodic Domains Figure} for an illustration of the domain $P_i$.

\begin{figure}
\centering
\includegraphics{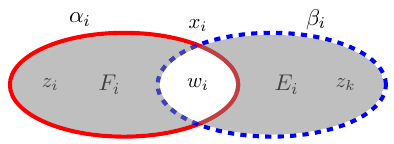}
\caption{The periodic domain $P_i = E_i - F_i$ has Maslov index zero.}
\label{Periodic Domains Figure}
\end{figure}

The structure of the argument is as follows: for each even $k$ such that $1 \leq k \leq 2n_1$, we will use the periodic domains $P_i$ to produce a set of $k$-chains $\{W_{\bf I}\}$ in $(M_0^{\operatorname{inv}} \times [0,1], X)$ whose relative homology classes generate the $k$th homology of $H_k(M_0^{\operatorname{inv}} \times[0,1], X)$.  We will then show that the restriction of $\Upsilon(M^{\operatorname{inv}}_0)|_{\operatorname{rel}}$ to each $W_{\bf I}$ is trivial as a relative vector bundle, and that therefore $\langle c_k(\Upsilon(M_0^{\operatorname{inv}})|_{\operatorname{rel}}), [W_{\bf I}] \rangle=0$.  Since the $[W_{\bf I}]$ generate $H_k(M_0^{\operatorname{inv}}\times[0,1], X)$, we will have proved that $c_k(\Upsilon(M_0^{\operatorname{inv}})|_{\operatorname{rel}})$ is identically zero.

More specifically, we will describe the chains $W_{\bf I}$ as a subset of a product of two-chains $Y_i$ in $(S^2 \backslash {\bf w}) \times [0,1]$ such that the projection to $S^2 \backslash {\bf w}$ is the periodic domain $P_i$ and the $Y_i$ are pairwise disjoint. We will then show that the restriction of the bundle $\Upsilon(M^{\operatorname{inv}}_0)|_{\operatorname{rel}}$ to $W_{\bf I}$ also breaks up as the restriction of a product of relative bundles $\Upsilon(Y_i)$ over the $Y_i$ which are known to be trivial through Maslov index arguments.  Let us begin by constructing these manifolds $Y_i$.

For $1 \leq i \leq n_1$, let $Y_i$ be a subspace of $S^2 \times [0,1]$ with the following properties:

\begin{itemize}

\item $Y_i$ is topologically $S^1 \times [0,1]$, and $Y_i \cap (S^2 \times \{t\})$ is $S^1$ for all $t$.

\item The boundary of $Y_i$ is $\beta_i \times \{0\} - \alpha_i \times \{1\}$.

\item The projection of $Y_i$ to the punctured sphere is a copy of the periodic domain $P_i$.

\item $Y_i \cap Y_j = \emptyset$ if $i \neq j$.

\item If $x_i$ is the point of positive intersection of $\alpha_i$ and $\beta_i$, $\{x_i\} \times [0,1] \subset Y_i$.

\end{itemize}

The 2-chains $Y_i$ are constructed as follows: for $t \in [0, \frac{1}{2}]$, choose a linear homotopy $H_t$ from $\beta_i$ to $\gamma_i$ inside the domain $E_i$ which fixes $\beta_i \cap \gamma_i$. Let the intersection of $Y_i$ with $S^2 \times \{t\}$ be the embedded circle $H_t(\beta_i) \times \{t\}$. Similarly, for $t \in [\frac{1}{2}, 1]$, choose a linear homotopy $J_t$ from $\gamma_i$ to $\alpha_i$ inside the domain $F_i$ which fixes $\alpha_i \cap \gamma_i$, and let the intersection of $Y_i$ with $S^2 \times \{t\}$ be $J_t(\gamma_i)$.

Observe that this description of $Y_i$ has the following properties: first, $Y_i$ is contained in $(S^2 \backslash {\bf w}) \times [0,1]$ as promised.  Second, $Y_i$ contains the line segment $\{x_i\} \times [0,1]$.  Third, the intersection of $Y_i$ with $(\alpha_i \cup \beta_i)\times (0,1)$ is entirely contained in the cylinder $\gamma_i \times (0,1)$.  Finally, the projection of $Y_i \cap (S^2 \times [0, \frac{1}{2}])$ to $S^2$ lies entirely ``inside'' $\beta_i$ - that is, on the component of $S^2 \backslash \beta_i$ not containing $w_0$ - implying that the sets $Y_i \cap (S^2 \times [0, \frac{1}{2}])$ are pairwise disjoint.  Similarly, the projection of $Y_i \cap (S^2 \times [\frac{1}{2},1])$ to $S^2$ lies entirely ``inside'' $\alpha_i$, and therefore the sets  $Y_i \cap (S^2 \times [\frac{1}{2},1])$ are pairwise disjoint.  Ergo the $Y_i$ are pairwise disjoint.

We are now ready to define the complex line bundles $\Upsilon(Y_i)$.  Recall from Section \ref{Floer Cohomology Section} that there is a holomorphic embedding
\begin{align*}
\iota \co M^{\operatorname{inv}}_0 = \operatorname{Sym}^{n_1}(S^2 \backslash {\bf w}) &\rightarrow \operatorname{Sym}^{n_1}(S^2 \backslash \widetilde{\bf w}) = M_0 \\
(x_1\cdots x_{n_1}) &\mapsto (x_1^1x_1^2\cdots x_{n_1}^1 x_{n_1}^2).
\end{align*}
\noindent Similarly, we have a holomorphic embedding
\begin{align*}
\iota' \co (S^2 \backslash {\bf w}) &\rightarrow \operatorname{Sym}^2(S^2 \backslash \widetilde{\bf w}) \\
x &\mapsto (x^1x^2)
\end{align*}
which takes a point $x$ on $S^2 \backslash {\bf w}$ to the unordered pair consisting of its two (not necessarily distinct) lifts on $S^2 \backslash \widetilde{\bf w}$ under the projection map $\pi \co (S^2 \backslash \widetilde{\bf w})\linebreak \rightarrow (S^2 \backslash {\bf w})$. Notice that for each $1 \leq i \leq n_1$, $\iota'(\alpha_i) \subset \alpha_i^1 \times \alpha_i^2$ and $\iota'(\beta_i) \subset \beta_i^1 \times \beta_i^2$.  Let $N(S^2 \backslash {\bf w})$ be the normal bundle to $S^2 \backslash {\bf w}$ in the second symmetric product $\operatorname{Sym}^2(S^2 \backslash \widetilde{\bf w})$.  This is a complex line bundle over a punctured sphere, hence trivial.  For each $1 \leq i \leq n_1$, let $N\alpha_i$ be the real normal bundle to $\alpha_i$ in $\alpha_i^1 \times \alpha_i^2$ and $N\beta_i$ be the real normal bundle to $\beta_i$ in $\beta_i^1 \times \beta_i^2$.  An argument similar to that of Lemma \ref{Torus Triviality Lemma} shows that $N\alpha_i$ and $N\beta_i$ are trivial real line bundles for all $i$.

Consider the pullback of the normal bundle $N(S^2 \backslash {\bf w})$ to $(S^2 \backslash {\bf w}) \times [0,1]$. The complex line bundle over $Y_i$ that interests us is the restriction of this pullback to $Y_i$, which we shall by analogy denote $\Upsilon(Y_i)$.  That is, $\Upsilon(Y_i) = (N(S^2 \backslash {\bf w}) \times [0,1])|_{Y_i}$. This complex line bundle has totally real subbundles $J(N\alpha_i \times \{1\})$ over $\alpha_i \times \{1\}$ and $N\beta_i \times \{0\}$ over $\beta_i \times  \{0\}$. Choose preferred real trivializations of these real line bundles, and extend to a complex trivialization of $\Upsilon(Y_i)|_{(\beta_i \times \{0\}) \cup (\alpha_i \times \{1\})}$ by tensoring with $\mathbb C$. For convenience, let this subspace $(\beta_i \times \{0\}) \cup (\alpha_i \times \{1\})$ be $X_i$.

\begin{lemma}

For $1 \leq i \leq n_1$, given any preferred trivialization of $\Upsilon(Y_i)|_{X_i}$ as a complex vector bundle, the relative vector bundle $\Upsilon(Y_i)|_{\operatorname{rel}}$ over $(Y_i,X_i)$ is stably trivial.

\end{lemma}

\begin{proof}

Let $\iota_i$ be the inclusion of $Y_i$ into $(S^2 \backslash {\bf w})\times[0,1]$, and $p$ be the projection of $(S^2 \backslash {\bf w})\times [0,1]$ to  $S^2 \backslash {\bf w}$. Then the image of $p \circ \iota_i(Y_i)$ is the periodic domain $P_i$.

Consider the commutative diagram below. The top horizontal inclusion of $S^2 \backslash \{w\}$ into $\operatorname{Sym}^{n_1}(S^2 \backslash {\bf w})$ is defined by mapping a point $x$ to $(x x_1\cdots\widehat{x_i}\cdots x_{n_1})$, where again each $x_j$ is the positively oriented point in $\alpha_j \cup \beta_j$.  The bottom inclusion of $\operatorname{Sym}^2(S^2 \backslash \widetilde{\bf w})$ into $\operatorname{Sym}^{2n_1}(S^2 \backslash \widetilde{\bf w})$ sends an unordered pair $(xy)$ to $(xyx_i^1x_i^2\cdots \widehat{x_i^1}\widehat{x_i^2}\cdots x_{n_1}^1x_{n_1}^2)$.

\[
\xymatrix{
Y_i \ar[r]^-{p \circ \iota_i} \ar[dr] & S^2 \backslash {\bf w} \ar@{^{(}->}[d]^-{\iota'} \lhook\mkern-7mu\ar[r]& \operatorname{Sym}^{n_1}(S^2 \backslash {\bf w}) \ar@{^{(}->}[d]^-{\iota} \\
& \operatorname{Sym}^2(S^2 \backslash \widetilde{\bf w}) \lhook\mkern-7mu\ar[r]& \operatorname{Sym}^{2n_1}(S^2 \backslash \widetilde{\bf w})
}
\]

Consider the map $\phi \co Y_i \rightarrow \operatorname{Sym}^{n_1}(S^2 \backslash {\bf w})$ given by composition along the top row of the diagram.  This is a topological annulus representing the periodic domain $P_i$. Since $P_i$ has Maslov index zero, by the discussion in Section \ref{Heegaard Floer Background Section}, the first Chern class of the pullback of the complex tangent bundle $T\operatorname{Sym}^{n_1}(S^2 \backslash {\bf w})$ to $Y_i$ relative to the complexification of the pullbacks of the real tangent bundles $J(T(\mathbb T_{\boldsymbol \alpha}))$ to one component of the boundary of $Y_i$ and $T(\mathbb T_{\boldsymbol \beta})$ to the other is zero. However, the pullback $\phi^*(T\operatorname{Sym}^{n_1}(S^2 \backslash {\bf w}))$ of the tangent bundle of the total symmetric product to $Y_i$ is exactly $(p \circ \iota_1)_*(T(S^2 \backslash {\bf w}) \oplus \mathbb C^{n_1 - 1})$, where the factors of $\mathbb C$ are the restriction of the tangent bundle of the punctured sphere to the points $x_j$ such that $j \neq i$. Moreover, $(p \circ \iota_1)^*(T(S^2 \backslash {\bf w}))$ is precisely the restriction of the pullback bundle $p^*(T(S^2 \backslash {\bf w})) = T((S^2 \backslash {\bf w})\times [0,1])$ over $(S^2 \backslash {\bf w}) \times [0,1]$ to the subspace $Y_i$.  Therefore $\phi^*(T\operatorname{Sym}^{n_1}(S^2 \backslash {\bf w})) = T((S \backslash {\bf w}) \times [0,1])|_{Y_i} \oplus \mathbb C^{n_1 - 1}$.

Similarly, the pullback $\phi^*(J(T (\mathbb T_{\boldsymbol \alpha})))$ to the boundary component $\alpha_i \times \{1\} \subset Y_i$ is $(p \circ \iota_i)^*(J(T(\alpha_i)\oplus \mathbb R^{n_1-1}))$, where the factors of $\mathbb R$ are the canonical real subspace of the tangent bundle to the punctured sphere at the points $x_j$ for $j \neq i$.  The pullback of this bundle to $\alpha_i \times \{1\} \subset Y_i$ is $J((T(\alpha_i) \times \{1\}) \oplus \mathbb R^{n_1-1})$.  A similar argument shows that $\phi^*(J(T(\mathbb T_{\boldsymbol \beta})))$ is $(T(\beta_i)\times \{0\}) \oplus \mathbb R^{n_1-1}$.  Therefore we have seen that the complex vector bundle\linebreak $T((S^2 \backslash {\bf w})\times [0,1])|_{Y_i} \oplus \mathbb C^{n_1-1}$ relative to the complexification of its totally real subbundles $(T(\beta_i)\times \{0\})\oplus \mathbb R^{n_1+1}$ over $\beta_i \times \{0\}$ and $J((T(\alpha_i)\times \{1\})\oplus \mathbb R^{n_1+1})$ over $\alpha_i \times \{1\}$ has relative first Chern class zero. Hence the same is true of the complex line bundle $T((S^2 \backslash {\bf w})\times [0,1])|_{Y_i}$ relative to the complexification of its totally real subbundles $J(T(\alpha_i)\times \{1\})$ and $T(\beta_i)\times\{0\}$.  As this is a line bundle, triviality of the first relative Chern class suffices to show stable triviality of the relative vector bundle.

Now consider the map $\widetilde{\phi} = \iota \circ \phi$ from $Y_i$ to $\operatorname{Sym}^{2n_1}(S^2 \backslash \widetilde{\bf w})$.  This is a topological annulus representing the periodic domain $\pi^{-1}(P_i)$ in $\widetilde{\mathcal D}$, which also has Maslov index zero.  Therefore the pullback along $\widetilde{\phi}$ of the complex tangent bundle $T(\operatorname{Sym}^{2n_1}(S^2 \backslash \widetilde{\bf w}))$ to $Y_i$ relative to complexifications of pullbacks of the totally real subbundles $J(T(\mathbb T_{\widetilde{\boldsymbol \alpha}}))$ and $T(\mathbb T_{\widetilde{\boldsymbol \beta}})$ to $\alpha_i \times \{1\}$ and $\beta_i \times \{0\}$ has trivial relative first Chern class. However, once again the pullback of this relative bundle along the inclusion map $\operatorname{Sym}^{2}(S^2 \backslash \widetilde{\bf w}) \hookrightarrow \operatorname{Sym}^{2n_1}(S^2 \backslash \widetilde{\bf w})$ decomposes into a much smaller tangent bundle together with a trivial summand. Observe that $\widetilde{\phi}^*(T(\operatorname{Sym}^{2n_1}(S^2 \backslash \widetilde{\bf w}))$ is 
$$(\iota' \circ (p \circ \iota_i))^*(T(\operatorname{Sym}^{2}(S^2 \backslash \widetilde{\bf w})) \oplus \mathbb C^{2n_1 - 2})).$$
Pulling back along the inclusion map $\iota'$, we see that this bundle breaks up as $T(S^2 \backslash {\bf w}) \oplus N(S^2 \backslash {\bf w}) \oplus \mathbb C^{2n_1 - 2}$ over $S^2 \backslash {\bf w}$, and that therefore its ultimate pullback to $Y_i$ is $T((S^2 \backslash {\bf w}) \times [0,1]) \oplus \Upsilon(Y_i) \oplus \mathbb C^{2n_1 -2}$.

Similarly, the pullback along $\widetilde{\phi}$ of $J(T(\mathbb T_{\widetilde{\boldsymbol \alpha}}))$ to $\alpha_i \times \{1\}$ is precisely $J(T{\alpha_i} \times \{1\}) \oplus J(N{\alpha_i} \times \{1\}) \oplus J(\mathbb R^{2n_1-2})$.  Finally, the pullback along $\widetilde{\phi}$ of $T(\mathbb T_{\widetilde{\boldsymbol \beta}})$ to $\beta_i \times \{0\}$ is $(T{\beta_i} \times \{0\}) \oplus (N_{\beta_i} \times \{0\}) \oplus \mathbb R^{2n_1-2}$.  Dropping the trivial summands, we conclude that the first relative Chern class of $T((S^2 \backslash {\bf w}) \times [0,1]) \oplus \Upsilon(Y_i)$ relative to complexifications of the totally real subbundle $(T{\beta_i} \times \{0\}) \oplus (N_{\beta_i} \times \{0\})$ over $\beta_i \times \{0\}$ and the totally real subbundle\linebreak $J(T{\alpha_i} \times \{1\}) \oplus J(N{\alpha_i} \times \{1\})$ over $\alpha_i \times \{1\}$ has relative first Chern class zero.  This, combined with our previous conclusions concerning relative triviality of $T((S^2 \backslash {\bf w}) \times [0,1])$ with respect to its subbundles $T\beta_i \times \{1\}$ over $\beta_i \times \{1\}$ and $J(T\alpha_i \times \{0\})$ over $\alpha_i \times \{0\}$, implies that $\Upsilon(Y_i)$ has relative first Chern class zero with respect to complexifications of its subbundles $N\beta_i \times \{0\}$ and $J(N\alpha_1 \times \{1\})$, as promised.
\end{proof}

We are now ready to construct generators for $H_k(M_0^{\operatorname{inv}} \times [0,1], X)$, and use them to prove that the relative Chern classes of $[\Upsilon(M_0^{\operatorname{inv}})]$ are identically zero.  Recall that for $k > 1$ there are short exact sequences on homology
\begin{align*}
0\!&\xymatrix{
\ar[r] & H_{k+1}(M_0^{\operatorname{inv}} \times [0,1], X) \ar[r] & H_k(\mathbb T_{\boldsymbol \beta}) \oplus H_k(\mathbb T_{\boldsymbol \alpha})}\\
&\hspace{12.5em}\xymatrix{\ar[r]^-{\iota_*} & H_k(M_0^{\operatorname{inv}}) \ar[r] &0.
}
\end{align*}

As in Section \ref{Geometry Section}, for $k>1$, the $(k+1)\,$st homology group $H_{k+1}(M_0^{\operatorname{inv}} \times [0,1], X)$ is the kernel of the map $\iota_*$.  Let us take a closer look at generators for this group. We have seen that the kernel of $\iota_*$ is $\mathbb Z \langle \textstyle{\bigwedge_{j=1}^{k}} [\beta_{i_j}] \oplus -\textstyle{\bigwedge_{j=1}^k} [\alpha_{i_j}] : 1 \leq i_1 < \cdots <i_k \leq n_1 \rangle$. Therefore for each ${\bf I} = (i_1,\ldots,i_k)$ such that $1 \leq i_1 < \cdots < i_k \leq n_1$, there is a $(k+1)$-chain $W_{\bf I}$ in $M_0^{\operatorname{inv}} \times [0,1]$ whose boundary is $\prod_{j=1}^{k} \beta_{i_j} - \prod_{j=1}^k \alpha_{i_j}$. We will describe this chain in terms of the two-chains $Y_i$.

Construct the $k$ chain $W_{\bf I}$ as follows. Let $W_{\bf I}$ be topologically $(S^1)^k \times [0,1]$, and insist that the intersection of $W_{\bf I}$ with $M_0^{\operatorname{inv}} \times \{t\}$ is the product $\prod_{j=1}^k Y_{i_j} \cap (S^2 \backslash {\bf w}) \times \prod_{i \notin {\bf I}}\{x_i\}$. Since the intersection of each $Y_{i_j}$ with $(S^2 \backslash {\bf w}) \times \{t\}$ is a circle for each $t$, this says that the intersection of $W_{\bf I}$ with $M_0^{\operatorname{inv}}\times \{t\}$ is a $k$-torus for all $t$.

Notice that for all $t$, $W_{\bf I} \cap (M_0^{\operatorname{inv}} \times \{t\})$ is a subset of the product $\prod_{i=1}^{n} (Y_i\linebreak \cap ((S^2 \{{\bf w}\})\times [0,1])$.  (Recall here that the line segments $\{x_i\} \times [0,1]$ are subsets of $Y_i$ for each $i$.) Since the $Y_i$ are pairwise disjoint, each $W_{\bf I}\cap(M_0^{\operatorname{inv}}\times \{t\})$ is a submanifold of $M_0^{\operatorname{inv}} = \operatorname{Sym}^{n_1}(S^2 \backslash {\bf w})$ and is disjoint from the fat diagonal.  Furthermore, since $W_{\bf I}$ is homeomorphic to $(S^1)^k \times [0,1]$ and $\partial W_{\bf I}$ is 
$\prod_{j=1}^{k} \beta_{i_j} - \prod_{j=1}^k \alpha_{i_j}$, we see that the collection $\{W_{\bf I}: 1 \leq i_1 < \cdots < i_k \leq n_1\}$ generates $H^k(M_0^{\operatorname{inv}}\times [0,1], X)$.

Let us consider the restriction of $\Upsilon(M_0^{\operatorname{inv}})$ to $W_{\bf I}$ for some particular ${\bf I}$.  We pick preferred trivializations of $N(L_0) \times \{0\}) = N(\mathbb T_{\boldsymbol \beta}) \times \{0\}$ and $J(N(L_1)  \times \{1\}) = J(N(\mathbb T_{\boldsymbol \alpha}) \times \{1\})$ which are products of preferred trivializations of the real subbundles $N(\beta_i) \times \{0\}$ and $J(N(\alpha_i) \times \{1\})$ in $\Upsilon(Y_i)$.

Because $W_i \cap (M_0^{\operatorname{inv}} \times \{t\})$ lies entirely off the fat diagonal for each $t$, the restriction of $\Upsilon(M_0^{\operatorname{inv}})$ to each $W_i$ also decomposes as a product bundle.  In other words, regard $W_i$ as a subset of the abstract product $\prod_{i=1}^{n_1} Y_i$ (which is \textit{not} a submanifold of $M_0^{\operatorname{inv}} \times [0,1]$). Then $\Upsilon(M_0^{\operatorname{inv}})|_{W_{\bf I}}$ is the restriction of the product bundle $(\prod\Upsilon(Y_i))|_{W_{\bf I}}$.  Therefore, since each $\Upsilon(Y_i)$ admits a trivialization with respect to a choice of trivialization of the bundle restricted to $\alpha_i$ and $\beta_i$, we conclude that $\Upsilon(M_0^{\operatorname{inv}})|_{W_{\bf I}}$ is trivializable with respect to preferred trivializations of the bundle over $\mathbb T_{\boldsymbol \beta}\times \{0\}$ and $J(\mathbb T_{\boldsymbol \alpha} \times \{1\})$.

\begin{remark}
We could as easily have shown that $(\operatorname{Sym}^{2n_1}(S^2 \backslash {\bf z}), \mathbb T_{\widetilde{\boldsymbol \beta}}, T_{\widetilde{\boldsymbol \alpha}})$ has a stable normal trivialization. In the case studied here this would produce a spectral sequence from $\widehat{\mathit{HF}}(S^3) \otimes W^{\otimes 2n_1}$ to $\widehat{\mathit{HF}}(S^3) \otimes W^{\otimes n_1}$, which is not especially interesting.
\end{remark}

\section{Examples}\label{Examples Section}

In this section we produce some examples of the behavior of the spectral sequences of Theorems \ref{Link Floer Homology Spectral Sequence} and \ref{Knot Floer Homology Spectral Sequence}.  We present the cases of the unknot and the trefoil as doubly-periodic knots. In general, it is difficult to show that the Seidel-Smith action on $\widetilde{\mathit{HFL}}(\widetilde{\mathcal D})$ agrees with the action induced by the map $\tau_{\#}$ on $\widehat{\mathit{CFL}}(\mathcal D)$ derived from the Heegaard diagram, and therefore difficult to compute the higher differentials in the spectral sequence. However, in these simple examples the higher differentials are determined by the homological grading, so the spectral sequences we compute here are in fact the same as the Seidel-Smith spectral sequence.

To simplify matters, we will compute all spectral sequences using appropriate chain complexes tensored with $\mathbb Z_2((\theta))$, instead of tensoring with $\mathbb Z_2[[\theta]]$ and later further tensoring the $E^{\infty}$ page with $\theta^{-1}$. Then the $E^1$ page of each spectral sequence is a free module with generators the link Floer or knot Floer homology of our doubly periodic knot, and the $E^{\infty}$ page is a free module over $\mathbb Z((\theta))$ with generators the link or knot Floer homology of the quotient.  (The only information we lose by this change is that the $E^{\infty}$ page is not the Borel Heegaard Floer link or knot homology, but rather the Borel homology tensored with $\theta^{-1}$.)

\subsection{The case of the unknot}

In the simplest possible case, let $\widetilde{K}$ be an unknot, and $K$ its quotient knot, a second unknot.  Consider the diagram $\mathcal D$ for $K \cup U$ on the sphere $S^2 = \mathbb C \cup \{\infty\}$ in which $K$ is the unit circle with basepoints $w_0, z_0$ on $U$ and basepoints $w_1, z_1$ on $K$ such that $w_0$ lies at $\infty$, $z_0$ lies at $0$, $w_1$ lies at $i$, and $z_1$ lies at $-i$.  Supply a single $\alpha_1$ and $\beta_1$ as in Figure \ref{Unknot Heegaard Diagrams Figure} coherently with a clockwise orientation of $K$.

\begin{figure} 
\centering
\includegraphics{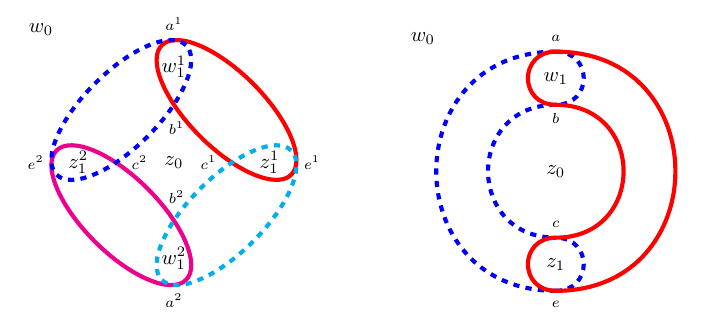}
\caption{An equivariant Heegaard diagram for the unknot together with an axis (i.e. a Hopf link), and its quotient Heegaard diagram (another Hopf link).}
\label{Unknot Heegaard Diagrams Figure}
\end{figure}

   Label the intersection points  $a, b, c, e$ vertically down the diagram as in the figure. (Since we plan to discuss the differentials $d_i$ in the spectral sequence, we will not also use $d$ to label any intersection points.) There are no differentials that count for $\widehat{\mathit{HFL}}(\mathcal D)$ --- which is exactly the link Floer homology of the Hopf link with positive linking number --- and three differentials that count for $\widehat{\mathit{HFK}}(\mathcal D) = \widehat{\mathit{HFK}}(S^3, K) \otimes W$.  See the table of Figure~\ref{Unknot Complexes Figure} for the Alexander gradings of these entries.

\begin{figure}
\centering
\includegraphics{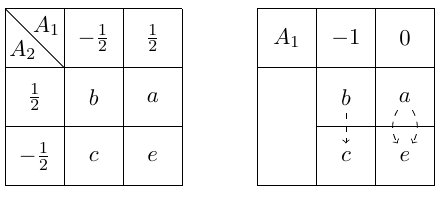}
\caption{Alexander gradings and differentials for $\widehat{\mathit{CFL}}(\mathcal D)$ (on the left) and $\widehat{\mathit{CFK}}(\mathcal D)$ (on the right).}
\label{Unknot Complexes Figure}
\end{figure}

Now lift to a diagram $\widetilde{\mathcal D}$ for the same Hopf link, which has basepoints $w_0, z_0$ on the axis $\widetilde{U}$ and $w_1^1, w_1^2, z_1^1, z_1^2$  on the lifted unknot $\widetilde{K}$.  These basepoints lie on $S^2 = \mathbb C \cup \{\infty\}$ as follows: $w_0$ and $z_0$ lie at $0$ and $\infty$ as previously, and $w_1^1, w_1^2, z_1^1, z_1^2$ lie at $i, -i, 1, -1$ respectively.  There are two curves $\alpha_1^1$ and $\alpha_1^2$ encircling the pairs $w_1^1, z_1^1$ and $w_1^2, z_1^2$, and two curves $\beta_1^1$ and $\beta_1^2$ encircling pairs $z_1^2, w_1^1$ and $z_1^1, w_1^2$.  The intersection points $a, b, c, e$ lift to eight points $a^1,a^2,b^1,b^2,c^1,c^2,e^1,e^2$ on the diagram in $\cup (\alpha_1^i \cap \beta_1^j)$.  (The numbering of each pair is arbitrarily determined by insisting that $a^i$ lie on $\alpha_1^i$, and so on.) The complex $\widehat{\mathit{CFL}}(\widetilde{\mathcal D})$ has eight generators, whose Alexander gradings are laid out in the left table of Figure \ref{Lifted Unknot Complexes Figure}; there are no differentials that count for the theory $\widetilde{\mathit{HFL}}(\widetilde{\mathcal D}) = \widehat{\mathit{HFL}}(S^3, \widetilde{K} \cup \widetilde{U}) \otimes V_1$.  Allowing differentials that pass over the basepoint $z_0$, we obtain the complex of the right table Figure~\ref{Lifted Unknot Complexes Figure} which computes $\widetilde{\mathit{HFK}}(\mathcal D) = \widehat{\mathit{HFK}}(S^3, \widetilde{K}) \otimes V_1 \otimes W$.

\begin{figure}
\centering
\includegraphics{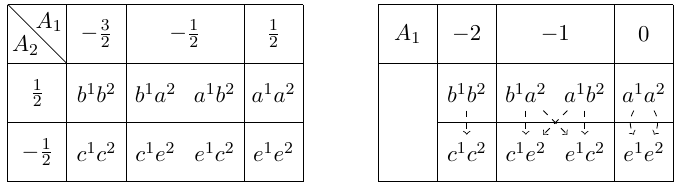}
\caption{Alexander gradings and differentials of $\widehat{\mathit{CFL}}(\widetilde{\mathcal D})$ (left) and $\widehat{\mathit{CFK}}(\widetilde{\mathcal D})$ (right).}
\label{Lifted Unknot Complexes Figure}
\end{figure}

First, consider the spectral sequence for link Floer homology of the double complex $(\widehat{\mathit{CFL}}(\widetilde{\mathcal D} \otimes \mathbb Z_2, \partial + (1 + \tau_{\#})\theta)$.  Since $\partial = 0$ on this complex, the only nontrivial differential occurs on the $E^1$ page and is exactly $d_1 = 1 + \tau_{\#}$.  Therefore the $E^2$ page of the spectral sequence is $\mathbb Z_2((\theta))\langle a^1a^2, b^1b^2, c^1c^2,\linebreak e^1e^2 \rangle$, which is isomorphic to $\widehat{\mathit{HFL}}(S^3, K \cup U) \otimes \mathbb Z_2((\theta))$, as expected.

Next, consider the spectral sequence of the double complex $(\widehat{\mathit{CFK}}(\widetilde{\mathcal D} \otimes \mathbb Z_2, \partial_{\widetilde{U}} + (1 + \tau_{\#})\theta)$.  We have the complex of Figure \ref{Lifted Unknot Complexes Figure}; the $E^1$ page of the spectral sequence is equal to $\mathbb Z_2\langle [a^1a^2], [e^1e^2], [a^1b^2+b^1a^2], [c^1e^2] \rangle \otimes \mathbb Z_2((\theta))$.  Bearing in mind that $[c^1e^2] = [c^2e^1]$, we see that $\tau_*$ is the identity on each of these four elements, and therefore the $E^2$ page of the spectral sequence is the same as the $E^1$ page. Since the differential $d_2$ must raise the Maslov grading by two, the only possible nontrivial differential on the $E^2$ page is $d_2([c^1e^2]\theta^n)$.  Let us compute this differential.  On the chain level, we have $(1+ \tau_{\#})(c^1e^2) = (c^1e^2 + c^2e^1)$.  We observe that $c^1e^2 + c^2e^1 = \partial(a^1b^2)$, so 
\begin{align*}
d_2([c^1e^2]\theta^n) &= [(1+\tau_{\#})(a^1b^2)]\theta^{n+1} \\
				&=[a^1b^2 + a^2b^1]\theta^{n+2}
\end{align*}

Therefore the $E^3$ page of this spectral sequence is exactly $\mathbb Z_2((\theta))\langle[a^1a^2]$, $[e^1e^2]\rangle$, which is isomorphic to $\big(\widehat{\mathit{HFK}}(S^3,K)\otimes W\big)\otimes \mathbb Z_2((\theta))$ as promised, and unchanged thereafter.

\subsection{The case of the trefoil}

Let us now compute some of the spectral sequence for the trefoil as a doubly-periodic knot with quotient the unknot, using the Heegaard diagrams $\mathcal D$ and $\widetilde{\mathcal D}$ of Figure \ref{Trefoil Heegaard Diagrams Figure}.  Recall that $\mathcal D$ is a Heegaard diagram for a link $L$ in $S^3$ consisting of two unknots with linking number $\lambda = 3$ and $\widetilde{\mathcal D}$ is a Heegaard diagram for a link $\widetilde{L}$ in $S^3$ consisting of the left-handed trefoil and the unknotted axis also with linking number $\lambda=3$. We label the twelve intersection points of $\alpha_1$ and $\beta_1$ in $\mathcal D$ and lift to twenty-four intersection points in $\widetilde{\mathcal D}$ in Figure \ref{Trefoil Intersection Points}.  The Alexander gradings and differentials of $\widehat{\mathit{CFL}}(\mathcal D)$ and $\widehat{\mathit{CFK}}(\mathcal D)$ are laid out in the tables of Figure \ref{Trefoil Quotient Complexes}.

\begin{figure}
\centering
\includegraphics{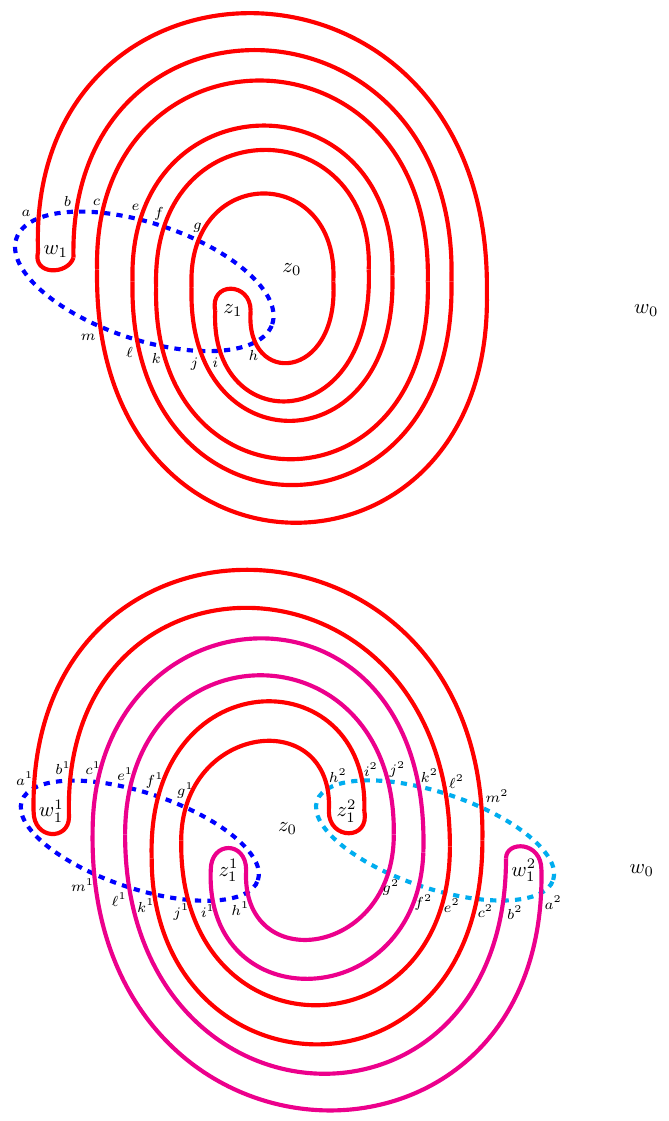}
\caption{Intersection points of $\alpha$ and $\beta$ curves in $\mathcal D$ and $\widetilde{\mathcal D}$.} \label{Trefoil Intersection Points}
\end{figure}

There are no differentials on $\widehat{\mathit{CFL}}(\mathcal D)$, so $\widetilde{\mathit{HFL}}(\mathcal D) = \widehat{\mathit{HFL}}(S^3,L)$ has the twelve generators and gradings of Figure \ref{Trefoil Quotient Complexes}.  From the remainder of that diagram, we observe that the group $\widetilde{\mathit{HFK}}(\mathcal D) = \widehat{\mathit{HFK}}(S^3,U) \otimes W$ is $\mathbb Z_2 \langle [a],[m] \rangle$.

\begin{figure}
\begin{center}
\includegraphics[width=\textwidth]{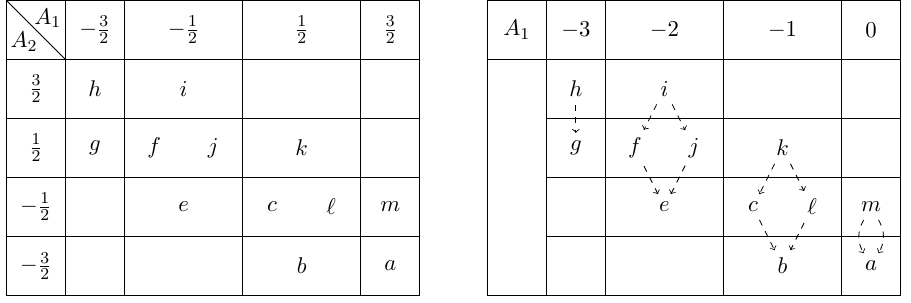}
\caption{Alexander gradings and differentials of $\widehat{\mathit{CFL}}(\mathcal D)$ (left) and $\widehat{\mathit{CFK}}(\mathcal D)$ (right).}
\label{Trefoil Quotient Complexes}
\end{center}
\end{figure}

Now consider the seventy-two generators of $\widehat{\mathit{CFL}}(\widetilde{\mathcal D})$, whose Alexander $A_1$ and $A_2$ gradings are laid out in Figure \ref{CFL(trefoil) Figure}.  It so happens that $\widetilde{\mathcal D}$ is a nice diagram in the sense of Sarkar and Wang \cite{MR2630063}, although the equivariant diagrams for periodic knots introduced in Section \ref{Periodic Knots Section} are not in general, so we may compute $\widetilde{\mathit{HFL}}(\widetilde{\mathcal D})$ with relative ease.  The chain complexes in each Alexander $A_1$ grading are shown in Figures \ref{-7/2 Figure}, \ref{-5/2 Figure}, \ref{-3/2 Figure}, \ref{-1/2 Figure}, \ref{1/2 Figure}, \ref{3/2 Figure} and \ref{5/2 Figure} at the close of this section.  Each of these figures shows the generators in a particular $A_1$ grading of $\widehat{\mathit{CFL}}(\widetilde{\mathcal D})$, which are also the generators of the $A_1 - \frac{3}{2}$ grading of $\widehat{\mathit{CFK}}(\widetilde{\mathcal D})$.  Solid arrows denote differentials that count for the differential $\partial$ and thus exist in both complexes, whereas dashed arrows denote differentials corresponding to disks with nontrivial intersection with the divisor $V_{z_0} = \{z_0\}\times \operatorname{Sym}^{2n_1}(S^2)$.  Therefore dashed differentials only count for the knot Floer complex $(\widehat{\mathit{CFK}}(\widetilde{\mathcal D}), \partial_{\widetilde{U}})$.

\begin{figure}
\centering
\includegraphics[width=\textwidth]{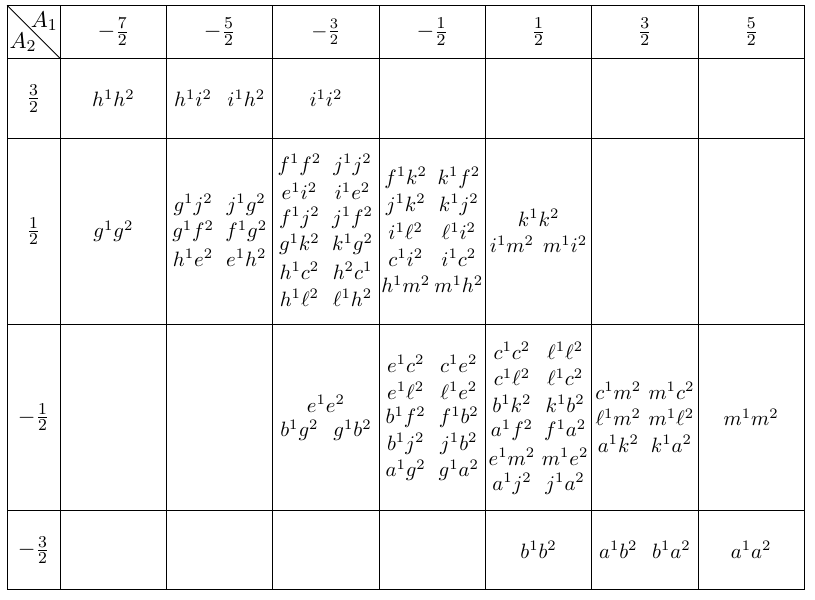}
\caption{Alexander $A_1$ and $A_2$ gradings of the generators of $\widehat{\mathit{CFL}}(\widetilde{D})$.  These generate the $E^0$ page of the link Floer homology spectral sequence.}
\label{CFL(trefoil) Figure}
\end{figure}

The link Floer homology spectral sequence associated to $\widetilde{\mathcal D}$ arises from the double complex $(\widehat{\mathit{CFL}}(\widetilde{\mathcal D}), \partial + (1 + \tau_{\#})\theta)$.
Computing homology of\linebreak $\widehat{\mathit{CFL}}(\widetilde{\mathcal D})$ with respect to the differential $\partial$, we obtain a set of generators for the $E^1$ page of the spectral sequence, which is $\widetilde{\mathit{HFL}}(\widetilde{\mathcal D}) \otimes \mathbb Z_2((\theta)) = (\widehat{\mathit{HFL}}(S^3, 3_1 \cup \widetilde{U}) \otimes V_1) \otimes \mathbb Z_2((\theta))$. These generators and their gradings may be found in Figure \ref{HFL(trefoil) figure}.  Whenever an element of $\widetilde{\mathit{HFK}}(\widetilde{\mathcal D})$ is invariant under the involution $\tau_*$ but has no representative which is invariant under the chain map $\tau_{\#}$, we have included two representatives of that element to make the $\tau_*$ invariance clear.  (For example, observe that $[h^1c^2] = [c^1h^2]$ is invariant under $\tau_*$.)

\begin{figure}
\centering
\includegraphics[width=\textwidth]{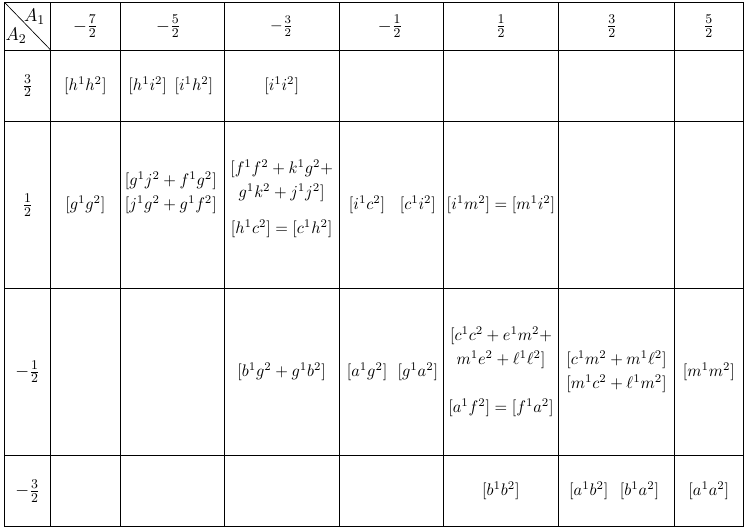}
\caption{The homology $\widetilde{\mathit{HFL}}(\widetilde{\mathcal D})$.  These elements generate the $E^1$ page of the link Floer spectral sequence as a $\mathbb Z_2((\theta))$ module.}
\label{HFL(trefoil) figure}
\vspace{-1.5em}
\end{figure}

The differential $d_1$ on the $E_1$ page of the link Floer spectral sequence for $\widetilde{\mathcal D}$ is $(1 + \tau_*)\theta$; in particular, computing the homology of $d_1$ has the effect of killing all elements of $\widetilde{\mathit{HFL}}(\mathcal D)$ not invariant under $\tau_*$.  Ergo we see that the $E^2$ page of this spectral sequence is generated as a $\mathbb Z_2((\theta))$ module by the elements of Figure \ref{HFL(E2) Figure}.  Notice that the ranks of the $E^2$ page in each Alexander grading $A_1 = 2k + \frac{1}{2}$ correspond precisely to the ranks of each Alexander grading $A_1 = k + \frac{1}{2}$ of Figure \ref{Trefoil Quotient Complexes}, which is the link Floer homology $\widetilde{\mathit{HFL}}(\mathcal D)$.  Therefore the link Floer homology spectral sequence converges on the $E^2$ page.

\begin{figure}
\centering
\includegraphics[width=\textwidth]{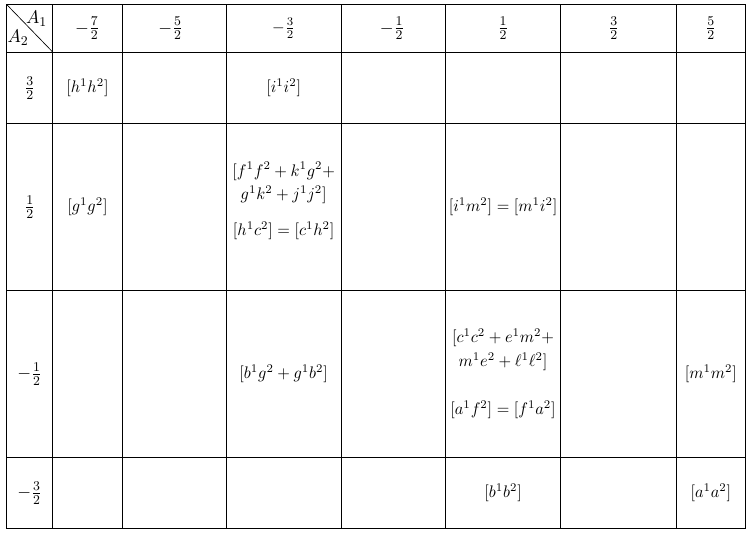}
\caption{Generators for the $E^2 = E^{\infty}$ page of the link Floer spectral sequence for $\widetilde{\mathcal D}$ as a $\mathbb Z_2((\theta))$-module.}
\label{HFL(E2) Figure}
\vspace{-1em}
\end{figure}

\enlargethispage{2em}
We now turn our attention to the knot Floer homology spectral sequence for $\widetilde{\mathcal D}$, which arises from the double complex $(\widehat{\mathit{CFK}}(\mathcal D) \otimes \mathbb Z_2((\theta)), \partial_{\widetilde{U}} + (1 + \tau_{\#})\theta))$. The Alexander $A_1$ gradings of the seventy-two generators in $\widehat{\mathit{CFK}}(\widetilde{\mathcal D})$ are laid out in Figure \ref{CFK(trefoil) Figure}.  Notice that these gradings are exactly the $A_1$ gradings of $\widehat{\mathit{CFL}}(\widetilde{\mathcal D})$ shifted downward by $\frac{\operatorname{lk}(3_1, \widetilde{U})}{2} = \frac{3}{2}$.  These elements generate the $E^0$ page of the link Floer spectral sequence as a $\mathbb Z_2((\theta))$ module.

\begin{figure}
\centering
\includegraphics[width=\textwidth]{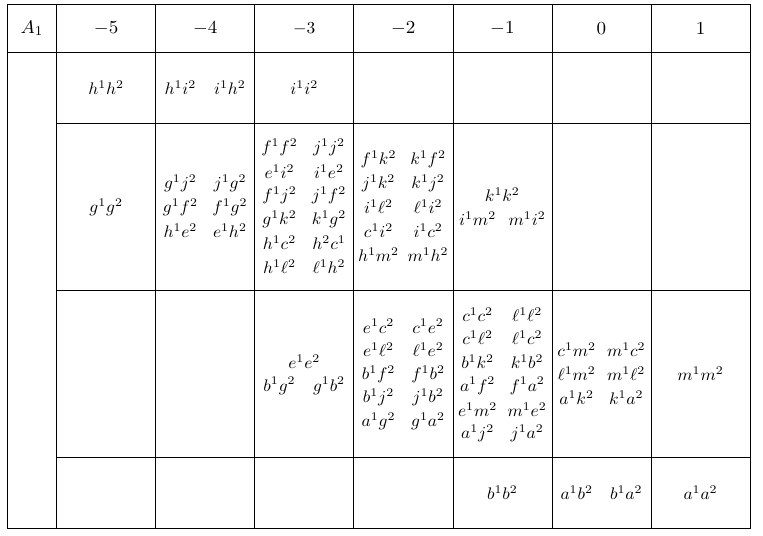}
\caption{Alexander $A_1$ gradings of elements of $\widehat{\mathit{CFK}}(\widetilde{\mathcal D})$. These elements generate the $E^0$ page of the link Floer spectral sequence as a $\mathbb Z_2((\theta))$ module.}
\label{CFK(trefoil) Figure}
\end{figure}

As before, the chain complexes in each Alexander grading may be found in \ref{-7/2 Figure}, \ref{-5/2 Figure}, \ref{-3/2 Figure}, \ref{-1/2 Figure}, \ref{1/2 Figure}, \ref{3/2 Figure} and \ref{5/2 Figure}.  Computing the homology of these complexes, we obtain $\widetilde{\mathit{HFK}}(\widetilde{\mathcal D}) = \widehat{\mathit{HFK}}(3_1)\otimes V_1 \otimes W$, whose generators are described in Figure \ref{HFK(trefoil) Figure}.  These elements generate the $E^1$ page of the knot Floer homology spectral sequence of $\widetilde{\mathcal D}$ as a $\mathbb Z_2((q))$-module.  Once again, homology classes which are equivalent under the induced involution $\tau_*$ but have no representative which is invariant under the chain map $\tau_{\#}$ have been included with two equivalent descriptions to emphasize their invariance.

The differential $d_1 = (1 + \tau_*)\theta$ on the $E^1$ page of the spectral sequence has the effect of eliminating all elements of $\widetilde{\mathit{HFK}}(\widetilde{\mathcal D})$ which are not invariant under the action of $\tau_*$.  Computing homology with respect to $d_1$ yields the set of generators of Figure \ref{E2 trefoil knot figure}.

The knot Floer spectral sequence stabilizes on the $E^2$ page in every Alexander grading except $A_1 = -2$.  There is a single nontrivial $d_2$ differential which behaves similarly to the nontrivial $d_2$ differential we earlier saw in the case of the unknot as a doubly-periodic knot.  In particular: consider $d_2([a^1g^2])$.  We compute this differential as follows: first, apply $1+\tau_{\#}$ to $a^1g^2$, obtaining $(a^1g^2 + g^1a^2)\theta$.  Next, we choose an element whose boundary under $\partial_{\widetilde{U}}$ is $(a^1g^2 + g^1a^2)\theta$; one such is $(h^1m^2)\theta$.  Then $d_2([a^1g^2]) = [1+\tau_{\#}(h^1m^2)]\theta) = [h^1m^2 + m^1h^2]\theta^2$. Moreover, we then have $d_2([a^1g^2]\theta^n)=$\linebreak $[h^1m^2 + m^1h^2]\theta^{n+2}$ in general. Therefore the Alexander grading $-2$ vanishes on the $E^3$ page of the knot Floer spectral sequence, and the $E^3 = E^{\infty}$ page is isomorphic to $\widetilde{\mathit{HFK}}(\mathcal D)$ after an appropriate shift and rescaling in Alexander $A_1$ gradings.  Generators for this page of the knot Floer spectral sequence appear in Figure \ref{E3(trefoil) Figure}.

Notice that for the left-handed trefoil considered as a two-periodic knot, Edmonds' condition is sharp: $g(3_1) = 1 = 2(0) + \frac{3 - 1}{2} = 2g(U) + \frac{\lambda -1}{2}$.  We also see here the realization of Corollary \ref{Fiberedness Corollary}: the left-handed trefoil is fibred, and the highest nontrivial Alexander grading in $\widetilde{\mathit{HFK}}(\widetilde{\mathcal D})$, namely $A_1 = 1$ has rank two.  The two generators in this Alexander grading, $[a_1a_2]$ and $[m_1m_2]$, are preserved over the course of the spectral sequence and become the two generators of $\widetilde{\mathit{HFK}}(\mathcal D)$ in grading $A_1 = 0$ under the isomorphism between the $E^{\infty}$ page of the knot Floer spectral sequence and $\widetilde{\mathit{HFK}}(\mathcal D)$.  Then the highest nontrivial Alexander grading of $\widetilde{\mathit{HFK}}(\mathcal D)$, $A_1 = 0$, also has rank two, corresponding to the unknot being a fibred knot.

\newpage
\begin{figure}[t]
\centering
\includegraphics[width=.86\textwidth]{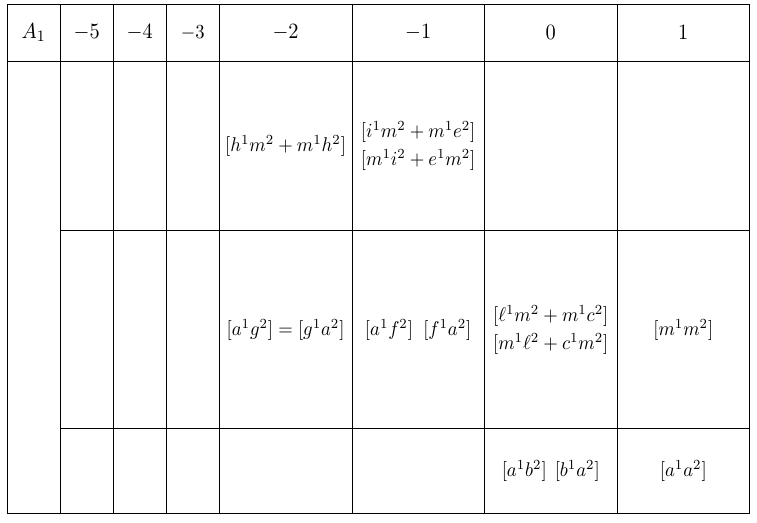}
\vspace{-1ex}
\caption{The homology $\widetilde{\mathit{HFK}}(\widetilde{\mathcal D})$, which is  $\widehat{\mathit{HFK}}(3_1)\otimes V_1 \otimes W$.  These elements generate the $E^1$ page of the knot Floer homology spectral sequence of $\widetilde{\mathcal D}$ as a $\mathbb Z_2((q))$-module.}
\label{HFK(trefoil) Figure}
\end{figure}

\begin{figure}[!h]
\includegraphics[scale=.9]{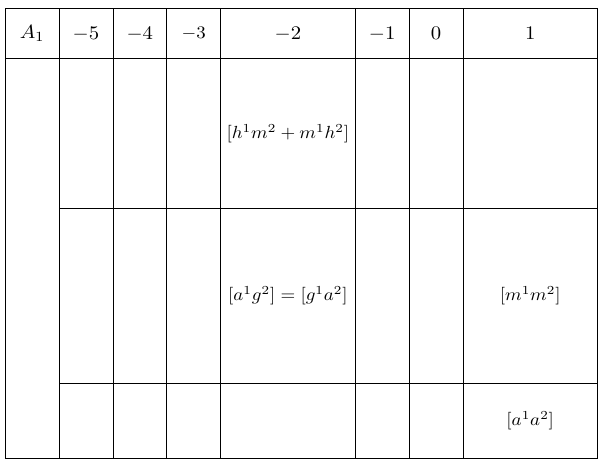}
\vspace{-1ex}
\caption{Generators for the $E^2$ page of the knot Floer spectral sequence associated to $\widetilde{\mathcal D}$ as a $\mathbb Z_2((q))$ module.}
\label{E2 trefoil knot figure}
\end{figure}

\begin{figure}
\centering
\includegraphics{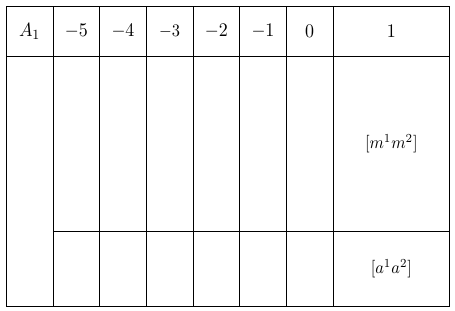}
\caption{Generators for the $E^3 = E^{\infty}$ page of the knot Floer spectral sequence associated to $\mathcal{\widetilde{D}}$.}
\label{E3(trefoil) Figure}
\end{figure}

\vspace{.7ex}

\begin{figure}
\centering
\includegraphics{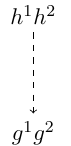}
\caption{The chain complex $\widehat{\mathit{CFL}}(\widetilde{\mathcal D})$ in Alexander grading $A_1 = -\frac{7}{2}$, and the chain complex $\widehat{\mathit{CFK}}(\widetilde{\mathcal D})$ in Alexander grading $A_1 = -5$.  Dashed arrows denote differentials appearing only in the latter.}
\label{-7/2 Figure}
\end{figure}

\vspace{.7ex}

\begin{figure}
\centering
\includegraphics{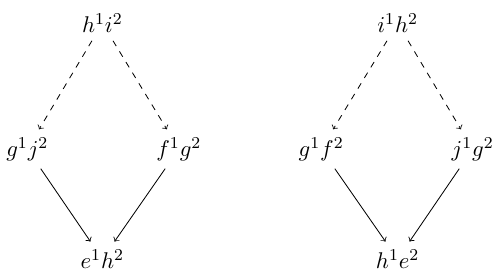}
\caption{The chain complex $\widehat{\mathit{CFL}}(\widetilde{\mathcal D})$ in Alexander grading $A_1 = -\frac{5}{2}$, and the chain complex $\widehat{\mathit{CFK}}(\widetilde{\mathcal D})$ in Alexander grading $A_1 = -4$.  Dashed arrows denote differentials appearing only in the latter.}
\label{-5/2 Figure}
\end{figure}

\newpage
\begin{figure}
\vspace{8em}
\centering
\includegraphics{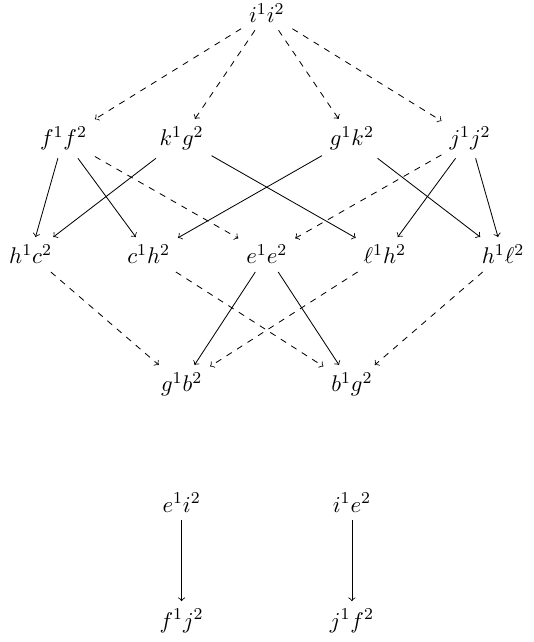}
\caption{The chain complex $\widehat{\mathit{CFL}}(\widetilde{\mathcal D})$ in Alexander grading $A_1 = -\frac{3}{2}$, and the chain complex $\widehat{\mathit{CFK}}(\widetilde{\mathcal D})$ in Alexander grading $A_1 = -3$.  Dashed arrows denote differentials appearing only in the latter.}
\label{-3/2 Figure}
\end{figure}
\clearpage

\begin{figure}
\vspace{3em}
\centering
\includegraphics{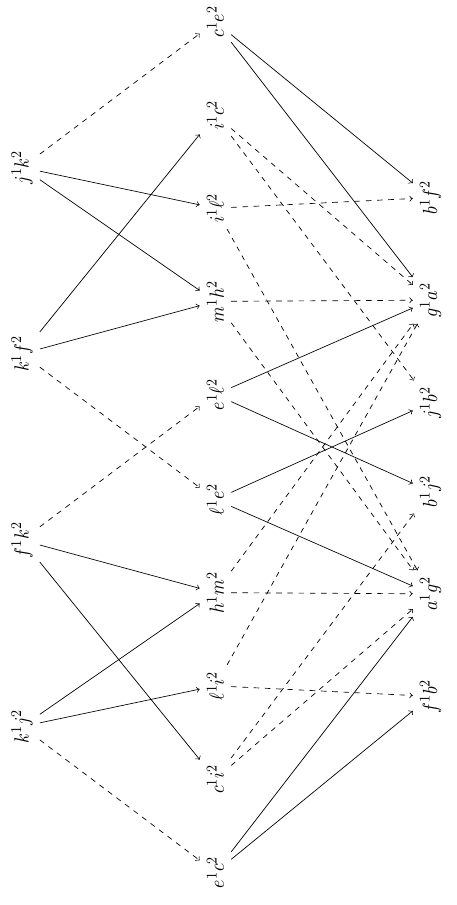}
\caption{The chain complex $\widehat{\mathit{CFL}}(\widetilde{\mathcal D})$ in Alexander grading $A_1 = -\frac{1}{2}$, and the chain complex $\widehat{\mathit{CFK}}(\widetilde{\mathcal D})$ in Alexander grading $A_1 = -2$.  Dashed arrows denote differentials appearing only in the latter.}
\label{-1/2 Figure}
\end{figure}
\clearpage

\begin{figure}
\centering
\includegraphics{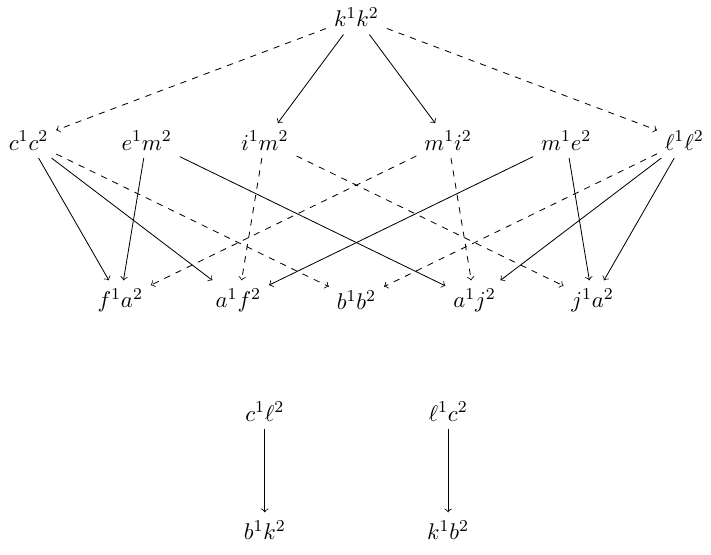}
\caption{The chain complex $\widehat{\mathit{CFL}}(\widetilde{\mathcal D})$ in Alexander grading $A_1 = \frac{1}{2}$, and the chain complex $\widehat{\mathit{CFK}}(\widetilde{\mathcal D})$ in Alexander grading $A_1 = -1$.  Dashed arrows denote differentials appearing only in the latter.}
\label{1/2 Figure}
\end{figure}

\begin{figure}
\centering
\includegraphics{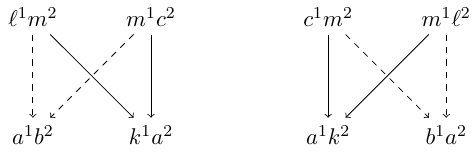}
\caption{The chain complex $\widehat{\mathit{CFL}}(\widetilde{\mathcal D})$ in Alexander grading $A_1 = \frac{3}{2}$, and the chain complex $\widehat{\mathit{CFK}}(\widetilde{\mathcal D})$ in Alexander grading $A_1 = 0$.  Dashed arrows denote differentials appearing only in the latter.}
\label{3/2 Figure}
\end{figure}

\begin{figure}
\centering
\includegraphics{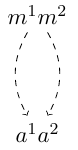}
\caption{The chain complex $\widehat{\mathit{CFL}}(\widetilde{\mathcal D})$ in Alexander grading $A_1 = \frac{5}{2}$, and the chain complex $\widehat{\mathit{CFK}}(\widetilde{\mathcal D})$ in Alexander grading $A_1 = 1$.  Dashed arrows denote differentials appearing only in the latter.}
\label{5/2 Figure}
\end{figure}

\bibliography{2-342}

\clearpage
\address{Department of Mathematics, University Of California Los Angeles\\
 Los Angeles, CA 90095-1555, USA\\
\email{hendricks@math.ucla.edu}\\
\received{October 15, 2012}}

\end{document}